\documentclass{amsart}
\usepackage{amsfonts}
\usepackage{amssymb}
\usepackage{graphicx}
\usepackage{amscd}
\usepackage{hyperref}
\usepackage{boxedminipage}

\newtheorem{theorem}{Theorem}[section]

\newtheorem{lemma}[theorem]{Lemma}
\newtheorem{proposition}[theorem]{Proposition}
\newtheorem{axiom}[theorem]{Condition}
\newtheorem{remark}[theorem]{Remark}

\newtheorem{definition}[theorem]{Definition}

\allowdisplaybreaks

\begin{document}
\title[Infinitely Degenerate Quasilinear Equations ]{Hypoellipticity for
Infinitely Degenerate Quasilinear Equations and the Dirichlet Problem }
\author[Rios]{Cristian Rios}
\address{University of Calgary\\
Calgary, Alberta\\
crios@ucalgary.ca}
\author[Sawyer]{Eric T. Sawyer}
\address{McMaster University\\
Hamilton, Ontario\\
sawyer@mcmaster.ca}
\author[Wheeden]{Richard L. Wheeden}
\address{Rutgers University\\
New Brunswick, N.J.\\
wheeden@math.rutgers.edu}
\date{\today }
\keywords{quasilinear equations, degenerate elliptic equations, a priori
estimates, Dirichlet problem, existence and uniqueness, regularity}
\subjclass[2000]{35J70, 35A05, 35H10, 35J60, 35B45, 35B30, 35B65, 35B50,
35D05, 35D10}

\begin{abstract}
{In \cite{RSaW1}, we considered a class of infinitely degenerate quasilinear
equations of the form 
\begin{equation*}
\mathop{\rm div}\mathcal{\mathcal{A}}\left( x,w\right) \nabla w+\vec{\gamma}%
\left( x,w\right) \cdot \nabla w+f\left( x,w\right) =0
\end{equation*}
and derived \textit{a priori} bounds for high order derivatives }$D^{\alpha
}w $ of their {solutions in terms of }$w$ and $\nabla w$. We now show that
it is possible to obtain bounds just in terms of $w$ for a further {subclass
of such equations, and we apply the resulting estimates to prove that
continuous weak solutions are necessarily smooth. We also obtain existence,
uniqueness and interior $\mathcal{C}^{\infty }$-regularity of solutions for
the corresponding Dirichlet problem with continuous boundary data.}
\end{abstract}

\maketitle

\section{{Introduction\label{firstsection}}}

{It is well-known (see \cite{GiTr}) that if }$A${\ is \emph{elliptic}, and }$%
A${\ and }$b${\ are smooth functions of their arguments, then quasilinear
operators in divergence form 
\begin{equation*}
\mathcal{Q}w=\mathop{\rm div}A\left( x,w,\nabla w\right) +b\left( x,w,\nabla
w\right)
\end{equation*}%
are hypoelliptic: any weak solution $w$ of $\mathcal{Q}w=0$ is smooth. When $%
\mathcal{Q}$ is \emph{subelliptic} - i.e., when ellipticity fails only to
finite order - then hypoellipticity still holds if $\mathcal{Q}$ is \emph{%
linear} (see e.g. \cite{Treves}). When $\mathcal{Q}$ is linear but fails to
be subelliptic, the situation is more delicate. For example, Fedi\u{\i}
showed in \cite{Fe} that the two-dimensional operator 
\begin{equation}
\partial _{x}^{2}+k\left( x\right) \partial _{y}^{2}  \label{fedi}
\end{equation}%
is hypoelliptic if $k$ is smooth and positive for all $x\neq 0$. In this
case $k$ is allowed to vanish at any rate at $x=0$. However, by \cite{KuStr}%
, $\partial _{x}^{2}+k\left( x\right) \partial _{y}^{2}+\partial _{z}^{2}$
is hypoelliptic in $\mathbb{R}^{3}$ if and only if $\lim_{x\rightarrow
0}x\log k\left( x\right) =0$. A quasilinear version of operators of the form
(\ref{fedi}) arises when one considers two-dimensional Monge-Amp\`{e}re
equations 
\begin{equation}
u_{ss}u_{tt}-u_{st}^{2}=k\left( s,t\right) ,\qquad \left( s,t\right) \in 
\widetilde{\Omega }\subset \mathbb{R}^{2},  \label{MA2D}
\end{equation}%
together with the classical partial Legendre transformation $\left(
x,y\right) =T\left( s,t\right) $ given by 
\begin{equation}
\left\{ 
\begin{array}{lll}
x & = & s \\ 
y & = & u_{t}%
\end{array}%
\right. .  \label{PLT2D}
\end{equation}%
Indeed, assuming that $T$ is invertible, (\ref{MA2D}) and (\ref{PLT2D}) lead
to the two-dimensional quasilinear equation 
\begin{equation}
\partial _{x}^{2}w+\partial _{y}\left\{ k\left( x,w\left( x,y\right) \right)
\partial _{y}w\right\} =0,\;\;\;\;\;\left( x,y\right) \in \Omega =T(%
\widetilde{\Omega }),  \label{qfedi}
\end{equation}%
satisfied in the weak sense by $w\left( x,y\right) =t$. In \cite{SaW} and 
\cite{SaW2}, two of the authors extended Fedi\u{\i}'s two-dimensional
regularity result for linear equations to certain solutions $w$ of (\ref%
{qfedi}) obtained through the transformation (\ref{PLT2D}) from a solution
of the Monge-Amp\`{e}re equation (\ref{MA2D}). The coefficient $k$
considered there is assumed to satisfy 
\begin{equation}
\left\vert k_{t}\left( s,t\right) \right\vert \leq C\,k\left( s,t\right) ^{%
\frac{3}{2}},\qquad \left( s,t\right) \in \widetilde{\Omega }.  \label{thw}
\end{equation}%
Thus $k$ is required to become more independent of its second variable as $k$
degenerates. Notice that the coefficient $k\,$ in (\ref{fedi}) is
independent of the second variable, and then (\ref{thw}) is automatically
true. The main result in \cite{SaW2} establishes that degenerate
two-dimensional Monge-Amp\`{e}re equations (\ref{MA2D}) with smooth
right-hand side $k$ satisfying (\ref{thw}) are hypoelliptic. This was the
first known hypoellipticity result for infinitely degenerate Monge-Amp\`{e}%
re equations. More general equations than (\ref{qfedi}) are also treated in
the papers above, including the equation for prescribed Gaussian curvature. }

\section{Description of the results}

In the present work, we improve the two-dimensional results above by
lowering the exponent $\frac{3}{2}$ in (\ref{thw}) to the optimal exponent $%
1 $; this optimallity is shown in Section 3 below. We also extend the theory
of regularity for degenerate quasilinear equations of the type treated in 
\cite{SaW} and \cite{SaW2} to any dimension $n\geq 2$ and to more general
quasilinear problems. In this process we have to deal with several
fundamental difficulties associated with higher dimensions and the more
general structure of the equations. {We consider quasilinear equations of
the divergence form 
\begin{equation}
\mathcal{Q}w=\mathop{\rm div}\mathcal{A}\left( x,w\right) \nabla w+\vec{%
\gamma}\left( x,w\right) \cdot \nabla w+f\left( x,w\right) =0\qquad \text{in 
}\Omega ,  \label{equation}
\end{equation}%
where $\Omega $ is an open bounded connected subset of $\mathbb{R}^{n}$ and
the matrix $\mathcal{A}$, vector function }${\vec{\gamma}}$ {and scalar
function $f$ are smooth functions of their arguments. }Here we adopt the
vector notation $\vec{u}=\left( u^{1},u^{2},\dots ,u^{n}\right) $; $\nabla w$
denotes the gradient $\nabla w=\left( \partial _{1}w,\partial _{2}w,\dots
,\partial _{n}w\right) ^{\prime }$ where $\partial _{i}=\frac{\partial }{%
\partial x_{i}}$ is the $i^{th}$ partial derivative; and the divergence
operator applied to a vector function $\vec{u}$ is given by $\mathop{\rm div}%
\vec{u}=\partial _{1}u^{1}+\partial _{2}u^{2}+\dots +\partial _{n}u^{n}$. {\
In our applications, (\ref{equation}) will sometimes be satisfied in the
strong sense, i.e. in the pointwise sense for }$\mathcal{C}^{2}$ functions $%
w $, and other times in the {weak sense; see Section \ref%
{Subsubsection-weak-solution} in the Appendix for a precise definition. Note
that in the case $n=2$, equation (\ref{qfedi}) is included among equations
of the form $\mathcal{Q}w\left( x_{1},x_{2}\right) =0$ by choosing }${\vec{%
\gamma}}${$=\vec{0}$, $f=0$ and $\mathcal{A}\left( x_{1},x_{2},z\right) $ to
be the diagonal matrix $\mathop{\rm diag}\left( 1,k\left(
x_{1},x_{2},z\right) \right) $ with $k$ independent of $x_{2}$. However,
equation (\ref{equation}) does not include systems obtained from the
Monge-Amp\`{e}re equation by the partial Legendre transform introduced in 
\cite{RSaW} for higher dimensions, and the treatment of such systems when
ellipticity fails to infinite order remains a challenging open problem in
dimensions bigger than two.}

{In Section \ref{hypo}, under structural restrictions on $\mathcal{A}$ and }$%
{\vec{\gamma}}${\ which are similar to (\ref{thw}) (although weaker, see
Conditions \ref{hyp1} and \ref{hyp2}), we first obtain local a priori bounds
for the Lipschitz norm of smooth solutions in terms of their $L^{\infty }$
norm and the parameters inherent to the equation. This result together with
the main theorem in \cite{RSaW1} (see Theorem \ref{quasi} below) provides a
priori control of \emph{all} derivatives of a smooth solution in terms of
the supremum norm of the solution. }

{In Section \ref{section-Proof}, we apply the a priori estimates together
with an approximation scheme to prove that continuous weak solutions are
smooth, and to establish existence, uniqueness, and regularity of solutions
of the Dirichlet problem. To do so, we use a class of }custom-built {\
barrier functions, a maximum principle and a comparison principle adapted to
our class of equations (see Sections \ref{barriersect}, \ref{maxpsection}
and \ref{compsect} in the Appendix).}

The method used to construct the barriers in Lemma \ref{barrier} takes into
account only the modulus of continuity of solutions on the boundary. This
generalizes most known barrier constructions, which usually require higher
regularity of solutions on the boundary.

{As already mentioned, one of our main results, Theorem \ref{application}
below, states that under certain hypotheses on the coefficients $\mathcal{A}$
, every \emph{continuous} weak solution $w$ of (\ref{equation}) is
infinitely differentiable, and all of its derivatives are locally controlled
by $\left\Vert w\right\Vert _{L^{\infty }}$. The conditions imposed on the
coefficients allow them to vanish to infinite order, so the quasilinear
operator $\mathcal{Q}$ in (\ref{equation}) is not in general uniformly
elliptic or even subelliptic. This is the first known hypoellipticity result
for infinitely degenerate quasilinear equations in $n$ dimensions.}

We now state special cases of the main Theorems \ref{DP} and \ref%
{application}.\ We include these simpler versions to illustrate the
principal features of our results without the technical assumptions of the
general case. {A \emph{domain} will always mean an open connected set. }

\begin{theorem}[Dirichlet problem]
Let $\Omega $ {be a strongly convex domain in $\mathbb{R}^{n}$ containing
the origin. }Let $k^{i}\left( x,z\right) $, $i=2,\dots ,n,$ be smooth
nonnegative functions in $\Omega \times \mathbb{R}$ such that 
\begin{equation*}
k^{i}\left( x,z\right) >0\qquad \text{if}\quad x_{j}\neq 0\quad \text{for
some }j\neq i
\end{equation*}%
(this means that $k^{i}\left( x,z\right) $ may vanish only for those $\left(
x,z\right) $ so that $x$ lies on the $i^{th}$-coordinate axis), and such
that for some $B>0,$ 
\begin{equation*}
\left\vert \frac{\partial }{\partial z}k^{i}\left( x,z\right) \right\vert
\leq B~k^{\ast }\left( x,z\right) \qquad \text{for all\quad }\left(
x,z\right) \in \Omega \times \mathbb{R},
\end{equation*}%
where $k^{\ast }=\min_{i=2,\dots ,n}k^{i}$. Then, {for any continuous
function }$\varphi $ on $\partial \Omega $, there exists a unique continuous
strong solution $w$ to the Dirichlet problem%
\begin{equation*}
\left\{ 
\begin{array}{rcll}
\frac{\partial ^{2}}{\partial x_{1}^{2}}w\left( x\right) +\sum_{i=2}^{n}%
\frac{\partial }{\partial x_{i}}k^{i}\left( x,w\left( x\right) \right) \frac{%
\partial }{\partial x_{i}}w\left( x\right) & = & 0\quad & \text{in }\Omega ,
\\ 
w & = & \varphi & \text{on }\partial \Omega ,%
\end{array}%
\right.
\end{equation*}%
i.e., there exists a unique $w$ that is both a strong solution of the
differential equation in $\Omega $ and continuous in $\overline{\Omega }$
with boundary values ${\varphi }$. Moreover, this solution $w\in \mathcal{C}%
^{0}\left( \overline{\Omega }\right) \bigcap \mathcal{C}^{\infty }\left(
\Omega \right) $.
\end{theorem}

\begin{theorem}[Regularity of solutions]
{\ Let $\Omega \subset \mathbb{R}^{n}$ be a bounded domain containing the
origin, and suppose that $k^{i}\left( x,z\right) $, }$i=2,\dots ,n,${\ are
as in Theorem 2.1. Then any }continuous weak solution{\ $w$ of}%
\begin{equation*}
\frac{\partial ^{2}}{\partial x_{1}^{2}}w\left( x\right) +\sum_{i=2}^{n}%
\frac{\partial }{\partial x_{i}}k^{i}\left( x,w\left( x\right) \right) \frac{%
\partial }{\partial x_{i}}w\left( x\right) =0{\quad }\text{in }\Omega {\ }
\end{equation*}%
is also a strong solution, and {satisfies $w\in \mathcal{C}^{\infty }\left(
\Omega \right) $. }
\end{theorem}

\subsection{A priori estimates\label{statements}}

For $\tilde{x}\in \mathbb{R}^{n}$ and $\vec{r}\in \mathbb{R}_{+}^{n},$ we
denote by $\mathcal{R}\left( \tilde{x},\vec{r}\right) $ the box centered at $%
\tilde{x}$ with edges parallel to the coordinate axes and half-edgelengths
given by $\vec{r}$, i.e., 
\begin{equation}
\mathcal{R}\left( \tilde{x},\vec{r}\right) =\left[ \tilde{x}_{1}-r^{1},%
\tilde{x}_{1}+r^{1}\right] \times \cdots \times \left[ \tilde{x}_{n}-r^{n},%
\tilde{x}_{n}+r^{n}\right] .  \label{rectangulo}
\end{equation}%
When $\tilde{x}$ is the origin we will just write $\mathcal{R}\left( \vec{r}%
\right) $ for $\mathcal{R}\left( 0,\vec{r}\right) $ and we also adopt the
summation notation $\mathcal{R}\left( \tilde{x},\vec{r}\right) =\tilde{x}+%
\mathcal{R}\left( \vec{r}\right) $. When $\vec{r}$ and $\tilde{x}$ are fixed
or clear from context, we will omit them and simply write $\mathcal{R}$ for $%
\mathcal{R}\left( \tilde{x},\vec{r}\right) $. For any positive constant $%
\gamma $, $\gamma \mathcal{R}\left( \tilde{x},\vec{r}\right) $ will denote
the box centered at $\tilde{x}$ with half-edgelengths given by $\gamma \vec{r%
}$, i.e.,%
\begin{equation}
\gamma \mathcal{R}\left( \tilde{x},\vec{r}\right) =\mathcal{R}\left( \tilde{x%
},\gamma \vec{r}\right) =\tilde{x}+\mathcal{R}\left( \gamma \vec{r}\right) .
\label{rectanguloG}
\end{equation}%
We define the \emph{i-wrap} $\mathcal{T}_{i}\left( \mathcal{R}\right) $ of
the box $\mathcal{R}$ as the set of faces of $\partial \mathcal{R}$
containing the direction $\vec{e}_{i}=\left( \delta _{ij}\right) _{j=1,\dots
,n}$: 
\begin{equation*}
\mathcal{T}_{i}\left( \mathcal{R}\right) =\overline{\partial \mathcal{R}%
\backslash \left\{ y:\left\vert y_{i}-\tilde{x}_{i}\right\vert
=r^{i}\right\} }.
\end{equation*}%
The following figure illustrates $i$-wraps in $\mathbb{R}^{3}$ :

\begin{center}
\includegraphics[width=5.7in]{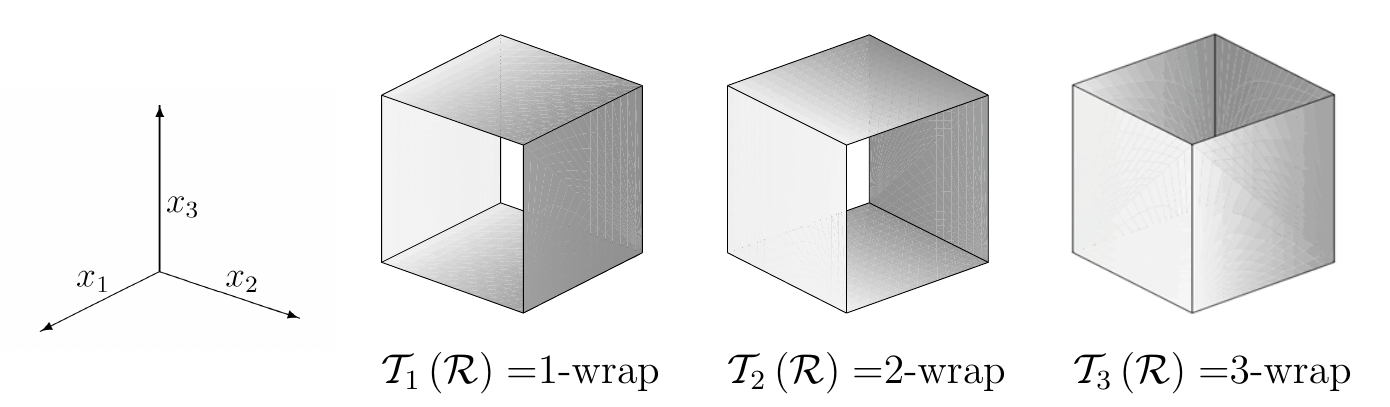}
\end{center}

We will use the notation $\Gamma =\Omega \times \mathbb{R}$ unless specified
to the contrary.

\begin{axiom}[Nondegeneracy condition]
\label{hyp1}We say that $\vec{k}=\left( k^{1},\dots ,k^{n}\right) $
satisfies Condition \ref{hyp1} in $\Gamma$ if $\vec{k}$ has continuous
second order derivatives in $\Gamma $, $\vec{k}\geq 0 $ in $\Gamma $, and $%
\vec{k}$ satisfies the following:

\begin{enumerate}
\item \label{ax1c2}For every subdomain $\Omega ^{\prime }$ with $\overline{%
\Omega ^{\prime }}\subset \Omega $, there exists $c>0$ such that if $\Gamma
^{\prime }=\Omega ^{\prime }\times \mathbb{R}$, then 
\begin{equation*}
\inf_{\left( x,z\right) \in \Gamma ^{\prime }}\max_{1\leq i\leq
n}k^{i}(x,z)\geq c>0.
\end{equation*}

\item \label{ax1c3}For all $x\in \Omega $ and $0<\varepsilon <n^{-\frac{1}{ 2%
}}\mathop{\rm dist}\left( x,\partial \Omega \right) $, there exists $\vec{r}
=\vec{r}\left( x,\varepsilon ,\vec{k}\right) =\left( r^{1},\dots
,r^{n}\right) $, $0<r^{1},\dots ,r^{n}<\varepsilon $, and a box $\mathcal{R}%
= \mathcal{R}\left( \tilde{x},\vec{r}\right)$ so that $x\in \frac{1}{3} 
\mathcal{R}$ and 
\begin{equation*}
k^{i}\left( y,z\right) >0\text{ whenever }y\in \mathcal{T}_{i}\left( 
\mathcal{R}\right) ,~z\in \mathbb{R},~i=1,\dots ,n.
\end{equation*}
The restriction on the size of $\epsilon$ ensures that $\mathcal{R} \subset
\Omega$.
\end{enumerate}
\end{axiom}

Note that if $k^{1}\equiv 1$, then property (\ref{ax1c2}) in Condition \ref%
{hyp1} is trivially satisfied. If each $k^{i}$ is nonnegative and has only 
\emph{isolated zeros}, then property (\ref{ax1c3}) in Condition \ref{hyp1}
holds.

\begin{remark}
\label{spans}Property (\ref{ax1c3}) in Condition \ref{hyp1} holds for $\vec{%
k }$ in $\Gamma =\Omega \times \mathbb{R}$ if and only if for every $x\in
\Omega $ and $\varepsilon >0,$ there exists a box $\mathcal{R}$ in $\Omega $
with $x\in \frac{1}{3}\mathcal{R}$ and $0<r^{1},\dots ,r^{n}<\varepsilon $
such that for all $y\in \partial \mathcal{R}$ and $z\in \mathbb{R}$, the set
of vectors 
\begin{equation*}
\mathcal{S}\left( y,z\right) =\left\{ k^{1}\left( y,z\right) \vec{e}
_{1},k^{2}\left( y,z\right) \vec{e}_{2},\dots ,k^{n}\left( y,z\right) \vec{e}
_{n}\right\} ,\quad \vec{e}_{i}=\left( \delta _{ij}\right) _{j=1,\dots ,n},
\end{equation*}
spans the tangent space to $\partial \mathcal{R}$ at $y$.
\end{remark}

{A structural condition that we impose on the matrix $\mathcal{A}$ in (\ref%
{equation}) is that }it is equivalent to a diagonal matrix in the following
sense:

\begin{axiom}[Diagonal condition]
\label{hellipp} For $\vec{k}$ as in Condition \ref{hyp1}, we assume that for
some $\Lambda \geq 1$, the matrix $\mathcal{A}$ satisfies 
\begin{equation}
\sum_{i=1}^{n}k^{i}\left( x,z\right) \xi _{i}^{2}\leq \xi ^{t}\mathcal{A}
\left( x,z\right) \xi \leq \Lambda \sum_{i=1}^{n}k^{i}\left( x,z\right) \xi
_{i}^{2}  \label{hellip}
\end{equation}
for all $\xi \in \mathbb{R}^{n}$ and $\left( x,z\right) \in \Gamma =\Omega
\times \mathbb{R}$.
\end{axiom}

\begin{remark}
Because of Condition \ref{hellipp} or Remark \ref{spans}, we can state
property (\ref{ax1c3}) in the nondegeneracy Condition \ref{hyp1} in terms of
the matrix $\mathcal{A}$ as follows: For every $x\in \Omega $ and $%
\varepsilon $ with $0<\varepsilon <n^{-\frac{1}{2}}\mathop{\rm dist}\left(
x,\partial \Omega \right) $, there exist $\vec{r}=\left( r^{1},\dots
,r^{n}\right) $ with $0<r^{1},\dots ,r^{n}<\varepsilon $ and a box $\mathcal{%
\ R}=\mathcal{R}\left( \tilde{x},\vec{r}\right) $ such that $x\in \frac{1}{3}
\mathcal{R}$ and for every $y\in \partial \mathcal{R}$ and nonzero $\vec{v}
\left( y\right) $ \emph{tangent} to $\partial \mathcal{R}$ at $y$, 
\begin{equation}
\vec{v}\left( y\right) \cdot \mathcal{A}\left( y,z\right) \vec{v}\left(
y\right) >0\qquad \text{for all }z\in \mathbb{R}\text{.}  \label{lidAk}
\end{equation}
\end{remark}

{Recall that a domain }$\Omega ${\ is an open connected set. $\Omega
^{\prime }$ will always denote a domain with compact closure in $\Omega$;
this will be abbreviated }$\Omega ^{\prime }\Subset \Omega ${. }

\begin{definition}[Subunit type]
{\label{subunitt}}We say that a vector field $G=\sum_{i=1}^{n}\gamma
^{i}\left( x,z\right) \frac{\partial }{\partial x_{i}}$ with bounded
coefficients $\gamma ^{i}$ is of \emph{subunit type} with respect to $%
\mathcal{A}$ {in $\Gamma =\Omega \times \mathbb{R}$ if for every $\Omega
^{\prime }\Subset \Omega $ and $M_{0}\geq 1,$ there is a constant $B_{\gamma
}=B_{\gamma }\left( \Omega ^{\prime },M_{0}\right) >0$ such that 
\begin{equation*}
\left( \sum_{i=1}^{n}\gamma ^{i}\left( x,z\right) \xi _{i}\right) ^{2}\leq
B_{\gamma }^{2}~\xi ^{t}\mathcal{A}\left( x,z\right) \xi \quad \text{for all 
}\left( x,z\right) \in \Gamma _{M_{0}}^{\prime }=\Omega ^{\prime }\times %
\left[ -M_{0},M_{0}\right], \xi \in \mathbb{R}^n.
\end{equation*}
\ }
\end{definition}

We will impose further conditions on $\mathcal{A}$. To motivate them, we
recall the classical inequality of Wirtinger type on a domain $\Phi \subset 
\mathbb{R}^{n+1}:$ if $k$ is nonnegative with bounded second derivatives on $%
\Phi$, then 
\begin{equation}
\left\vert \nabla _{\!\!x,z}k\left( x,z\right) \right\vert \leq C\left\{
\left\Vert \nabla _{\!\!x,z}^{2}k\right\Vert _{L^{\infty }\left( \Phi
\right) }^{\frac{1}{2}}+\left( \mathop{\rm dist}\left( \left( x,z\right)
,\partial \Phi \right) \right) ^{-\frac{1}{2}}\right\} \sqrt{k\left(
x,z\right) }  \label{genWirt}
\end{equation}
for all $\left( x,z\right) \in \Phi $ (see e.g. the appendix in \cite{SaW}).

{Inequality (\ref{genWirt}) is crucial in our calculations, and although it
has an analogue for nonnegative diagonal matrices, it does not extend to
general matrix functions.}

\begin{definition}[Subordinate matrix]
\label{submat}{We say that $\mathcal{A}$ is \emph{subordinate} in $\Gamma
=\Omega \times \mathbb{R}$ if for every $\Omega ^{\prime }\Subset \Omega $
and $M_{0}\geq 1$, there exists $B_{\mathcal{A}}=B_{\mathcal{A}}\left(
\Omega ^{\prime },M_{0}\right) >0$ such that 
\begin{equation}
\sum_{i=1}^{n}\left\vert \partial_i\mathcal{A}\left( x,z\right) \xi
\right\vert ^{2}+\left\vert \partial_z\mathcal{A}\left( x,z\right) \xi
\right\vert ^{2}\leq B_{\mathcal{A}}^{2}~\xi ^{t}\mathcal{A}\left(
x,z\right) \xi  \label{wirtm}
\end{equation}
}for all $\xi \in \mathbb{R}^{n}$, $\left( x,z\right) \in \Gamma
_{M_{0}}^{\prime }$, {where $\Gamma_{M_0}^{\prime }=\Omega ^{\prime }\times %
\left[ -M_{0},M_{0}\right] $. }
\end{definition}

{We always consider locally bounded solutions $w$, i.e. $\left\Vert
w\right\Vert _{L^{\infty }\left( \Omega ^{\prime }\right) }<\infty $ for all
subdomains $\Omega ^{\prime }\Subset \Omega $. Thus, we deal only with a
solution $w$ whose graph on $\Omega ^{\prime }$ is contained in a \emph{%
bounded} connected set 
\begin{equation}
\Gamma _{M_{0}}^{\prime }=\Omega ^{\prime }\times \left[ -M_{0},M_{0}\right]
\Subset \Gamma =\Omega \times \mathbb{R}  \label{domain}
\end{equation}%
for some $M_{0}=M_{0}\left( w,\Omega ^{\prime }\right) <\infty $. For
convenience we also assume $M_{0}\geq 1$.\ }

{To obtain our main results, we will use the following a priori estimates
obtained in \cite{RSaW1} for the class of equations (\ref{equation} ). \ }

\begin{theorem}[Theorem 1.8 in~%
\cite{RSaW1}%
]
{\ \label{quasi}Let $\Omega $ be a bounded domain in $\mathbb{R}^{n}$ and $%
\Gamma =\Omega \times \mathbb{R}$. Let $\vec{k}\left( x,z\right) \in 
\mathcal{C}^{2}\left( \Gamma \right) $ and $\mathcal{A}\left( x,z\right) $, $%
f\left( x,z\right) $ and $\vec{\gamma}\left( x,z\right) \in \mathcal{C}
^{\infty }\left( \Gamma \right) $. Suppose that}

\begin{enumerate}
\item {$\mathcal{A\,}$satisfies (\ref{hellip}), where $\vec{k}\left(
x,z\right) $ satisfies the nondegeneracy }Condition{\ \ref{hyp1} in $\Gamma $
,}

\item {$\mathcal{A\,}$is subordinate in $\Gamma $ (Definition \ref{submat}),}

\item {\ $\vec{\gamma}$ is of subunit type with respect to }$\mathcal{A}${\
in $\Gamma $ (Definition \ref{subunitt}). }
\end{enumerate}

{Then for every smooth solution $w$ of (\ref{equation}) in $\Omega $,
integer $N\geq 0$ and subdomains $\Omega ^{\prime \prime }\Subset \Omega
^{\prime }\Subset \Omega ,$ there exists a constant }$\mathcal{C}
_{\left\Vert w\right\Vert _{L^{\infty }\left( \Omega ^{\prime }\right)
},\left\Vert \nabla w\right\Vert _{L^{\infty }\left( \Omega ^{\prime
}\right) },N}=${\ 
\begin{equation*}
\mathcal{C}_{\left\Vert w\right\Vert _{L^{\infty }\left( \Omega ^{\prime
}\right) },\left\Vert \nabla w\right\Vert _{L^{\infty }\left( \Omega
^{\prime }\right) },N}\left( n,B,\vec{k},\Lambda ,\left\Vert \mathcal{A}
\right\Vert _{\mathcal{C}^{N+2}\left( \tilde{\Gamma}\right) },\left\Vert
f\right\Vert _{\mathcal{C}^{N+1}\left( \tilde{\Gamma}\right) },\left\Vert 
\vec{\gamma}\right\Vert _{\mathcal{C}^{N+1}\left( \tilde{\Gamma}\right)
},\Omega ,\Omega ^{\prime },\Omega ^{\prime \prime }\right)
\end{equation*}
such that 
\begin{equation}
\sum_{\left\vert \vec{\alpha}\right\vert \leq N}\left\Vert D^{\alpha
}w\right\Vert _{L^{\infty }\left( \Omega ^{\prime \prime }\right) }\leq 
\mathcal{C}_{\left\Vert w\right\Vert _{L^{\infty }\left( \Omega ^{\prime
}\right) },\left\Vert \nabla w\right\Vert _{L^{\infty }\left( \Omega
^{\prime }\right) },N}.  \label{control}
\end{equation}
Here $\tilde{\Gamma}=\Omega ^{\prime }\times \left[ -2\left\Vert
w\right\Vert _{L^{\infty }\left( \Omega ^{\prime }\right) },2\left\Vert
w\right\Vert _{L^{\infty }\left( \Omega ^{\prime }\right) }\right] $, and $B$
denotes $B_\gamma$, $B_{\mathcal{A}}$. }
\end{theorem}

{Our main application of these a priori estimates is the hypoellipticity
result stated in Theorem \ref{application} for (infinitely degenerate)
quasilinear equations of the form (\ref{equation}). In it, as in the special
two-dimensional case contained in \cite{SaW}, we will assume extra
conditions on the coefficients, namely, that the \emph{nonlinear} and the 
\emph{\ infinitely degenerate} characters do not occur simultaneously in the
sense described below. We denote}%
\begin{equation}
{k^{\ast }\left( x,z\right) =\min_{i=1,\dots ,n}\left\{ k^{i}\left(
x,z\right) \right\} .}  \label{kstar}
\end{equation}

\begin{axiom}[Super Subordination condition]
{\label{hyp2}We say that $\mathcal{A}$ satisfies the super subordination
condition in $\Gamma =\Omega \times \mathbb{R}$} if {for every $\Omega
^{\prime }\Subset \Omega $ and $M_{0}\geq 1,$ there exist constants $B_{ 
\mathcal{A}}=B_{\mathcal{A}}\left( \Omega ^{\prime },M_{0}\right) $ and }$%
B_{ \mathcal{A}}^{\prime }=B_{\mathcal{A}}^{\prime }\left( \Omega ^{\prime
},M_{0}\right) ${\ such that if }$\left( x,z\right) \in \Gamma
_{M_{0}}^{\prime }$ and $~\xi \in \mathbb{R}^{n},$ then{\ 
\begin{eqnarray}
\left\vert \partial _{z}\mathcal{A}\left( x,z\right) \xi \right\vert ^{2}
&\leq &B_{\mathcal{A}}^{2}~k^{\ast }\left( x,z\right) \,\xi ^{t}\mathcal{A}
\left( x,z\right) \xi ,  \label{xtra} \\
\sum_{i=1}^{n}\left\vert \partial _{i}\partial _{z}\mathcal{A}\left(
x,z\right) \xi \right\vert ^{2}+\left\vert \partial _{z}^{2}\mathcal{A}
\left( x,z\right) \xi \right\vert ^{2} &\leq &\left( B_{\mathcal{A}}^{\prime
}\right) ^{2}\,\xi ^{t}\mathcal{A}\left( x,z\right) \xi .  \label{xxtra}
\end{eqnarray}
\indent%
\noindent If $\mathcal{A}$ is diagonal, then condition (\ref{xxtra}) follows
from (\ref{xtra}) by Wirtinger's inequality: see Remark \ref{diagcase}. }

We say $\vec{\gamma}\left( x,z\right) $ satisfies the super subordination
condition if {for all }$\left( x,z\right) \in \Gamma _{M_{0}}^{\prime }$ and 
$\xi \in \mathbb{R}^{n},$ {\ 
\begin{equation}
\left\vert \partial _{z}\vec{\gamma}\left( x,z\right) \cdot \xi \right\vert
^{2}\leq B_{\gamma }^{2}~k^{\ast }\left( x,z\right) \xi ^{t}\mathcal{A}%
\left( x,z\right) \xi ,  \label{gxtra}
\end{equation}%
for some $B_{\gamma }=B_{\gamma }\left( \Omega ^{\prime },M_{0}\right) $.}
\end{axiom}

{The extra vanishing }condition{\ (\ref{xtra}) on $\partial _{z}\mathcal{A}$
is a stronger form of the part of (\ref{wirtm}) involving $\partial _{z} 
\mathcal{A}$. In the two-dimensional diagonal case $\mathcal{A}= 
\mathop{\rm
diag}\left( 1,k\right) $, inequality (\ref{wirtm}) always holds for any $%
\mathcal{C}^{2}$ nonnegative $k\left( x,z\right) $, and it takes the form (%
\ref{genWirt}), while the more restrictive (\ref{xtra}) with $k^{\ast }=k$
takes the form 
\begin{equation}
\left\vert \partial _{z}k\left( x,z\right) \right\vert \leq B~k\left(
x,z\right)  \label{thw2}
\end{equation}
(compare (\ref{thw})), which does not hold in general for nonnegative $%
k\left( x,z\right) $. On the other hand, if $f\left( x\right) $ is any
smooth nonnegative function in $\mathbb{R}^{n}$ and $h\left( x,z\right) $ is
a nonnegative Lipschitz function in $\mathbb{R}^{n+1}$, then 
\begin{equation*}
k\left( x,z\right) =f\left( x\right) \left[ 1+h\left( x,z\right) \right]
\end{equation*}
satisfies (\ref{thw2}). Indeed, we have the following lemma; see Section 6.4
in the appendix of \cite{SaW} for details. }

\begin{lemma}[\protect\cite{SaW}]
{\ \label{chaz}Let $k(x,z)$ be a smooth nonnegative function in a bounded
region $T \subset \mathbb{R}^{n}\times \mathbb{R}$, and assume that for some 
$\gamma, B \geq 1$, 
\begin{equation}
\left\vert \partial _{z}k(x,z)\right\vert \leq B\, k\left( x,z\right)
^{\gamma }.  \label{xwirt}
\end{equation}%
Then, for every $\left( x_{0},z_{0}\right) \in T $, there exists a smooth
function $f(x)\geq 0$ and a Lipschitz function $h(x,z)$, with Lipschitz
constant depending only on $B$, $\left\Vert k\right\Vert _{L^{\infty}(T)}$
and $T$, such that 
\begin{equation}
k\left( x,z\right) =f\left( x\right) \left( 1+f\left( x\right) ^{\gamma
-1}h\left( x,z\right) \right)  \label{charat}
\end{equation}%
for all $\left( x,z\right) $ in a neighborhood of $\left( x_{0},z_{0}\right) 
$. Moreover, $h\left( x_{0},z_{0}\right) =0$. In particular, 
\begin{equation*}
C^{-1}k\left( x,z^{\prime }\right) \leq k\left( x,z\right) \leq Ck\left(
x,z^{\prime }\right) ,\qquad \left( x,z\right) ,\,\left( x,z^{\prime
}\right) \in T ,
\end{equation*}
where $C=C\left( B,\mathop{\rm diam}T \right) $. Conversely, if $h\left(
x,z\right) $ is smooth and $f\left( x\right) \,$is a nonnegative smooth
function such that $f\left( x\right) ^{\gamma }$ is smooth, then $k\left(
x,z\right) $ given by (\ref{charat}) is smooth and satisfies (\ref{xwirt})
for some $B=B\left( h,f,\mathop{\rm
  diam}T \right) $.}
\end{lemma}

\begin{remark}
\label{diagcase} As noted earlier, if $\mathcal{A}$ is diagonal then the
second extra vanishing condition (\ref{xxtra}) follows from the first one (%
\ref{xtra}) by (\ref{genWirt}). Indeed, it is enough to prove this for a
scalar function $k\left( x,z\right) $. If (\ref{thw2}) holds, then 
\begin{equation*}
\tilde{k}\left( x,z\right) : =\partial _{z}k\left( x,z\right) +B~k\left(
x,z\right) \geq 0,
\end{equation*}
so by (\ref{genWirt}), {for all }$\left( x,z\right) \in \Gamma
_{M_{0}}^{\prime }$, 
\begin{equation*}
\left\vert \nabla _{\!\!x,z}\tilde{k}\left( x,z\right) \right\vert \leq
C\left\{ \left\Vert \nabla _{\!\!x,z}^{2}\tilde{k}\right\Vert _{L^{\infty
}\left( \Gamma _{2M_{0}}^{\prime \prime }\right) }^{\frac{1}{2}}+\left( %
\mathop{\rm dist}\left( \left( x,z\right) ,\partial \Gamma _{2M_{0}}^{\prime
\prime }\right) \right) ^{-\frac{1}{2}}\right\} \sqrt{\tilde{k}\left(
x,z\right), }
\end{equation*}
where $\Gamma _{2M_{0}}^{\prime \prime }=\Omega ^{\prime \prime }\times %
\left[ -2M_{0},2M_{0}\right] $ with $\Omega ^{\prime }\Subset \Omega
^{\prime \prime }\Subset \Omega $. Hence from (\ref{thw2}) we obtain 
\begin{equation*}
\sum_{i=1}^{n}\left\vert \partial _{i}\partial _{z}k\right\vert
^{2}+\left\vert \partial _{z}^{2}k\right\vert ^{2}\leq \tilde{B}%
^{2}\,k\left( x,z\right)
\end{equation*}
for a suitable constant $\tilde{B}$.
\end{remark}

{Under the extra assumptions in }the super subordination Condition{\ \ref%
{hyp2}, we will obtain interior a priori control of all derivatives
(including first order ones) of smooth solutions in terms of the supremum
norm of the solutions: }

\begin{theorem}[A priori estimate]
{\ \label{apriorial}Let $\Omega \subset \mathbb{R}^{n}$ be a bounded domain.
Suppose that $\Gamma $, $\mathcal{A}\left( x,z\right) $, $f\left( x,z\right) 
$, $\vec{\gamma}\left( x,z\right) $ and $\vec{k}\left( x,z\right) $ are as
in Theorem \ref{quasi}, and also that $\mathcal{A}$ and $\vec{\gamma}$
satisfy }the super subordination Condition{\ \ref{hyp2} in $\Gamma $. If $w$
is a smooth solution of (\ref{equation}) in $\Omega $, then for any
nonnegative integer $N$ and subdomain $\Omega ^{\prime }\Subset \Omega $,
there exists a constant }$\mathcal{C}_{\left\Vert w\right\Vert _{L^{\infty
}\left( \Omega ^{\prime }\right) },N}=${\ 
\begin{equation*}
\mathcal{C}_{\left\Vert w\right\Vert _{L^{\infty }\left( \Omega ^{\prime
}\right) },N}\left( n,B,\vec{k},\Lambda ,\left\Vert \mathcal{A}\right\Vert _{%
\mathcal{C}^{N+3}\left( \tilde{\Gamma}\right) },\left\Vert f\right\Vert _{%
\mathcal{C}^{N+3}\left( \tilde{\Gamma}\right) },\left\Vert \vec{\gamma}%
\right\Vert _{\mathcal{C}^{N+3}\left( \tilde{\Gamma}\right) },\Omega ,\Omega
^{\prime }\right)
\end{equation*}%
such that 
\begin{equation*}
\sum_{\left\vert \vec{\alpha}\right\vert \leq N}\left\Vert D^{\vec{\alpha}%
}w\right\Vert _{L^{\infty }\left( \Omega ^{\prime }\right) }\leq \mathcal{C}%
_{\left\Vert w\right\Vert _{L^{\infty }\left( \Omega ^{\prime }\right) },N}.
\end{equation*}%
Here $\tilde{\Gamma}=\Omega ^{\prime }\times \left[ -2\left\Vert
w\right\Vert _{L^{\infty }\left( \Omega ^{\prime }\right) },2\left\Vert
w\right\Vert _{L^{\infty }\left( \Omega ^{\prime }\right) }\right] $ and $B$
denotes $B_{\gamma },B_{\mathcal{A}},B_{\mathcal{A}^{\prime }}$. }
\end{theorem}

\begin{remark}
{\ A special case of Theorem \ref{apriorial} is established in Theorem 2.4
in \cite{SaW}, namely when $n=2$, $\vec{k}\left( x_{1},x_{2},z\right)
=\left( 1,k\left( x_{1},z\right) \right) $ is independent of $x_{2}$ with 
\begin{equation}
\left\vert \partial _{z}k\left( x_{1},z\right) \right\vert \leq Ck\left(
x_{1},z\right) ^{\frac{3}{2}},  \label{thww}
\end{equation}
and $f=0$, $\vec{\gamma}=0$. Also, much more is required there of the
solution $w$, namely $w$ must satisfy (see (2.22) in \cite{SaW}) 
\begin{equation*}
\begin{array}{rrl}
1+\left( \partial _{x_{1}}w\left( x_{1},x_{2}\right) \right) ^{2} & \leq & 
C~\partial _{x_{2}}w\left( x_{1},x_{2}\right) , \\ 
k\left( x_{1},w\left( x_{1},x_{2}\right) \right) ~\partial _{x_{2}}w\left(
x_{1},x_{2}\right) & \leq & C.%
\end{array}%
\end{equation*}
These restrictions are removed in Theorem \ref{apriorial}, in which we also
generalize the result to higher dimensions and allow lower order terms. }%
\newline
\indent%
{An important improvement found in Theorem \ref{apriorial} relative to
Theorem 2.4 in \cite{SaW} is the reduction of the power $3/2$ in (\ref{thww}%
) to the sharp power $1$. See Section \ref{section-sharp}.}
\end{remark}

The next lemma shows that if only the first part of Condition \ref{hyp2}
holds, then the bilinear form induced by $\mathcal{A}\left( x,z\right) $ is
equivalent to one which is independent of $z$ in any set on which $z$ is
bounded.

\begin{lemma}
\label{admisall}{Let $\mathcal{A}=\left( a_{ij}\left( x,z\right) \right)
_{i,j=1,\dots ,n}$ be a smooth symmetric matrix satisfying (\ref{hellip}) in 
$\Gamma $ and such that for every $\Omega ^{\prime }\Subset \Omega $ and $%
M_{0}\geq 1,$ there exists $\tilde{B}=\tilde{B}_{\mathcal{A}}\left( M_{0}, %
\mathop{\rm diam}\Omega ,\mathop{\rm dist}\left( \Omega ^{\prime },\partial
\Omega \right) \right) $ such that 
\begin{equation}
\left\vert \partial _{z}\mathcal{A}\left( x,z\right) \xi \right\vert
^{2}\leq \tilde{B}^{2}k^{\ast }\left( x,z\right) \,\xi ^{t}\mathcal{A}\left(
x,z\right) \xi ,\qquad \left( x,z\right) \in \Gamma _{M_{0}}^{\prime },~\xi
\in \mathbb{R}^{n}.  \label{extratwo}
\end{equation}
Then there exists $\mathcal{C}=\mathcal{C}\left( \tilde{B},M_{0}\right) $
such that for all }$\left( x,z\right) ,\left( x,\tilde{z}\right) \in \Gamma
_{M_{0}}^{\prime }$ and $\,\xi \in \mathbb{R}^{n},$ 
\begin{equation*}
\mathcal{C}^{-1}\,\xi ^{t}\mathcal{A}\left( x,z\right) \xi \leq \xi ^{t}%
\mathcal{A}\left( x,\tilde{z}\right) \xi \leq \mathcal{C}\,\xi ^{t}\mathcal{%
A }\left( x,z\right) \xi .
\end{equation*}
Moreover, for all $i=1,\cdots ,n$ and $\left( x,z\right) ,\left( x,\tilde{z}
\right) \in \Gamma _{M_{0}}^{\prime },$ 
\begin{equation}
\mathcal{\tilde{C}}^{-1}k^{i}\left( x,z\right) \leq k^{i}\left( x,\tilde{z}
\right) \leq \mathcal{\tilde{C}}\,k^{i}\left( x,z\right),  \label{chaak}
\end{equation}
where $\mathcal{\tilde{C}}=\mathcal{\tilde{C}}\left( \mathcal{C},\Lambda
\right) $.
\end{lemma}

\proof%
{By (\ref{extratwo}), for all $\xi \in \mathbb{R}^{n}$, the function $%
h\left( x,z,\xi \right) =\xi ^{t}\mathcal{A}\left( x,z\right) \,\xi $
satisfies 
\begin{eqnarray*}
\left\vert \partial _{z}h\left( x,z,\xi \right) \right\vert ^{2}
&=&\left\vert \xi ^{t}\mathcal{A}_{z}\left( x,z\right) \,\xi \right\vert ^{2}
\\
&\leq &\left\vert \xi \right\vert ^{2}\left\vert \mathcal{A}_{z}\left(
x,z\right) \,\xi \right\vert ^{2} \\
&\leq &\left\vert \xi \right\vert ^{2}\tilde{B}^{2}k^{\ast }\left(
x,z\right) \,\xi ^{t}\mathcal{A}\left( x,z\right) \xi \\
&=&\left\vert \xi \right\vert ^{2}\tilde{B}^{2}k^{\ast }\left( x,z\right)
h\left( x,z,\xi \right) ,
\end{eqnarray*}
for all $\left( x,z\right) \in \Gamma _{M_{0}}^{\prime }=\Omega ^{\prime
}\times \left[ -M_{0},M_{0}\right] $ and$~\xi \in \mathbb{R}^{n}$. Since
from (\ref{hellip}) we have 
\begin{equation*}
k^{\ast }\left( x,z\right) \left\vert \xi \right\vert
^{2}=\sum_{i=1}^{n}k^{\ast }\left( x,z\right) \xi _{i}^{2}\leq
\sum_{i=1}^{n}k^{i}\left( x,z\right) \xi _{i}^{2}\leq \xi ^{t}\mathcal{A}
\left( x,z\right) \xi =h\left( x,z,\xi \right) ,
\end{equation*}
we obtain 
\begin{equation*}
\left\vert \partial _{z}h\left( x,z,\xi \right) \right\vert \leq \tilde{B}
\,h\left( x,z,\xi \right) .
\end{equation*}
By Gronwall's inequality it follows that for some $C=C\left( \tilde{B}
,M_{0}\right) ,$ 
\begin{equation*}
C^{-1}\,\xi ^{t}\mathcal{A}\left( x,z\right) \xi \leq \xi ^{t}\mathcal{A}
\left( x,\tilde{z}\right) \xi \leq C\,\xi ^{t}\mathcal{A}\left( x,z\right)
\xi ,\qquad \left( x,z\right) ,\left( x,\tilde{z}\right) \in \Gamma
_{M_{0}}^{\prime },\,\xi \in \mathbb{R}^{n}.
\end{equation*}
Inequality (\ref{chaak}) then follows immediately from (\ref{hellip}). } 
\endproof%

\subsection{Hypoellipticity Main Results}

{Our main results Theorems \ref{DP} and \ref{application} are obtained as
applications of the a priori estimates above. Theorem \ref{application}
establishes smoothness of weak solutions, and Theorem \ref{DP} deals with
existence, uniqueness, and regularity of strong solutions to the Dirichlet
problem. The concept of weak solutions of our infinitely degenerate
operators is similar to that of classical weak solutions for accretive
operators, defined via the associated Hilbert space. We denote by }$H_{%
\mathcal{X}}^{1,2}\left( \Omega \right) ${\ the Hilbert space on }$\Omega ${%
\ induced by the quadratic form }$\mathcal{X}\left( x,\xi \right)
=\sum_{i=1}^{n}k^{i}\left( x,0\right) \xi _{i}^{2}$ (see Appendix, Section %
\ref{Section-Weak-Sol}). {We say that }$w$ is a weak solution of (\ref%
{equation}) in $\Omega $ if $w\in H_{\mathcal{X}}^{1,2}\left( \Omega \right)
\bigcap L^{\infty }\left( \Omega \right) $ and {\ } 
\begin{equation*}
\int \mathcal{A}\left( x,w\right) \nabla w\cdot \nabla \varphi ~dx-\int
\varphi \vec{\gamma}\left( x,w\right) \cdot \nabla w~dx-\int f\left(
x,w\right) \varphi ~dx=0
\end{equation*}%
for all $\varphi \in Lip_{0}(\Omega )$. See Section \ref{deg-sob-section} of
the Appendix for a detailed discussion of the degenerate Sobolev spaces $H_{%
\mathcal{X}}^{1,2}\left( \Omega \right) $ and the meaning of $\nabla w$ if $%
w\in H_{\mathcal{X}}^{1,2}\left( \Omega \right) $.

Some extra conditions are required in order to make our approximation scheme
work.

\begin{definition}[Strongly convex domain]
{\ \label{strcv} A convex domain $\Phi \subset \mathbb{R}^{n}$ with $%
\partial \Phi \in \mathcal{C}^{2}$ is called \emph{\ strongly convex} (with
convex character $\lambda_0 $) if there exists $\lambda_0 >0$ such that 
\begin{equation*}
\inf_{p\in \partial \Omega }\min_{1\leq i\leq n-1}\lambda ^{i}\left(
p\right) =\lambda_0 >0,
\end{equation*}
where $\lambda ^{1}\left( p\right) ,\dots ,\lambda ^{n-1}\left( p\right) $
denote the principal curvatures (see \cite{GiTr} p.354) of $\partial \Phi $
at a point $p\in \partial \Phi $. }
\end{definition}

{Theorem \ref{application} is obtained from the following existence and
uniqueness theorem for the Dirichlet problem.}

\begin{theorem}[Dirichlet problem]
\bigskip \label{DP}{Let $\widetilde{\Omega }\subset \mathbb{R}^{n}$ be a
bounded open set and let }$\Omega \Subset \widetilde{\Omega }${\ be a
strongly convex domain. For }$\Gamma =\widetilde{\Omega }\times \mathbb{R},$
suppose{\ $\vec{k}\in \mathcal{C}^{2}\left( \Gamma \right) $ and $\mathcal{A}%
\left( x,z\right) $, $f\left( x,z\right) $, $\vec{\gamma}\left( x,z\right)
\in \mathcal{C}^{\infty }\left( \Gamma \right) $ are such that }

\begin{enumerate}
\item {\ \label{appc01}$\vec{k}\left( x,z\right) $ satisfies }the
nondegeneracy Condition{\ \ref{hyp1} in $\Gamma $, }

\item {\ \label{appc02}$\mathcal{A\,}$satisfies the diagonal }Condition{\ %
\ref{hellipp} in $\Gamma $, }

\item {\ \label{appc07}$f\left( x,z\right) \mathop{\rm sign}z\leq 0\text{
and }f_{z}\left( x,z\right) \leq 0\text{ in }\Gamma $, }

\item {\ \label{appc04}$\vec{\gamma}$ }is of subunit type with respect to $%
\mathcal{A}${\ in $\Gamma $, }

\item {\ }\label{appc06}$\vec{\gamma}$ has compact support in $\Omega $ in
the $x$ variable, locally in the $z$ variable, i.e., for all $M_{0}>0$ there
exists open $\Omega ^{\prime }\Subset \Omega $ such that $\vec{\gamma}
\left( x,z\right) =0$ if $\left( x,z\right) \in \left( \Omega \backslash 
\overline{ \Omega ^{\prime }}\right) \times \left[ -M_{0},M_{0}\right] $,

\item {\ }\label{appc03}{$\mathcal{A}$ satisfies }the super subordination
Condition{\ \ref{hyp2} in $\Gamma $}$,${\ }

\item {\ \label{appc05}$\vec{\gamma}$ satisfies }the super subordination
Condition{\ \ref{hyp2} in $\Gamma $}$.${\ }
\end{enumerate}
\end{theorem}

\noindent 
\it%
{Then given any continuous function }$\varphi $ on $\partial \Omega$, there
exists a strong solution {$w$ to the Dirichlet problem} 
\begin{equation}
\left\{ 
\begin{array}{rcll}
\mathop{\rm div}\mathcal{A}\left( x,w\right) \nabla w+\vec{\gamma}\left(
x,w\right) \cdot \nabla w+f\left( x,w\right) & = & 0\quad & \text{in }\Omega
\\ 
w & = & \varphi & \text{on }\partial \Omega ,%
\end{array}%
\right.  \label{dirchlet}
\end{equation}%
{i.e., there exists} $w$ which is continuous in $\overline{\Omega }$, equal
to $\varphi $ on $\partial \Omega $, and a strong solution of the
differential equation in $\Omega $. {Moreover, $w\in \mathcal{C}^{0}\left( 
\overline{\Omega }\right) $}$\bigcap \mathcal{C}^{\infty }\left( \Omega
\right) $, and {for any nonnegative integer $N$ and subdomain $\Omega
^{\prime }\Subset $}${\Omega }${, there exists a constant $\mathcal{C}_{N}=$ 
\begin{equation*}
\mathcal{C}_{N}\left( n,B,\vec{k},\Lambda ,\left\Vert \varphi \right\Vert
_{L^{\infty }\left( \partial \Omega \right) },\left\Vert \mathcal{A}%
\right\Vert _{\mathcal{C}^{N+2}\left( \tilde{\Gamma}\right) },\left\Vert
f\right\Vert _{\mathcal{C}^{N+1}\left( \tilde{\Gamma}\right) },\left\Vert 
\vec{\gamma}\right\Vert _{\mathcal{C}^{N+1}\left( \tilde{\Gamma}\right)
},\lambda _{0},\Omega ,\Omega ^{\prime }\right)
\end{equation*}%
such that }$\sum_{\left\vert \vec{\alpha}\right\vert \leq N}\left\Vert D^{%
\vec{\alpha}}w\right\Vert _{L^{\infty }\left( \Omega ^{\prime }\right) }\leq 
\mathcal{C}_{N}$. H{ere $\tilde{\Gamma}=\Omega \times \left[ -2\left\Vert
\varphi \right\Vert _{L^{\infty }\left( \partial \Omega \right)
},2\left\Vert \varphi \right\Vert _{L^{\infty }\left( \partial \Omega
\right) }\right] $, $B$ denotes the various constants in Condition \ref{hyp2}%
, and }${\lambda _{0}}${$=\lambda _{0}\left( \mathcal{A},\Omega \right) $ is
the convex character of }$\partial \Omega ${. }\newline
Moreover, {if }$\vec{\gamma}\equiv 0$, then the solution $w$ is unique. 
\rm%

An important consequence of Theorem \ref{DP} in the case $\vec{\gamma}\equiv
0$ is the following interior regularity result:

\begin{theorem}[Regularity of solutions]
{\ \label{application}Let $\Omega \subset \mathbb{R}^{n}$ be a bounded
domain and $\Gamma =\Omega \times \mathbb{R}$. Suppose that $\vec{k}\in 
\mathcal{C}^{2}\left( \Gamma \right) $, that $\mathcal{A}\left( x,z\right) $%
, $f\left( x,z\right) \in \mathcal{C}^{\infty }\left( \Gamma \right) $ and
satisfy (\ref{appc01}),(\ref{appc02}), (\ref{appc07}), (\ref{appc03}) from
Theorem \ref{DP}, and that }$\mathcal{A}$ satisfies the super subordination
Condition{\ \ref{hyp2}} in $\Gamma ${. Then any weak solution $w$ of} 
\begin{equation*}
\mathop{\rm div}\mathcal{A}\left( x,w\right) \nabla w+f\left( x,w\right)
=0\qquad \text{in }\Omega
\end{equation*}%
{which is \emph{continuous} in $\Omega $ is also a strong solution and
satisfies $w\in \mathcal{C}^{\infty }\left( \Omega \right) $. Moreover, for
any nonnegative integer $N$ and subdomain $\Omega ^{\prime \prime }\Subset
\Omega ^{\prime }\Subset \Omega $, there exists a constant 
\begin{equation*}
\mathcal{C}_{N}=\mathcal{C}_{N}\left( \left\Vert w\right\Vert _{L^{\infty
}\left( \Omega ^{\prime }\right) },n,B,\vec{k},\Lambda ,\left\Vert \mathcal{A%
}\right\Vert _{\mathcal{C}^{N+2}\left( \tilde{\Gamma}\right) },\left\Vert
f\right\Vert _{\mathcal{C}^{N+1}\left( \tilde{\Gamma}\right) },\Omega
,\Omega ^{\prime },\Omega ^{\prime \prime }\right)
\end{equation*}%
such that }$\sum_{\left\vert \vec{\alpha}\right\vert \leq N}\left\Vert D^{%
\vec{\alpha}}w\right\Vert _{L^{\infty }\left( \Omega ^{\prime \prime
}\right) }\leq \mathcal{C}_{N}${. Here $\tilde{\Gamma}=\Omega ^{\prime
}\times \left[ -2\left\Vert w\right\Vert _{L^{\infty }\left( \Omega ^{\prime
}\right) },2\left\Vert w\right\Vert _{L^{\infty }\left( \Omega ^{\prime
}\right) }\right] $ and $B$ denotes the relevant constants in Condition \ref%
{hyp2}.}
\end{theorem}

{Note that in case $n=2,$ }the super subordination Condition{\ \ref{hyp2}
reduces to (\ref{thw2}) if $\mathcal{A}=\mathop{\rm diag}\left( 1,k\right) $
and $\vec{\gamma}=0$. Moreover, in this case, whether $\vec{\gamma}=0$ or
not, }the nondegeneracy Condition{\ \ref{hyp1} means that given $\left(
x_{1},x_{2}\right) \in \Omega $ and $\varepsilon >0$, there exist $%
0<r^{1},r^{2}<\varepsilon $ such that $k\left( x_{1}\pm r^{1},\xi
_{2},z\right) >0$ if $\left\vert \xi _{2}-x_{2}\right\vert \leq r^{2}$. \ }

{As an example in case $n\geq 2$, we consider a diagonal matrix 
\begin{equation*}
\mathcal{A}\left( x,z\right) =\left( 
\begin{array}{cccc}
1 & 0 & \cdots & 0 \\ 
0 & k^{2}\left( x,z\right) & \cdots & 0 \\ 
\vdots & \vdots & \ddots & \vdots \\ 
0 & 0 & \cdots & k^{n}\left( x,z\right)%
\end{array}%
\right) ,
\end{equation*}%
where $k^{i}$ are smooth nonnegative functions satisfying }the nondegeneracy
Condition{\ \ref{hyp1}, and such that 
\begin{equation*}
\left\vert k_{z}^{i}\left( x,z\right) \right\vert \leq B\,k^{\ast }\left(
x,z\right) ,\qquad i=1,\dots ,n,\quad \left( x,z\right) \in \Gamma ,
\end{equation*}%
with $k^{\ast }=\min \left( k^{1},\dots ,k^{n}\right) $. Then $\mathcal{A}$
satisfies the hypotheses of Theorem \ref{application}. In particular, if $%
k\left( x,z\right) $ is nonnegative and satisfies $\left\vert
k_{z}\right\vert =O\left( k\right) $ as $k\rightarrow 0$, then $\mathcal{A}=%
\mathop{\rm diag}\left( 1,k,\dots ,k\right) $ is an admissible matrix for
Theorem \ref{application} provided property (\ref{ax1c3}) of }the
nondegeneracy Condition{\ \ref{hyp1} holds. }

\begin{remark}
{\ As a consequence of the previous observations in the case $\mathcal{A}= %
\mathop{\rm diag}\left( 1,k^{2},\dots ,k^{n}\right) $, we \emph{partially}
recover Fedi\u{\i}'s two-dimensional result \cite{Fe} that $\partial
_{x_{1}}^{2}+k\left( x_{1}\right) \partial _{x_{2}}^{2}$ is hypoelliptic if $%
k$ is smooth and positive for all $x_{1}\neq 0$. Indeed, since $k\left(
x_{1}\right) $ is independent of the $z$ variable it automatically satisfies 
$\left\vert k_{z}\right\vert =O\left( k\right) $ as $k\rightarrow 0$. We
only partially recover Fedi\u{\i}'s result because our theorem applies only
to continuous weak solutions. }
\end{remark}

We {also obtain a partial extension (namely, for \emph{continuous} weak
solutions) to higher dimensions of Fedi\u{\i}'s result:}

\begin{theorem}
\label{fediihigh} {Let $k^{i}\left( x_{1},\dots ,x_{n}\right) $, $i=2,\dots
,n$,$\,$be smooth functions in $\mathbb{R}^{n}$ such that $k^{i}$ is
independent of the $i^{\text{th}}$ variable, i.e., 
\begin{equation*}
k^{i}\left( x_{1},\dots ,x_{n}\right) =k^{i}\left( \hat{x}^{i}\right) ,\quad 
\text{with }\hat{x}^{i}=\left( x_{1},\dots ,x_{i-1},x_{i+1},\dots
,x_{n}\right) ,
\end{equation*}%
and $k^{i}\left( x\right) >0$ if $\hat{x}^{i}\neq 0$. Then any continuous
weak solution of 
\begin{equation*}
\left\{ \partial _{x_{1}}^{2}+k^{2}\left( x\right) \partial
_{x_{2}}^{2}+\dots +k^{n}\left( x\right) \partial _{x_{n}}^{2}\right\} w=0
\end{equation*}%
in $\mathbb{R}^{n}\,$is smooth everywhere. }
\end{theorem}

\begin{remark}
{\ It is shown in \cite{KuStr} that if $k\left( x_{1}\right) $ is smooth and
positive for all $x_{1}\neq 0$, then $\partial _{x_{1}}^{2}+k\left(
x_{1}\right) \partial _{x_{2}}^{2}+\partial _{x_{3}}^{2}$ is hypoelliptic in 
$\mathbb{R}^{3}$ if and only if $\lim_{x_{1}\rightarrow 0}x_{1}\log k\left(
x_{1}\right) =0$. We are not able to recover this result since $k\left(
x_{1}\right) $ vanishes identically at all points of the form $(x_1,x_2,x_3)
= (0, x_2,x_3)$, and so the hypothesis of our theorem is not met. Also, our
solutions are required to be continuous. On the other hand, from Theorem \ref%
{fediihigh}, we see that if $k\left( x_{1},x_{3}\right) $ is smooth and
positive for all $\left( x_{1},x_{3}\right) \neq \left( 0,0\right) $, then
every weak solution of $\left\{ \partial _{x_{1}}^{2}+k\left(
x_{1},x_{3}\right) \partial _{x_{2}}^{2}+\partial _{x_{3}}^{2}\right\} w=0$
is smooth in the interior of the domain of continuity of $w$ in $\mathbb{R}%
^{3}$. }
\end{remark}

{The following is a striking consequence of Theorem \ref{application} in $%
\mathbb{R}^{2}$:}

\begin{theorem}
\label{striking} \label{RSaW3}{\ If $k\left( x_{1},x_{2},z\right) \,$is
smooth, nonnegative, satisfies 
\begin{equation}
\left\vert \partial _{z}k\right\vert =O\left( k\right) ,  \label{thw1}
\end{equation}%
and $k\left( \cdot ,\cdot ,0\right) $ \emph{does not vanish identically on
any horizontal line segment} in $\Omega $, then any continuous weak solution 
$w$ of 
\begin{equation}
\partial _{x_{1}}^{2}w+\partial _{x_{2}}k\left( x_{1},x_{2},w\left(
x_{1},x_{2}\right) \right) \partial _{x_{2}}w=0,\;\;\;\;\;\left(
x_{1},x_{2}\right) \in \Omega ,  \label{maeasy}
\end{equation}%
is smooth in $\Omega $.}
\end{theorem}

{%
\proof%
We will prove that the hypotheses of Theorem \ref{application} are satisfied
by $\vec{k}=\left( 1,k\left( x_{1},x_{2},z\right) \right) $ and $\mathcal{A}%
=\left( 
\begin{array}{ll}
1 & 0 \\ 
0 & k%
\end{array}%
\right) $. Since $k^{1}\equiv 1$, property (\ref{ax1c2}) in }the
nondegeneracy Condition{\ \ref{hyp1} is trivially satisfied. Since $k\left(
\cdot ,\cdot ,0\right) $ is nonnegative, continuous and does not vanish
identically on any horizontal segment, given $\varepsilon >0$ and $\left(
x_{1},x_{2}\right) \in \Omega $, there exist $\bar{x}_{1}<x_{1}<\bar{x}%
_{1}^{\prime }$ with $\left\vert x_{1}-\bar{x}_{1}\right\vert =\left\vert
x_{1}-\bar{x}_{1}^{\prime }\right\vert =r^{1}<\varepsilon $, such that $%
k\left( \bar{x}_{1},x_{2},0\right) >0$ and $k\left( \bar{x}_{1}^{\prime
},x_{2},0\right) >0$. From (\ref{thw1}) and Lemma \ref{chaz} it follows that 
$k\left( x_{1},x_{2},z\right) $ is either identically zero as a function of $%
z$ or strictly positive in $z$, and hence $k\left( \bar{x}%
_{1},x_{2},z\right) >0$ and $k\left( \bar{x}_{1}^{\prime },x_{2},z^{\prime
}\right) >0$. Then property (\ref{ax1c3}) in }the nondegeneracy Condition{\ %
\ref{hyp1} follows from the continuity of $k$ with respect to the second
variable and therefore $\vec{k}$ satisfies hypothesis (\ref{appc01}) in
Theorem \ref{application}. Since (\ref{thw1}) holds, $\mathcal{A}$ satisfies
(\ref{xtra}). Since $\mathcal{A}$ is diagonal, (\ref{xxtra}) then follows
from Remark \ref{diagcase}. Thus $\mathcal{A}$ is super subordinate{, so
hypotheses (\ref{appc02}) and (\ref{appc04}) in Theorem \ref{application}
are satisfied. Since $f=0$, hypothesis (\ref{appc07}) of Theorem \ref%
{application} is trivially satisfied. 
\endproof%
} }

{From Theorem \ref{striking} and the techniques used in \cite{SaW}, we can
derive an extension of Theorem 2.1 in \cite{SaW}, which is a regularity
result for convex solutions to Monge-Amp\`{e}re equations. Consider a
smooth, bounded, strongly convex domain $\Phi \subset \mathbb{R}^{2}$. Given
a convex function $u\in \mathcal{C}^{1}\left( \Phi \right) $, following \cite%
{SaW} we set 
\begin{equation*}
\omega _{-}\left( s\right) =\omega _{-}\left( s,\Phi ,u\right)
=\inf_{t:\left( s,t\right) \in \Phi }u_{t}\left( s,t\right) \quad \text{and}%
\quad \omega _{+}\left( s\right) =\omega _{+}\left( s,\Phi ,u\right)
=\sup_{t:\left( s,t\right) \in \Phi }u_{t}\left( s,t\right)
\end{equation*}%
for any $s$ lying in the projection of $\Phi $ onto the $s$ -axis. Let 
\begin{equation*}
\mathtt{I}_{s}=\left\{ \omega :\omega _{-}\left( s\right) <\omega <\omega
_{+}\left( s\right) \right\} \quad \text{if $\omega _{-}\left( s\right)
<\omega _{+}\left( s\right) $ and $\mathtt{I}_{s}=\emptyset $ otherwise, and}
\end{equation*}%
\begin{equation*}
\Phi _{u}=\left\{ \left( s,t\right) \in \Phi :u_{t}(s,t)\in \mathtt{I}%
_{s}\right\} .
\end{equation*}%
Note that}

\begin{itemize}
\item $u\left( s,t \right) $ is affine in the $t$ variable if and only if $%
\mathtt{I}_{s}=\emptyset $.

\item If $u\in \mathcal{C}^{1}\left( \overline{\Phi }\right) $ is not affine
in the $t$ variable for any fixed $s$, {then }$\Phi _{u}$ is an open
connected set. Indeed, to see that $\Phi _{u}$ is open, let $\left(
s,t\right) \in \Phi _{u}$. Then $\omega _{-}(s)<u_{t}\left( s,t\right)
<\omega _{+}(s)$, and since $u\in \mathcal{C}^{1}(\overline{\Phi })$, the
functions $\omega _{-},\omega _{+}$ and $u_{t}$ are continuous. Hence, for $%
(s^{\prime },t^{\prime })$ near $(s,t)$, $\omega _{-}(s^{\prime
})<u_{t}\left( s^{\prime },t^{\prime }\right) <\omega _{+}(s^{\prime })$.
Therefore $\Phi _{u}$ is open. To see that $\Phi _{u}$ is connected it
suffices to show it is pathwise connected. This follows from the fact that
the midpoint $(\omega _{-}(s)+\omega _{+}(s))/2$ of I$_{s}$ is a continuous
function of $s$. Note the arguments above only use the continuity of $u_{t}$%
. See \cite{SaW} for further discussion about when $\Phi _{u}$ is connected.

\item Even if $u$ is not affine in the $t$ variable for any fixed $s$, the
set 
\begin{equation*}
\Pi _{u}=\left\{ \left( s,t\right) \in \Phi _{u}:u_{t}\left( s,t\right)
=\omega _{-}\left( s\right) ~\text{or}~u_{t}\left( s,t\right) =\omega
_{+}\left( s\right) \right\}
\end{equation*}%
may be non-empty. This exceptional set $\Pi _{u}$ is composed by what are
called \textquotedblleft Pogorelov segments\textquotedblright\ in \cite{SaW}.
\end{itemize}

\begin{theorem}
\label{RSaW2} Suppose $k\left( s,t\right) \,$is smooth, nonnegative,
satisfies $\left\vert \partial _{t}k\right\vert =O\left( k\right) $, and $%
k\left( \cdot ,0\right) $ \emph{does not vanish identically on any
horizontal line segment} in $\Phi $, where $\Phi $ is as above. If $u\in 
\mathcal{C}^{1}\left( \overline{\Phi }\right) $ is a convex solution of the
Monge-Amp\`{e}re boundary value problem 
\begin{eqnarray*}
\det \left( 
\begin{array}{cc}
u_{ss} & u_{st} \\ 
u_{ts} & u_{tt}%
\end{array}
\right) &=&k\left( s,t\right) ,\;\;\;\;\;\left( s,t\right) \in \Phi , \\
u &=&\phi \left( s,t\right) ,\qquad \left( s,t\right) \in \partial \Phi ,
\end{eqnarray*}
where $\phi $ is smooth on $\partial \Phi $, then $u$ is smooth in $\Phi
_{u}.$
\end{theorem}

{To obtain our main results, we follow the approach in \cite{SaW} for
two-dimensional equations, although our objectives are more general. We
consider equations in any dimension at least two, our equations may include
a first order drift term and a zero order term, and our notion of solution
is more general since we only require continuity instead of Lipschitz
continuity. To prove our main hypoellipticity result, Theorem \ref%
{application}, we use an approximation argument based on the a priori
estimates in Theorem \ref{quasi} and on the construction in Lemma \ref%
{barrier} of new custom-built barriers.} One of the core ingredients needed
to derive Theorem \ref{quasi} is the interpolation inequality given in Lemma %
\ref{lemma-interpolation}, proved in \cite{RSaW1}.

\section{Sharpness\label{section-sharp}}

Our results are sharp in the sense that the power $1$ in the super
subordination Condition \ref{hyp2} cannot be decreased. Indeed, we will now
show that for any $\varepsilon >0,$ there exists a nonnegative smooth
function $k=k\left( x,y,z\right) $ in $\mathbb{R}^{2}\times \mathbb{R}$
which is not identically zero on any horizontal segment in $\mathbb{R}^{2}$
(moreover $k>0$ unless $x=z=0$) and satisfies 
\begin{equation*}
\left\vert \partial _{z}k\left( x,y,z\right) \right\vert \leq C~\left[
k\left( x,y,z\right) \right] ^{1-\varepsilon }\qquad \text{in }\Omega \times 
\mathbb{R},
\end{equation*}%
and that there is a continuous weak solution $w$ of {%
\begin{equation}
\partial _{x}^{2}w+\partial _{y}k\left( x,y,w\left( x,y\right) \right)
\partial _{y}w=0  \label{equation-x}
\end{equation}%
in any neighborhood $\Omega $ of the origin, \emph{but $w$ is not smooth} in 
}$\Omega $. This example is derived by applying the partial Legendre
transform to non-smooth solutions of the Monge-Amp\'{e}re equation which are
suitable powers of the distance function to the origin.

Given $\varepsilon >0,$ let $m$ be a positive integer such that 
\begin{equation}
\frac{1}{4m-2}<\varepsilon .  \label{meps}
\end{equation}%
Consider the function $w=w\left( x,y\right) $ defined implicitly by the
equation 
\begin{equation}
F(x,y,w)=0\quad \text{where $F(x,y,z)=z\left( |x|^{2}+|z|^{2}\right) ^{m-%
\frac{1}{2}}+y$.}  \label{eq-implicit}
\end{equation}

\begin{center}
\includegraphics[width=3.5in]{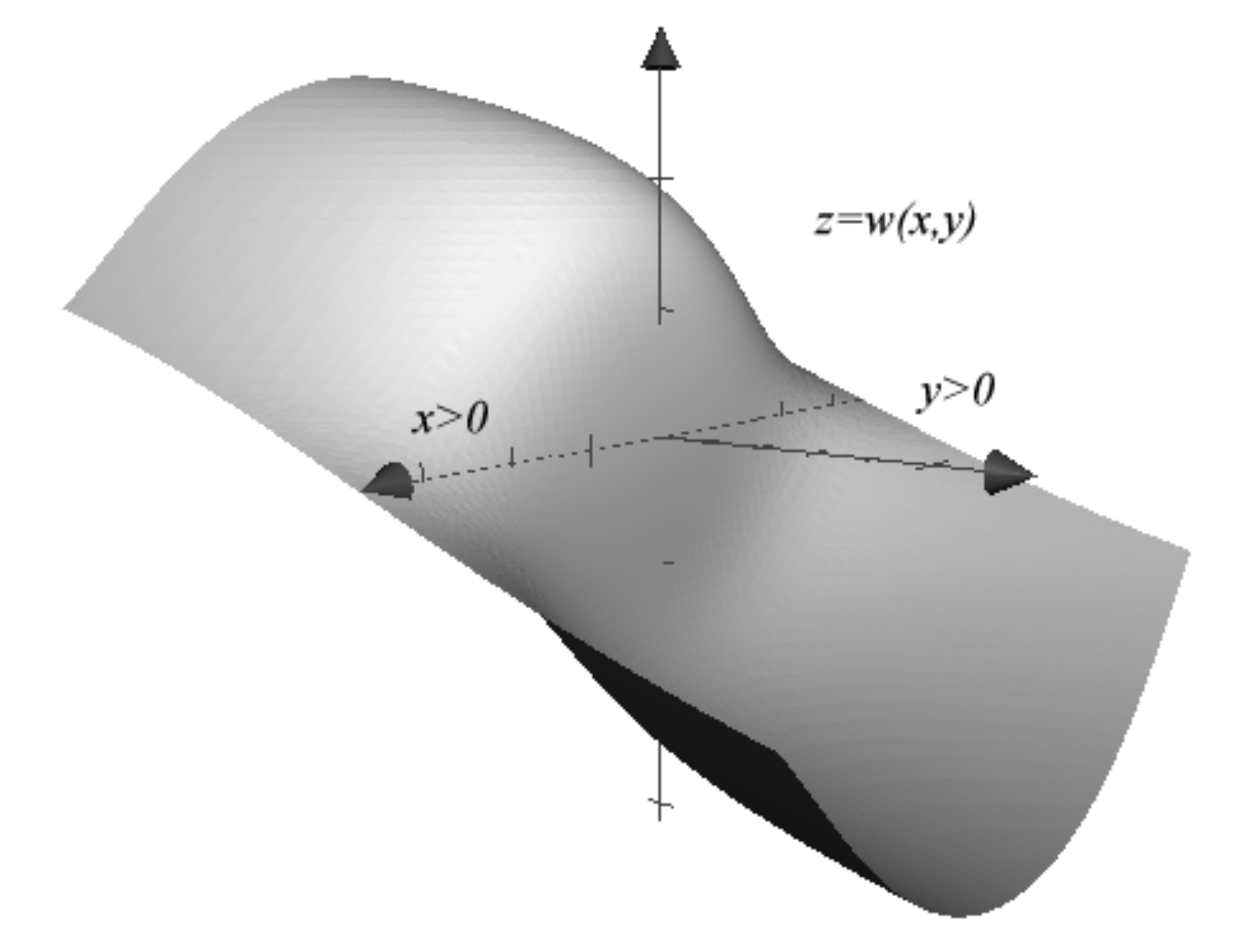}
\end{center}

Since $F(0,0,0)=0$ and $F_{z}(x,y,z)=(|x|^{2}+|z|^{2})^{m-\frac{3}{2}%
}(x^{2}+2mz^{2})$, it follows from the implicit function theorem that $%
w=w(x,y)$ is well-defined by (\ref{eq-implicit}) and smooth in $\mathbb{R}%
^{2}\setminus \{(0,0)\}$. Also, $w$ extends continuously to $(x,y)=(0,0)$
with $w(0,0)=0$, and thus for any neighborhood $\Omega $ of the origin, $%
w\in \mathcal{C}^{0}\left( \overline{\Omega }\right) \bigcap \mathcal{C}%
^{\infty }\left( \Omega \backslash \left\{ \left( 0,0\right) \right\}
\right) $. Moreover, $w(x,y)=0$ if and only if $y=0$.

Now let 
\begin{equation*}
f\left( x,y\right) =x\left( |x|^{2}+|w(x,y)|^{2}\right) ^{m-\frac{1}{2}%
}\qquad \left( x,y\right) \in \Omega .
\end{equation*}%
Then $f(0,y)=0$, $f(x,0)=x|x|^{2m-1}$ and $f(x,y)=-\frac{xy}{w(x,y)}$ if $%
w(x,y)\neq 0$. By direct computation, if $(x,y)\neq (0,0)$, 
\begin{equation}
\left\{ 
\begin{array}{rcl}
f_{x}\left( x,y\right) & = & -2m\left( x^{2}+w^{2}\right) ^{2m-1}w_{y}\left(
x,y\right) \\ 
f_{y}\left( x,y\right) & = & w_{x}\left( x,y\right) .%
\end{array}%
\right.  \label{eq-CR}
\end{equation}%
Thus, if $k\left( x,y,z\right) =k\left( x,z\right) $ is the smooth function
in $\mathbb{R}^{2}\times \mathbb{R}$ defined by 
\begin{equation*}
k\left( x,z\right) =2m\left( x^{2}+z^{2}\right) ^{2m-1},
\end{equation*}%
then by using the formulas $F_{x}=(2m-1)xz(|x|^{2}+|z|^{2})^{m-\frac{3}{2}}$
and $F_{y}=1$, we have 
\begin{eqnarray*}
\left\vert w_{x}\right\vert ^{2} &=&\left\vert -\frac{\left( 2m-1\right) xw}{%
x^{2}+2mw^{2}}\right\vert ^{2}\leq \left\vert \frac{\left( 2m-1\right) }{2}%
\frac{x^{2}+w^{2}}{x^{2}+2mw^{2}}\right\vert ^{2}\leq m^{2}, \\
k\left\vert w_{y}\right\vert ^{2} &=&\frac{2m\left( x^{2}+w^{2}\right)
^{2m-1}}{\left( x^{2}+w^{2}\right) ^{2m-3}\left( x^{2}+2mw^{2}\right) ^{2}}
\\
&=&2m\frac{\left( x^{2}+w^{2}\right) ^{2}}{\left( x^{2}+2mw^{2}\right) ^{2}}%
\leq 2m.
\end{eqnarray*}%
In particular, this implies that%
\begin{equation*}
\int_{\Omega }(\left\vert w_{x}\right\vert ^{2}+k\left\vert w_{y}\right\vert
^{2})~dxdy\leq 3m^{2}\left\vert \Omega \right\vert .
\end{equation*}%
that is, $w\in H_{\mathcal{X}}^{1,2}\left( \Omega \right) \bigcap L^{\infty
}\left( \Omega \right) $, where $\mathcal{X}\left( x,y,w,\xi _{1},\xi
_{2}\right) =\xi _{1}^{2}+k\xi _{2}^{2}$. From (\ref{eq-CR}) it follows that 
$w\left( x,y\right) $ is a continuous weak solution of the quasilinear
equation (\ref{equation-x}). Moreover, as a function of $(x,y)$, $k\left(
x,y,z\right) $ does not vanish on any horizontal line segment, and it
satisfies 
\begin{eqnarray}
\left\vert \partial _{z}k\left( x,z\right) \right\vert &=&4m\left(
2m-1\right) \left\vert z\right\vert \left( x^{2}+z^{2}\right) ^{2m-2}  \notag
\\
&\leq &C_{m}\left( |x|^{2}+|z|^{2}\right) ^{2m-\frac{3}{2}}=C_{m}k\left(
x,z\right) ^{1-\frac{1}{4m-2}}\leq C_{m}k^{1-\varepsilon },  \label{xtra-k}
\end{eqnarray}%
where we used the inequality $\left\vert z\right\vert \leq \sqrt{x^{2}+z^{2}}
$ and the bound (\ref{meps}) for $m$.

On the other hand, from (\ref{eq-implicit}), noting that $y$ and $w(x,y)$
have the same sign, we have 
\begin{equation*}
w\left( 0,y\right) =\left( \mathop{\rm sign}y\right) \left\vert y\right\vert
^{\frac{1}{2m}}.
\end{equation*}
Hence $w\left( x,y\right) \not\in \mathcal{C}^{\frac{1}{2m}+ \delta
}(\Omega) $ for any $\delta >0$. In particular, $w$ is not smooth in $\Omega 
$. {%
\endproof%
}

\section{Preliminaries}

\subsection{{Notation\label{notation}}}

{\ Throughout the paper, $C$ will denote a constant that may change from
line to line but that is independent of any significant quantities. In
general, $C$ may depend on the dimension $n$, $\vec{k},$ and the fixed
cutoff functions defined below. We will use calligraphic $\mathcal{C}$ to
denote a function of one or more variables, increasing in each variable
separately, that may also change from line to line, but that is independent
of any significant parameters except its variables. }

{We denote by $\left\Vert f\right\Vert _{L^{\infty }\left( \Omega \right) }$
the essential supremum of $|f|$ in $\Omega $\ and by $\left\Vert
f\right\Vert _{L^{p}\left( \Omega \right) }$ the $L^{p}$-norm of $f\,$in $%
\Omega $: 
\begin{equation*}
\left\Vert f\right\Vert _{L^{p}\left( \Omega \right) }=\left( \int_{\Omega
}\left\vert f\right\vert ^{p}dx\right) ^{\frac{1}{p}},\quad 1\le p <\infty.
\end{equation*}
When $\Omega =\mathbb{R}^{n}$ we omit mentioning the set. The collection of
real-valued functions in $\Omega $ with $m\,$continuous (but not necessarily
bounded) derivatives in $\Omega $ will be denoted $\mathcal{C}^{m}\left(
\Omega \right) $, and for $f\in \mathcal{C}^{m}\left( \Omega \right) $ we
let 
\begin{equation}
\left\Vert f\right\Vert _{\mathcal{C}^{m}\left( \Omega \right)
}=\sum_{i=0}^{m}\sum_{\left\vert \vec{\alpha}\right\vert =i}\left\Vert
\partial ^{\vec{\alpha}}f\right\Vert _{L^{\infty }\left( \Omega \right) },
\label{cmnorm}
\end{equation}
where $\partial^{\vec{\alpha}}=\left(\partial _{1}\right) ^{\alpha
_{1}}\cdots \left( \partial _{n}\right) ^{\alpha _{n}}$} for $\vec{\alpha}%
=\left( \alpha _{1},\dots ,\alpha _{n}\right) $.

{Let $\Gamma _{M_{0}}^{\prime }\subset \Gamma $ be the domains in $\mathbb{R}%
^{n}\times \mathbb{R}$ given in (\ref{domain}) and let $k^{i}\left(
x,z\right) \in \mathcal{C}^{2}\left( \Gamma \right) $ be nonnegative, $%
i=1,\dots ,n$. From now on, let $\vec{k}(x,z)=\left( k^{1}(x,z),\dots
,k^{n}(x,z)\right) $ satisfy the nondegeneracy Condition~ \ref{hyp1} with 
\begin{equation}
k^{1}\left( x,z\right) =1\qquad \text{for all }\left( x,z\right) \in \Gamma ,
\label{unifbdk1}
\end{equation}%
which clearly implies property (\ref{ax1c2}) in }Condition{\ \ref{hyp1}.
Since our theorems are local, this assumption causes no loss of generality. }

\subsection{Boxes around points}

{Given $x\in \Omega $, we will consider rectangular neighborhoods }$\mathcal{%
\ R}${\ of $x$ of the form described in Section \ref{statements}, with $x\in 
\frac{1}{3}\mathcal{R}$. The maximum sidelength }$R${\ of }$\mathcal{R}${\
will be chosen so that $2\mathcal{R}\subset \Omega $ and possibly even
smaller to allow the absorption of various terms involving $R$ as a factor.
We will always assume that }$\mathcal{R}${\ satisfies property (\ref{ax1c3})
of }the nondegeneracy Condition{\ \ref{hyp1}. }Since{\ $\vec{k}$ is
continuous, given such $\mathcal{R}$ and $M_{0}>0,$ there exist positive
numbers $\delta ^{i}=\delta ^{i}\left( \vec{k},\mathcal{R},M_{0}\right) $
such that $\delta ^{i}<\frac{1}{2}r^{i}$, where $r^{i}$ denotes the $i^{th}$
half-sidelength of $\mathcal{R}$, $i=1,\dots ,n$, and } 
\begin{eqnarray}
k^{i}\left( y,z\right) &>&0\qquad \text{whenever }z\in \left[ -M_{0},M_{0}%
\right] \text{ and}  \label{lidrfat} \\
y &\in &\mathcal{T}_{i}\left( \mathcal{R},\delta ^{i}\right) =\left\{ Y\in 
\mathbb{R}^{n}:\mathop{\rm dist}\left( Y,\mathcal{T}_{i}\left( \mathcal{R}%
\right) \right) \leq \delta ^{i}\right\} .  \notag
\end{eqnarray}

\begin{remark}
\label{noM0} Under the hypotheses of Theorem \ref{application}, or more
precisely, if $\mathcal{A}$ satisfies (\ref{xtra}) and (\ref{hellip}), the
parameters $\delta ^{i}$ above can be taken independent of $M_{0}$. Indeed,
( \ref{xtra}) and Lemma \ref{admisall} imply that for fixed $x,\xi \in 
\mathbb{R} ^{n}$, the function $\xi ^{t}\mathcal{A}\left( x,z \right) \xi $
is either identically zero in $z$ or strictly positive in $z$. Therefore, (%
\ref{hellip}) implies a similar property for each $k^{i}\left( x,z\right) $.
\end{remark}

\subsection{A class of adapted cutoff functions}

A \emph{cutoff function} is any nonnegative smooth function with compact
support, i.e., $\varphi \,$is a cutoff function if $\varphi \in \mathcal{C}
_{0}^{\infty }\left( \mathbb{R}^{n}\right) $ and $\varphi \geq 0$.

We now define a special class of cutoff functions around $x$ which are
adapted to our operator as in \cite{SaW2} (see also \cite{Fe}). The main
property of these functions is that they are supported in a (small enough)
neighborhood of $x$, while their derivatives are supported \emph{away} from $%
x$, essentially where $\vec{k}$ has positive components.

\begin{definition}[Supporting relation]
\label{rhost}Given two cutoff functions $\zeta $, $\xi \in \mathcal{C}
_{0}^{\infty }\left( \Omega \right) $, we say that $\zeta $ \emph{supports} $%
\xi $ and denote it $\xi \succeq \zeta $ or $\zeta \preceq \xi $ if $\xi =1 $
on a neighborhood of $\mathop{\rm support}\left( \zeta \right) $. Note in
particular that if $\xi \succeq \zeta $ then $\zeta \,\xi =\zeta $ and $%
\left\Vert \zeta \right\Vert _{L^{\infty }}\xi \geq \zeta $.
\end{definition}

\begin{definition}[Special cutoff functions]
\label{rhos}Let $x\in \Omega $, $M_{0}\geq 1$, $\mathcal{R}=\tilde{x}+ \big(%
\left[-r^{1},r^{1}\right] \times ...\times \left[ -r^{n},r^{n}\right]\big) $
be a rectangular box with $x\in \frac{1}{3}\mathcal{R}$, and $\delta
^{i}=\delta^{i}( \vec{k},\mathcal{R},M_{0}) >0$ be as in (\ref{lidrfat}).
Let 
\begin{equation*}
\eta _{i},\phi _{i},\zeta _{i},\theta _{i}\in \mathcal{C}_{0}^{\infty
}\left( \tilde{x}_{i}+\left( -2r^{i},2r^{i}\right) \right) ,\qquad 1\leq
i\leq n,
\end{equation*}
be functions of one variable which satisfy the following:

\begin{enumerate}
\item \label{one}$0\leq \eta _{i},$ $\phi _{i},~\zeta _{i},~\theta _{i}\leq
1 $

\item \label{intervals}$\eta _{i}$, $\phi _{i}$ and$~\zeta _{i}$ are equal
to $1$ in $\tilde{x}_{i}+\left[ -\left( r^{i}-\delta ^{i}\right)
,r^{i}-\delta ^{i}\right] $ and vanish \newline
outside $\tilde{x}_{i}+\left[ -\left( r^{i}+\delta ^{i}\right) ,r^{i}+\delta
^{i}\right] $; and for all integers $m\geq 1$ and some universal constants $%
c_{m}\geq 1$, 
\begin{equation}
\left( \frac{1}{c_{m}\delta ^{i}}\right) ^{m}\leq \left\Vert \frac{d^{m}}{
dt^{m}}\eta _{i}\right\Vert _{L^{\infty }\left( \mathbb{R}\right)
}+\left\Vert \frac{d^{m}}{dt^{m}}\phi _{i}\right\Vert _{L^{\infty }\left( 
\mathbb{R}\right) }+\left\Vert \frac{d^{m}}{dt^{m}}\zeta _{i}\right\Vert
_{L^{\infty }\left( \mathbb{R}\right) }\leq \left( \frac{c_{m}}{\delta ^{i}}
\right) ^{m}.  \label{intervalsconstants}
\end{equation}

\item $\eta _{i}\preceq \phi _{i}\preceq \zeta _{i}$ and $\eta _{i}\leq \phi
_{i}\leq \zeta _{i}$

\item \label{includd}Let $J^{i}\,$be the smallest set of the form $J^{i}= 
\tilde{x}_{i}+ \left(\left[ -b^{i},-a^{i}\right] \bigcup \left[ a^{i},b^{i}%
\right] \right) \subset \mathbb{R}$ ,with $0<a^{i}<b^{i}$, such that 
\begin{equation*}
\mathop{\rm support}\left( \eta _{i}^{\prime }\right) \bigcup \mathop{\rm
support}\left( \phi _{i}^{\prime }\right) \bigcup \mathop{\rm support}\left(
\zeta _{i}^{\prime }\right) \subset J^{i}.
\end{equation*}
Note by (\ref{intervals}) that $J^{i}\subset \tilde{x}_{i}+ \left(\left[
-r^{i}-\delta ^{i},-r^{i}+\delta ^{i}\right] \bigcup \left[
r^{i}-\delta^{i},r^{i}+\delta ^{i}\right]\right)$. Let $\theta _{i}\equiv 1$
on $J^{i}$, so that $\theta _{i}\equiv 1$ on the supports of $\eta
_{i}^{\prime }$, $\phi _{i}^{\prime }$ and $\zeta _{i}^{\prime }$, i.e., $%
\eta _{i}^{\prime }\preceq \theta _{i},\,\phi _{i}^{\prime }\preceq \theta
_{i},$ and $\zeta _{i}^{\prime }\preceq \theta _{i}$.

\item {\ \label{five}$\mathop{\rm support}\left( \theta _{i}\right) \subset 
\tilde{x}_{i}+\left(\left[ -r^{i}-2\delta ^{i},-r^{i}+2\delta ^{i}\right]
\bigcup \left[ r^{i}-2\delta ^{i},r^{i}+2\delta ^{i}\right] \right)$; in
particular, $\tilde{x}_{i}\notin \mathop{\rm
support}\left( \theta _{i}\right) $. Moreover, we assume that }for all
integers $m\geq 1${\ and for some universal constants $C_{m}>0,$} {\ 
\begin{equation}
\left( \frac{1}{C_{m}\delta ^{i}}\right) ^{m}\leq \left\Vert \frac{d^{m}}{%
dt^{m}}\theta _{i}\right\Vert _{L^{\infty }\left( \mathbb{R}\right) }\leq
\left( \frac{C_{m}}{\delta ^{i}}\right) ^{m}.  \label{fiveconstants}
\end{equation}
}

\item \label{six}$\left\vert \eta _{i}^{\prime }\right\vert $, $\left\vert
\phi _{i}^{\prime }\right\vert $ and $\left\vert \zeta _{i}^{\prime
}\right\vert $ are smooth functions (see the discussion following the figure
below). \newline
\end{enumerate}
\end{definition}

Finally, let 
\begin{equation*}
\begin{array}{ll}
\eta \left( x\right) =\prod_{i=1}^{n}\eta _{i}\left( x_{i}\right) ,\qquad & 
\phi \left( x\right) =\prod_{i=1}^{n}\phi _{i}\left( x_{i}\right) , \\ 
\zeta \left( x\right) =\prod_{i=1}^{n}\zeta _{i}\left( x_{i}\right) , & 
\varrho _{i}\left( x\right) =\theta _{i}\left( x_{i}\right) \prod_{j\neq
i}\zeta _{j}\left( x_{j}\right) ,%
\end{array}%
\end{equation*}%
and let $\xi ,\chi \in \mathcal{C}_{0}^{\infty }\left( 2\mathcal{R}\right) $
satisfy 
\begin{eqnarray*}
\xi &=&\chi =1\quad \text{in }\mathcal{R} \\
\zeta &\preceq &\xi \preceq \chi ,~\quad \zeta \leq \xi \leq \chi ,\quad 
\text{and} \\
\varrho _{i} &\preceq &\xi ,\quad i=1,\dots ,n.
\end{eqnarray*}

The figure below will serve as a reminder of the order between some of the
cutoff functions:

\begin{center}
\includegraphics[width=4.04in]{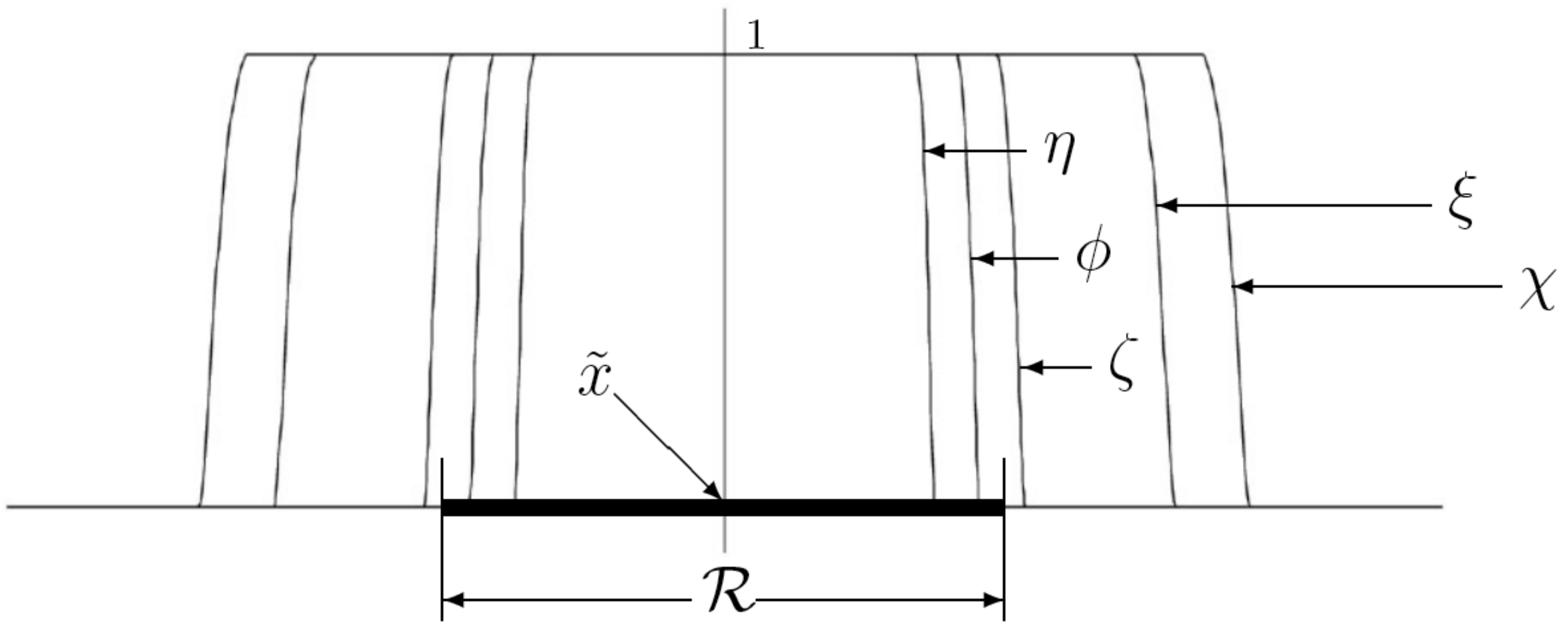}
\end{center}

Property (\ref{six}) above is easily satisfied by assuming (in addition to
properties (\ref{one}) to (\ref{includd})) that $\eta _{i}$, $\phi _{i}$ and 
$\zeta _{i}$ are smooth, non-decreasing in $\left( -\infty ,\tilde{x}%
_{i}\right) $ and non-increasing in $\left( \tilde{x}_{i},\infty \right) $.
Indeed, under such conditions and since these functions are constant on a
neighborhood of $\tilde{x}_{i}$, their derivatives are of the form $\eta
_{i}^{\prime }=\left( \eta _{i}^{\prime }\right) ^{+}-\left( \eta
_{i}^{\prime }\right) ^{-}$, $\phi _{i}^{\prime }=\left( \phi _{i}^{\prime
}\right) ^{+}-\left( \phi _{i}^{\prime }\right) ^{-}$ and $\zeta
_{i}^{\prime }=\left( \zeta _{i}^{\prime }\right) ^{+}-\left( \zeta
_{i}^{\prime }\right) ^{-}$ where all these functions are compactly
supported, $\left( \eta _{i}^{\prime }\right) ^{+}$, $\left( \phi
_{i}^{\prime }\right) ^{+}$ and $\left( \zeta _{i}^{\prime }\right) ^{+}$
are smooth, supported in $\left( -\infty ,\tilde{x}_{i}\right) $ and
nonnegative, while $\left( \eta _{i}^{\prime }\right) ^{-}$, $\left( \phi
_{i}^{\prime }\right) ^{-}$ and $\left( \zeta _{i}^{\prime }\right) ^{-}$
are smooth, supported in $\left( \tilde{x}_{i},\infty \right) $ and
nonnegative. It follows that $\left\vert \eta _{i}^{\prime }\right\vert
=\left( \eta _{i}^{\prime }\right) ^{+}+\left( \eta _{i}^{\prime }\right)
^{-}$, $\left\vert \phi _{i}^{\prime }\right\vert =\left( \phi _{i}^{\prime
}\right) ^{+}+\left( \phi _{i}^{\prime }\right) ^{-}$ and $\left\vert \zeta
_{i}^{\prime }\right\vert =\left( \zeta _{i}^{\prime }\right) ^{+}+\left(
\zeta _{i}^{\prime }\right) ^{-}$ are smooth functions.

It will be convenient to set 
\begin{eqnarray}
\qquad A^{6} &=&1+\left\Vert \nabla \eta \right\Vert _{L^{\infty
}}^{6}+\left\Vert \nabla \phi \right\Vert _{L^{\infty }}^{6}+\left\Vert
\nabla \zeta \right\Vert _{L^{\infty }}^{6}+\left\Vert \nabla \varrho
_{1}\right\Vert _{L^{\infty }}^{6}+\dots +\left\Vert \nabla \varrho
_{n}\right\Vert _{L^{\infty }}^{6}  \label{defA} \\
&&+\left\Vert \nabla ^{2}\eta \right\Vert _{L^{\infty }}^{3}+\left\Vert
\nabla ^{2}\phi \right\Vert _{L^{\infty }}^{3}+\left\Vert \nabla ^{2}\zeta
\right\Vert _{L^{\infty }}^{3}  \notag \\
&&+\left\Vert \nabla ^{3}\eta \right\Vert _{L^{\infty }}^{2}+\left\Vert
\nabla ^{3}\phi \right\Vert _{L^{\infty }}^{2}+\left\Vert \nabla ^{3}\zeta
\right\Vert _{L^{\infty }}^{2}  \notag
\end{eqnarray}
in order to collect constants in front of the lower order terms in what
follows.

\begin{remark}
\label{AdependsonR}Note that $A$ depends only on $\vec{k}$, $\mathcal{R}$
and $M_{0}$. Let $\delta ^{\ast }=\delta ^{\ast }\left( \vec{k}, \mathcal{R}%
,M_{0}\right) =\min_{i=1,\dots ,n}\delta ^{i}$ where $\delta ^{i}=\delta
^{i}\left( \vec{k},\mathcal{R},M_{0}\right) $ are as in (\ref{lidrfat}).
Then from (\ref{defA}),(\ref{intervalsconstants}) and (\ref{fiveconstants}), 
\begin{equation*}
A\approx \left( \delta ^{\ast }\right) ^{-1}.
\end{equation*}
\end{remark}

The main property of the cutoff functions $\eta $, $\phi $ and $\zeta $
above is that their $i^{th}$ partial derivative, $i=1,\dots ,n$, is
supported in the set $\mathcal{K}_{i}=\bigcap_{\ell \neq i,~\left\vert
z\right\vert \leq M_{0}}\left\{ x:k^{\ell }\left( x,z\right) >0\right\} $.
Indeed, let $\sigma \left( x\right) =\prod_{\ell =1}^{n}\sigma _{\ell
}\left( x_{\ell }\right) $ with $\eta _{\ell }\leq \sigma _{\ell }\leq \zeta
_{\ell }$, and fix $z\in \left[ -M_{0},M_{0}\right] $. It follows from
Definition \ref{rhos} (\ref{includd}) and (\ref{five}) that $\sigma _{\ell
}^{\prime }\preceq \theta _{\ell }$, i.e., $\mathop{\rm
supp}\left( \sigma _{\ell }^{\prime }\right) \subset \left\{ \theta _{\ell
}=1\right\} $. Hence, 
\begin{equation*}
\mathop{\rm supp}\left( \sigma _{i}^{\prime }\right) \subset \mathop{\rm
supp}\left( \theta _{i}\right) \subset \tilde{x}_{i}+\left( \left[
-r^{i}-2\delta ^{i},-r^{i}+2\delta ^{i}\right] \bigcup \left[ r^{i}-2\delta
^{i},r^{i}+2\delta ^{i}\right] \right) .
\end{equation*}%
Hence, since $\mathop{\rm supp}(\sigma _{\ell })\subset \tilde{x}_{\ell }+%
\left[ -r^{\ell }-\delta ^{\ell },r^{\ell }+\delta ^{\ell }\right] $, we
have $\mathop{\rm supp}\left( \partial _{i}\sigma \right) \subset %
\mathop{\rm supp}\left( \varrho _{i}\right) $. Now set $\mathrm{I}_{\ell }=%
\left[ -r^{\ell }-\delta ^{\ell },r^{\ell }+\delta ^{\ell }\right] $ and $%
\mathcal{R}_{i}=\mathcal{R}_{i}^{+}\bigcup \mathcal{R}_{i}^{-}$, where 
\begin{equation}
\begin{array}{l}
\mathcal{R}_{i}^{+}=\tilde{x}+\left\{ \left( \prod_{\ell <i}\mathrm{I}_{\ell
}\right) \times \left[ r^{i}-2\delta ^{i},r^{i}+2\delta ^{i}\right] \times
\left( \prod_{\ell >i}\mathrm{I}_{\ell }\right) \right\} , \\ 
\mathcal{R}_{i}^{-}=\tilde{x}+\left\{ \left( \prod_{\ell <i}\mathrm{I}_{\ell
}\right) \times \left[ -(r^{i}+2\delta ^{i}),-(r^{i}-2\delta ^{i})\right]
\times \left( \prod_{\ell >i}\mathrm{I}_{\ell }\right) \right\} .%
\end{array}
\label{rss}
\end{equation}%
From (\ref{lidrfat}) it follows that 
\begin{equation*}
\mathop{\rm supp}\left( \partial _{i}\sigma \right) \subset \mathop{\rm supp}%
\left( \varrho _{i}\right) \subset \mathcal{R}_{i}\subset \mathcal{K}_{i},
\end{equation*}%
as wanted. The figure below represents the sets $\mathcal{R}_{i}$ when $n=3$.

\begin{center}
\includegraphics[width=5.56in]{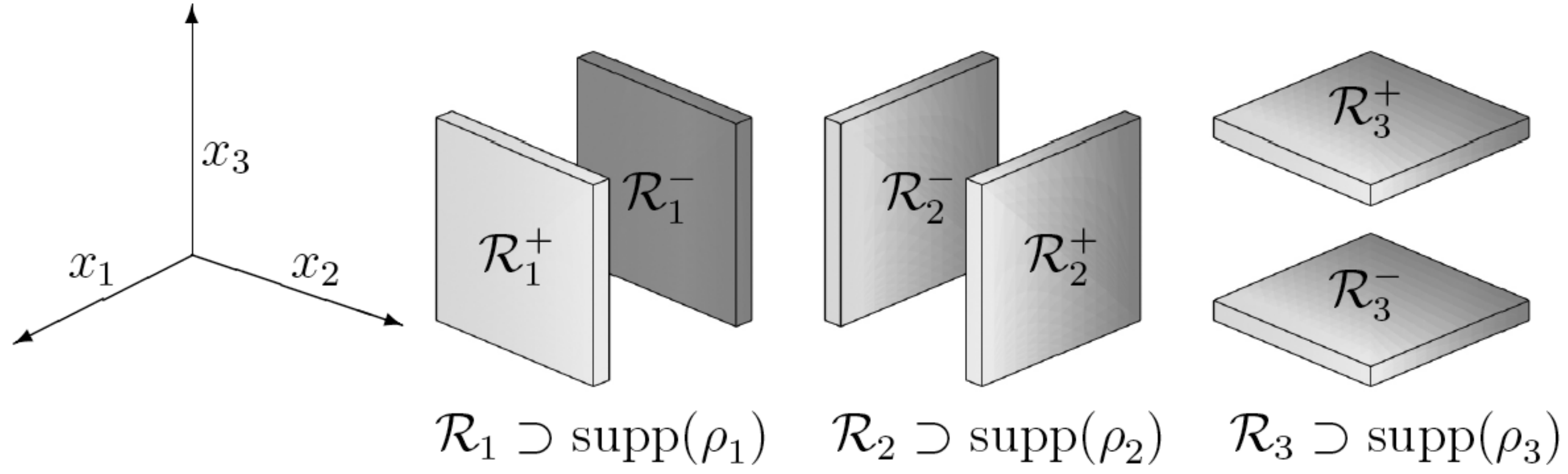}
\end{center}

Since $\mathop{\rm supp}\left( \varrho _{i}\right) $ is compact and $%
\mathcal{K}_{i}$ is open, it follows from (\ref{chaak}) that there exists $%
\tilde{C}_{1}=\tilde{C}_{1}\left( \vec{k},\mathcal{R},M_0\right) >0$ such
that 
\begin{equation*}
\varrho _{i}\left( x\right) \min_{\ell \neq i} k^{\ell }(x,z)\geq \frac{1 }{%
\tilde{C}_{1}}\varrho _{i}\left( x\right),\qquad i=1,\dots ,n,\quad \hspace{%
-0.0943pc}x\in \Omega \quad\hspace{-0.0943pc} z\in \left[ -M_{0},M_{0}\right]%
.
\end{equation*}
Since $\vec{k}$ is bounded in any compact set, it follows from Condition \ref%
{hyp1} that there exists $C_{1}=C_{1}\left( \vec{k},\mathcal{R},M_{0}\right) 
$ such that for $1\leq i\leq n,$ 
\begin{equation}
\varrho _{i}\left( x\right) \,k^{i}(x,z)\,\leq C_{1}\,\varrho _{i}\left(
x\right) \,\min_{1\leq j\leq n}k^{j}(x,z),\quad \hspace{-0.0939pc}\left(
x,z\right) \in \Omega \times \left[ -M_{0},M_{0}\right] .  \label{mink}
\end{equation}

We will often want to show that a certain term is small by applying the
one-dimensional Sobolev inequality in the $x_{1}$-variable, i.e., by
applying the estimate 
\begin{equation}
\left\Vert \varphi \right\Vert _{L^{2}\left( \mathbb{R}^{n}\right) }\,\leq
Cr^{1}\left\Vert \partial _{1}\varphi \right\Vert _{L^{2}\left( \mathbb{R}%
^{n}\right) }\,,  \label{spoinc0}
\end{equation}%
where $\varphi $ is a function with compact support in $2\mathcal{R}$ and $C$
is a universal constant. Then by (\ref{hellip}), (\ref{unifbdk1}), and the
definition in (\ref{gradk}) below, we have 
\begin{equation}
\left\Vert \varphi \right\Vert _{L^{2}\left( \mathbb{R}^{n}\right) }\,\leq
Cr^{1}\left\Vert \nabla _{\!\!\mathcal{A},w}\,\varphi \right\Vert
_{L^{2}\left( \mathbb{R}^{n}\right) }\qquad\text{if support$(\varphi)
\subset 2\mathcal{R}.$}  \label{spoinc}
\end{equation}%
The constant factor $r^{1}$ which appears here will often be chosen small to
help in absorption arguments, but it is important to observe that since $%
A\geq \left( \delta ^{1}\right) ^{-1}\geq 3\left( r^{1}\right) ^{-1}$, we
must ensure that a term to be shown small because it contains an $r^{1}$
factor \emph{does not also include a factor of positive powers of }$A$.

For simplicity, we will often restrict our calculations to the case when the
center $\tilde{x}$ of $\mathcal{R}$ is the origin.

\subsection{{Auxiliary Results\label{prelressection}}}

Given a weak solution $w\in H_{\mathcal{X}}^{1,2}\left(\Omega\right) \bigcap
L_\infty \left(\Omega \right) $ to (\ref{equation}) , we denote $\mathbf{A}%
\left( x\right) =\mathcal{A}\left( x,w\left( x\right) \right) $. It is
convenient to define the \emph{linear} operator 
\begin{equation}
\mathcal{L}_{w}=\mathop{\rm div}\mathbf{A}\left( x\right) \nabla = %
\mathop{\rm div}\mathcal{A}\left( x,w\left( x\right) \right) \nabla
,\;\;\;\;\;x\in \Omega .  \label{operatorw}
\end{equation}
Given $\varphi \in H_{\mathcal{X}}^{1,2}\left( \Omega \right) $, we{\ denote
by }$\nabla _{\!\!\sqrt{\mathcal{A}},w}\varphi \,${\ the }$\sqrt{\mathcal{A}}
${-gradient of }$\varphi, $ formally defined by 
\begin{equation}
\nabla _{\!\!\sqrt{\mathcal{A}},w}\varphi =\sqrt{\mathcal{A}\left( x,w\left(
x\right) \right) }\nabla {\varphi .}  \label{gradk}
\end{equation}
See the Appendix for a discussion of the meaning of $\nabla\varphi$ in case $%
\varphi$ is not smooth.

{We now list four useful lemmas obtained in \cite{RSaW1}.}

\begin{lemma}
{\ \label{Moser}Let $u\in \mathcal{C}^{\infty }\left( 2\mathcal{R}\right) $, 
$\psi $ be a nonnegative cutoff function supported in $2\mathcal{R},$ and $%
\beta \in \mathbb{N}$. Then 
\begin{align*}
\int_{2\mathcal{R}}\left\vert \psi \nabla _{\!\!\sqrt{\mathcal{A}}%
,w}\,u^{\beta }\right\vert ^{2}& \leq \frac{2\beta ^{2}}{2\beta -1}%
\left\vert \int_{2\mathcal{R}}\left( \psi \mathcal{L}_{w}u\right) \left(
\psi u^{2\beta -1}\right) \right\vert \\
& +\left( \frac{4\beta }{2\beta -1}\right) ^{2}\int_{2\mathcal{R}}\left\vert
\nabla _{\!\!\sqrt{\mathcal{A}},w}\,\psi \right\vert ^{2}\left\vert u^{\beta
}\right\vert ^{2},
\end{align*}%
where $\mathcal{L}_{w}$ is the linear operator (\ref{operatorw}), and $%
\nabla _{\!\!\sqrt{\mathcal{A}},w}$ is given by (\ref{gradk}). }
\end{lemma}

\begin{lemma}
\label{minkgrad} Let ${\vec{k}}$ satisfy Condition \ref{hyp1} and $\mathcal{R%
}$ be a box satisfying property (\ref{ax1c3}) in Condition \ref{hyp1}. For
any smooth function $\varphi $ and smooth cutoff function $\psi $ of the
form $\psi =\prod_{i=1}^{n}\psi _{i}\left( x_{i}\right) $, where $\eta
_{i}\leq \psi _{i}\leq \zeta _{i}$ for all $i$ (see Definition (\ref{rhos}%
)), we have 
\begin{equation*}
\left\vert \nabla _{\!\!\sqrt{\mathcal{A}},w}\psi \right\vert ^{2}\left\vert
\nabla \varphi \right\vert ^{2}\leq C_{1}\Lambda \left\vert \nabla \psi
\right\vert ^{2}\left\vert \nabla _{\!\!\sqrt{\mathcal{A}},w}\varphi
\right\vert ^{2},
\end{equation*}
with $\Lambda,$ $C_{1}=C_{1}\left( \vec{k},\mathcal{R},M_{0}\right) $ and $%
\nabla _{\!\!\sqrt{\mathcal{A}},w}$ as in (\ref{hellip}), (\ref{mink}) and (%
\ref{gradk}).
\end{lemma}

\begin{lemma}
\label{lemma-interpolation} Suppose that $\zeta $, $\chi $ are cutoff
functions as in Definition \ref{rhos}, and $\mathcal{R}$ is a rectangle with 
$2\mathcal{R} \subset \Omega $ which satisfies property (\ref{ax1c2}) of
Condition \ref{hyp1}. Then for each $q>n$, there exists $1<p<2$ such that
for all $u\in \mathcal{C}_{0}^{1}\left( 2\mathcal{R}\right) $ satisfying $%
u\preceq \chi $ (see Definition \ref{rhost}), all $\beta \in \mathbb{N}$,
and all $0<\epsilon \le 1$, 
\begin{equation*}
\int_{2\mathcal{R}}\left\vert \nabla _{\!\!\sqrt{\mathcal{A}},w}\, \zeta
\right\vert ^{2}\left\vert u^{\beta }\right\vert ^{2}\leq \varepsilon ^{-1}%
\mathcal{C} \left( n,A,q,\mathcal{K},\mathcal{R}, C_{1},\Lambda \right)
\left\Vert u^{\beta }\right\Vert _{L^{p}}^{2}+\varepsilon \int_{2\mathcal{R}%
}\left\vert \zeta \nabla _{\!\!\sqrt{\mathcal{A}},w}\,u^{\beta }\right\vert
^{2}\,,
\end{equation*}
where $\mathcal{K}=\left\Vert \nabla \chi \mathbf{A}\right\Vert _{L^{q}}$,
and $C_1$ is as in Lemma \ref{minkgrad}.
\end{lemma}

\begin{lemma}
{\ \label{secondcacc}} Suppose that $w$ is a smooth solution of (\ref%
{equation}) in $2\mathcal{R}\subset \Omega ^{\prime }$. Then for $\beta \in 
\mathbb{N}$ and any $0<\alpha \leq 1$, 
\begin{equation*}
\sum_{i,j=1}^{n}\left\vert \int_{2\mathcal{R}}\zeta ^{2}\left( \mathcal{L}
_{w}w_{ij}\right) w_{ij}^{2\beta -1}\right\vert
\end{equation*}
\begin{eqnarray*}
&\leq &2\sum_{i,j=1}^{n}\left\vert \int_{2\mathcal{R}}\left( \nabla
w_{j}\right) \cdot \mathbf{A}_{i}\nabla \zeta ^{2}w_{ij}^{2\beta
-1}\right\vert +2\sum_{i,j=1}^{n}\left\vert \int_{2\mathcal{R}}w_{i}\left(
\nabla w_{j}\right) \cdot \mathbf{A}_{z}\nabla \zeta ^{2}w_{ij}^{2\beta
-1}\right\vert \\
&&+\sum_{i,j=1}^{n}\left\vert \int_{2\mathcal{R}}w_{ij}\left( \nabla
w\right) \cdot \mathbf{A}_{z}\nabla \zeta ^{2}w_{ij}^{2\beta -1}\right\vert
+ \frac{\mathcal{C}_{0}}{\alpha }\int_{2\mathcal{R}}\zeta ^{2}\left\vert
\nabla ^{2}w\right\vert ^{2\beta } \\
&&+\alpha \sum_{i,j=1}^{n}\int_{2\mathcal{R}}\zeta ^{2}\left\vert \vec{\gamma%
}\cdot \nabla w_{ij}^{\beta }\right\vert ^{2}+\alpha \int_{2\mathcal{R}%
}\zeta ^{2}\left\vert \vec{\gamma}_{z}\cdot \nabla w\right\vert
^{2}\left\vert \nabla ^{2}w\right\vert ^{2\beta } \\
&&+2\sum_{i,j=1}^{n}\left\vert \int_{2\mathcal{R}}\zeta ^{2}\left( \left(
\partial _{j}\vec{\gamma}\right) \cdot \nabla w_{i}\right) w_{ij}^{2\beta
-1}\right\vert +2\sum_{i,j=1}^{n}\left\vert \int_{2\mathcal{R}}w_{i}\left(
\nabla w\right) \cdot \mathbf{A}_{jz}\nabla \zeta ^{2}w_{ij}^{2\beta
-1}\right\vert \\
&&+\sum_{i,j=1}^{n}\left\vert \int_{2\mathcal{R}}w_{i}w_{j}\left( \nabla
w\right) \cdot \mathbf{A}_{zz}\nabla \zeta ^{2}w_{ij}^{2\beta -1}\right\vert
+\mathcal{C}_{0}\int_{2\mathcal{R}}\zeta ^{2}\left\vert \nabla w\right\vert
^{6\beta }+\mathcal{C}_{0},
\end{eqnarray*}
{where $\vec{\gamma}=\vec{\gamma}\left( x,w\right) $, $f=f\left( x,w\right)$%
, ${\mathbf{A}}=\mathcal{A}\left( x,w\right) $, $\mathbf{A}_{i}=\mathcal{A}%
_{i}\left( x,w\right) $, etc., and 
\begin{equation*}
\mathcal{C}_{0}=C\left\{ \left\Vert \mathcal{A}\right\Vert _{\mathcal{C}%
^{3}\left( \tilde{\Gamma}\right) }+\left\Vert f\right\Vert _{\mathcal{C}
^{2}\left( \tilde{\Gamma}\right) }+\left\Vert \vec{\gamma}\right\Vert _{%
\mathcal{C}^{2}\left( \tilde{\Gamma}\right) }+1\right\} .
\end{equation*}
}
\end{lemma}

\begin{remark}
Lemma \ref{secondcacc} is a slightly different version of Lemma 5.6 in \cite%
{RSaW1} and readily follows from its proof. Indeed the fourth term on the
right side of the conclusion of the original lemma is replaced, via
straightforward changes in the proof, by the sixth and seventh terms on the
right side above.
\end{remark}

The following result is used in the proof of Theorem \ref{DP}.

\begin{lemma}
\label{concave-majorant}Given $u$ continuous on $\left[ 0,1\right] $, there
exists $w\in \mathcal{C}^0\left( \left[ 0,1\right] \right) \bigcap \mathcal{C%
}^{2}\left( \left( 0,1\right] \right) $ such that $w$ is concave, strictly
increasing, $w\left( x\right) \geq u\left( x\right) $ for all $x\in \left[
0,1\right] $, and $w\left( 0\right) =u\left( 0\right) $.
\end{lemma}

\begin{proof}
Let $\tilde{u}\left( x\right) =\max_{t\in \left[ 0,x\right] } u\left(
t\right) $ for $x\in \left[ 0,1\right] $, and $\tilde{u}\left( x\right)
=\max_{t\in \left[ 0,1\right] }u\left( t\right) $ for $x>1$. Since $%
\max_{t\in \left[ 0,x\right] }u\left( t\right) $ is nondecreasing, $\tilde{u}%
\left( x\right) $ is nondecreasing in $\left[ 0,\infty \right) $, $\tilde{u}%
\left( x\right) \geq u\left( x\right) $ in $\left[ 0,1\right] $, and $\tilde{%
u}\left( 0\right) =u\left( 0\right) $.

Next, for a smooth nonnegative function $\eta $ with compact support in $%
\left[ -1,1\right] $ and $\int \eta =1$, set $\eta _{\varepsilon }\left(
x\right) =\frac{1}{\varepsilon }\eta \left( \frac{x}{\varepsilon }\right) $
and let $v\left( x\right) =\tilde{u}\ast \eta _{\frac{x}{2}}\left( 2x\right) 
$ for $x>0$, and $v\left( 0\right) =u\left( 0\right) $. Then $v\left(
x\right) \geq \tilde{u}\left( x\right) $ for all $x,$ and $v\in \mathcal{C}%
^0\left( \left[ 0,\infty \right) \right) \bigcap \mathcal{C}^{\infty }\left(
\left( 0,\infty \right) \right) $.

Taking now $\tilde{v}\left( x\right) =\max_{t\in \left[ 0,x\right] }v\left(
t\right) +x$, we have that $\tilde{v}\in \mathcal{C}^0\left( \left[ 0,\infty
\right) \right) \bigcap \mathcal{C}^{\infty }\left( \left( 0,\infty \right)
\right) $, $\tilde{v}\left( x\right) \geq v\left( x\right) \geq \tilde{u}%
\left( x\right) $ for all $x$, $\tilde{v}\left( 0\right) =u\left( 0\right) $%
, and $\tilde{v}$ is strictly increasing.

By the fundamental theorem of calculus,%
\begin{equation*}
\tilde{v}\left( x\right) =\tilde{v}\left( 1\right) -\int_{x}^{1}\tilde{v}%
^{\prime }\left( t\right) ~dt=\tilde{v}\left( 1\right) -\left( 1-x\right) 
\tilde{v}^{\prime }\left( 1\right) +\int_{x}^{1}\int_{t}^{1}\tilde{v}%
^{\prime \prime }\left( s\right) ~ds~dt
\end{equation*}%
for all $x\in \left( 0,\infty \right) $. Let $\left[ \tilde{v}^{\prime
\prime }\left( s\right) \right] ^{+}=\max \left\{ \tilde{v}^{\prime \prime
}\left( s\right) ,0\right\} $ and $\left[ \tilde{v}^{\prime \prime }\left(
s\right) \right] ^{-}=\max \left\{ -\tilde{v}^{\prime \prime }\left(
s\right) ,0\right\} $. Then $\left[ \tilde{v}^{\prime \prime }\left(
s\right) \right] ^{\pm }$ are continuous in $\left( 0,\infty \right) $ and $%
\tilde{v}^{\prime \prime }\left( s\right) =\left[ \tilde{v}^{\prime \prime
}\left( s\right) \right] ^{+}-\left[ \tilde{v}^{\prime \prime }\left(
s\right) \right] ^{-}$. Since%
\begin{equation*}
\max_{x\in \left[ 0,1\right] }\left\vert \int_{x}^{1}\int_{t}^{1}\tilde{v}%
^{\prime \prime }\left( s\right) ~ds~dt\right\vert =\max_{x\in \left[ 0,1%
\right] }\left\vert \tilde{v}\left( x\right) -\tilde{v}\left( 1\right)
+\left( 1-x\right) \tilde{v}^{\prime }\left( 1\right) \right\vert <\infty ,
\end{equation*}%
it follows that the functions $w_{+}$ and $w_{-}$ defined by 
\begin{equation*}
w_{+}\left( x\right) =\int_{x}^{1}\int_{t}^{1}\left[ \tilde{v}^{\prime
\prime }\left( s\right) \right] ^{+}~ds~dt\qquad \text{and}\qquad
w_{-}\left( x\right) =\int_{x}^{1}\int_{t}^{1}\left[ \tilde{v}^{\prime
\prime }\left( s\right) \right] ^{-}~ds~dt
\end{equation*}%
are finite and belong to $\mathcal{C}^{2}\left( 0,1\right] $. Also note that 
$\tilde{v}\left( x\right) =\tilde{v}\left( 1\right) -\left( 1-x\right) 
\tilde{v}^{\prime }\left( 1\right) +w_{+}\left( x\right) -w_{-}\left(
x\right) $ . Moreover, since%
\begin{equation*}
\left( w_{\pm }\left( x\right) \right) ^{\prime \prime }=\left[ \tilde{v}%
^{\prime \prime }\left( x\right) \right] ^{\pm }\geq 0,
\end{equation*}%
it follows that $w_{\pm }$ are convex in $\left[ 0,1\right] $. In
particular, 
\begin{equation*}
w_{+}\left( x\right) \leq \left( 1-x\right) ~w_{+}\left( 0\right)
+x~w_{+}\left( 1\right) ,\qquad x\in \left[ 0,1\right] .
\end{equation*}%
We claim that the function 
\begin{equation*}
w\left( x\right) =\tilde{v}\left( 1\right) -\left( 1-x\right) \tilde{v}%
^{\prime }\left( 1\right) +\left( 1-x\right) ~w_{+}\left( 0\right)
+x~w_{+}\left( 1\right) -w_{-}\left( x\right)
\end{equation*}%
satisfies all the properties stated in the lemma. Indeed, it is clear that $%
w $ is continuous in $\left[ 0,1\right] $, $\mathcal{C}^{2}$ in $\left( 0,1%
\right] $, and $w\left( 0\right) =u\left( 0\right)$ since $w\left( 0\right)
= \tilde{v}(1) -\tilde{v}(0) + w_{+}(0) - w_{-}(0) = \tilde{v}(0) = u(0)$.
From the last inequality,%
\begin{equation*}
w\left( x\right) \geq \tilde{v}\left( 1\right) -\left( 1-x\right) \tilde{v}%
^{\prime }\left( 1\right) +w_{+}\left( x\right) -w_{-}\left( x\right) =%
\tilde{v}\left( x\right) \geq \tilde{u}\left( x\right) \geq u\left( x\right)
.
\end{equation*}%
Finally, since%
\begin{equation*}
w^{\prime \prime }\left( x\right) =-\left( w_{-}\left( x\right) \right)
^{\prime \prime }\leq 0
\end{equation*}%
we have that $w$ is concave, as required.
\end{proof}

\section{Proof of the {a priori estimates\label{hypo}}}

\subsection{{$L^{p}$ estimates for the gradient\label{apriori}}}

{In this section we prove higher integrability properties of $\nabla w$
(Theorem \ref{extraintegrability}) by using the extra hypotheses in
Condition \ref{hyp2}.}

\begin{lemma}
{\label{extraprev}Under the hypotheses of Theorem \ref{quasi}, for all
integers }$\beta \geq 1$, {every smooth solution $w$ of (\ref{equation}) in $%
\Omega $ satisfies } 
\begin{eqnarray*}
&&\sum_{j=1}^{n}\int_{2\mathcal{R}}\left\vert \nabla _{\!\!\sqrt{\mathcal{A}}
,w}\zeta \right\vert ^{2}\left\vert w_{j}\right\vert ^{2\beta } \\
&\leq &CC_{1}\Lambda \left( \left( A^{4}+A^2B^{2}\right) M_{0}^{2}+
A^2\left\Vert f\right\Vert _{L^{\infty }\left( \tilde{\Gamma}\right)
}\right) \sum_{j=1}^{n}\int_{2\mathcal{R}}\xi ^{2} w_{j}^{2\beta -2} \\
&&+\mathcal{C}\left( C_{1},\Lambda ,M_{0}\right) \sum_{j=1}^{n}\int_{2 
\mathcal{R}}\left\vert \nabla \zeta \right\vert ^{2}\left\vert \nabla _{\!\! 
\sqrt{\mathcal{A}},w}w_{j}^{\beta -1}\right\vert ^{2}.
\end{eqnarray*}
{Here $C_1$ as in Lemma \ref{minkgrad}, $\mathcal{R}$ is any box such that }$%
3\mathcal{R} \subset \Omega $ and $\mathcal{R}$ satisfies property (\ref%
{ax1c3}) in the nondegeneracy Condition{\ \ref{hyp1}, }$M_{0}=\left\Vert
w\right\Vert _{L^{\infty }\left( 2 \mathcal{R}\right) }${, }$\tilde{\Gamma}=2%
\mathcal{R}\times \left[ -M_{0},M_{0}\right] $, {and $B=B_{\gamma }\left( 2%
\mathcal{R},M_{0}\right) $ is as in Definition \ref{subunitt}. }
\end{lemma}

{\ \vspace*{-0.2in}%
\proof%
By the product rule, 
\begin{eqnarray*}
\left\vert w_{j}\right\vert ^{2\beta } &=&\left( \partial _{j}\big(%
ww_{j}^{\beta -1}\big)-w\partial _{j}\big(w_{j}^{\beta -1}\big)\right) ^{2}
\\
&\leq &2\left( \partial _{j}\big(ww_{j}^{\beta -1}\big)\right)
^{2}+2w^{2}\left( \partial _{j}\big(w_{j}^{\beta -1}\big)\right) ^{2}.
\end{eqnarray*}%
Then, using that $w$ is bounded by }$M_{0}${\ and applying Lemma \ref%
{minkgrad} gives 
\begin{eqnarray}
&&\sum_{j=1}^{n}\int_{2\mathcal{R}}\left\vert \nabla _{\!\!\sqrt{\mathcal{A}}%
,w}\zeta \right\vert ^{2}\left\vert w_{j}\right\vert ^{2\beta }  \notag \\
&\leq &2\sum_{j=1}^{n}\int_{2\mathcal{R}}\left\vert \nabla _{\!\!\sqrt{%
\mathcal{A}},w}\zeta \right\vert ^{2}\left\vert \partial _{j}\big(%
ww_{j}^{\beta -1}\big)\right\vert ^{2}+2M_{0}^{2}\sum_{j=1}^{n}\int_{2%
\mathcal{R}}\left\vert \nabla _{\!\!\sqrt{\mathcal{A}},w}\zeta \right\vert
^{2}\left\vert \partial _{j}\big(w_{j}^{\beta -1}\big)\right\vert ^{2} 
\notag \\
&\leq &CC_{1}\Lambda \sum_{j=1}^{n}\int_{2\mathcal{R}}\left\vert \nabla
\zeta \right\vert ^{2}\left\vert \nabla _{\!\!\sqrt{\mathcal{A}},w}\big(%
ww_{j}^{\beta -1}\big)\right\vert ^{2}  \notag \\
&&+CC_{1}\Lambda M_{0}^{2}\sum_{j=1}^{n}\int_{2\mathcal{R}}\left\vert \nabla
\zeta \right\vert ^{2}\left\vert \nabla _{\!\!\sqrt{\mathcal{A}}%
,w}w_{j}^{\beta -1}\right\vert ^{2}  \notag \\
&\leq &CC_{1}\Lambda \sum_{j=1}^{n}\int_{2\mathcal{R}}\left\vert \nabla
\zeta \right\vert ^{2}\left\vert \nabla _{\!\!\sqrt{\mathcal{A}}%
,w}w\right\vert ^{2}w_{j}^{2\beta -2}  \label{xint003} \\
&&+\mathcal{C}\left( C_{1},\Lambda ,M_{0}\right) \sum_{j=1}^{n}\int_{2%
\mathcal{R}}\left\vert \nabla \zeta \right\vert ^{2}\left\vert \nabla _{\!\!%
\sqrt{\mathcal{A}},w}w_{j}^{\beta -1}\right\vert ^{2}.  \notag
\end{eqnarray}%
Integrating by parts in the first term on the right, and using the fact that 
$w\,$is a solution of (\ref{equation}), we obtain} 
\begin{equation*}
\sum_{j=1}^{n}\int_{2\mathcal{R}}\left\vert \nabla \zeta \right\vert
^{2}\left\vert \nabla _{\!\!\sqrt{\mathcal{A}},w}w\right\vert
^{2}w_{j}^{2\beta -2}=-\sum_{j=1}^{n}\int_{2\mathcal{R}}w\mathop{\rm div}%
\left[ \left\vert \nabla \zeta \right\vert ^{2}w_{j}^{2\beta -2}\mathcal{A}%
(x,w)\nabla w\right]
\end{equation*}%
\begin{equation*}
=-\sum_{j=1}^{n}\int_{2\mathcal{R}}ww_{j}^{2\beta -2}\left( \nabla
\left\vert \nabla \zeta \right\vert ^{2}\right) \cdot \mathcal{A}\left(
x,w\right) \nabla w
\end{equation*}%
\begin{equation*}
-\sum_{j=1}^{n}\int_{2\mathcal{R}}w\left\vert \nabla \zeta \right\vert
^{2}\left( \nabla \big(w_{j}^{2\beta -2}\big)\right) \cdot \mathcal{A}\left(
x,w\right) \nabla w
\end{equation*}%
\begin{equation*}
+\sum_{j=1}^{n}\int_{2\mathcal{R}}w\left\vert \nabla \zeta \right\vert
^{2}w_{j}^{2\beta -2}\big(\vec{\gamma}\left( x,w\right) \cdot \nabla
w+f\left( x,w\right) \big)
\end{equation*}%
\begin{equation}
=I+II+III.  \label{xint004}
\end{equation}%
In $I$, using the estimate 
\begin{equation*}
\left\vert \left( \nabla \left\vert \nabla \zeta \right\vert ^{2}\right)
\cdot \mathcal{A}\left( x,w\right) \nabla w\right\vert \leq \left\vert
\nabla _{\!\!\sqrt{\mathcal{A}},w}\left( \left\vert \nabla \zeta \right\vert
^{2}\right) \right\vert \,\left\vert \nabla _{\!\!\sqrt{\mathcal{A}}%
,w}w\right\vert \leq CA^{2}\left\vert \nabla \zeta \right\vert \left\vert
\nabla _{\!\!\sqrt{\mathcal{A}},w}w\right\vert ,
\end{equation*}%
we have 
\begin{eqnarray}
\left\vert I\right\vert &\leq &M_{0}\sum_{j=1}^{n}\int_{2\mathcal{R}%
}\left\vert w_{j}\right\vert ^{2\beta -2}\left\vert \left( \nabla \left\vert
\nabla \zeta \right\vert ^{2}\right) \cdot \mathcal{A}\left( x,w\right)
\nabla w\right\vert  \notag \\
&\leq &\dfrac{1}{4}\sum_{j=1}^{n}\int_{2\mathcal{R}}\left\vert
w_{j}\right\vert ^{2\beta -2}\left\vert \nabla \zeta \right\vert
^{2}\left\vert \nabla _{\!\!\sqrt{\mathcal{A}},w}w\right\vert
^{2}+CA^{4}M_{0}^{2}\sum_{j=1}^{n}\int_{2\mathcal{R}}\xi ^{2}\left\vert
w_{j}\right\vert ^{2\beta -2}.  \label{xint005}
\end{eqnarray}%
Using the identity $\nabla w_{j}^{2\beta -2}=2w_{j}^{\beta -1}\nabla
w_{j}^{\beta -1}$ in $II$, we obtain 
\begin{eqnarray}
\left\vert II\right\vert &\leq &2M_{0}\sum_{j=1}^{n}\int_{2\mathcal{R}%
}\left\vert \nabla \zeta \right\vert ^{2}\left\vert w_{j}\right\vert ^{\beta
-1}\left\vert \left( \nabla w_{j}^{\beta -1}\right) \cdot \mathcal{A}\left(
x,w\right) \nabla w\right\vert  \notag \\
&\leq &\frac{1}{4}\sum_{j=1}^{n}\int_{2\mathcal{R}}\left\vert \nabla \zeta
\right\vert ^{2}\left\vert \nabla _{\!\!\sqrt{\mathcal{A}},w}w\right\vert
^{2}w_{j}^{2\beta -2}+CM_{0}^{2}\sum_{j=1}^{n}\int_{2\mathcal{R}}\left\vert
\nabla \zeta \right\vert ^{2}\left\vert \nabla _{\!\!\sqrt{\mathcal{A}}%
,w}w_{j}^{\beta -1}\right\vert ^{2}.  \label{xint0055}
\end{eqnarray}%
{Finally, since }$\vec{\gamma}$ is subunit with respect to $\mathcal{A}$ in $%
\Gamma =\Omega \times \mathbb{R}$, we have $\left\vert \vec{\gamma}\left(
x,w\right) \cdot \nabla w\right\vert \leq B\left\vert \nabla _{\!\!\sqrt{%
\mathcal{A}},w}w\right\vert $ for all $x\in 2\mathcal{R}$, where {$B=B\left(
2\mathcal{R},M_{0}\right) $. Then} 
\begin{eqnarray}
\left\vert III\right\vert &\leq &M_{0}\sum_{j=1}^{n}\int_{2\mathcal{R}%
}\left\vert \nabla \zeta \right\vert ^{2}w_{j}^{2\beta -2}\big(\left\vert 
\vec{\gamma}\left( x,w\right) \cdot \nabla w\right\vert +\left\vert f\left(
x,w\right) \right\vert \big)  \notag \\
&\leq &\dfrac{1}{4}\sum_{j=1}^{n}\int_{2\mathcal{R}}\left\vert \nabla \zeta
\right\vert ^{2}\left\vert \nabla _{\!\!\sqrt{\mathcal{A}},w}w\right\vert
^{2}w_{j}^{2\beta -2}  \label{xint006} \\
&&+\left( CB^{2}M_{0}^{2}+\left\Vert f\right\Vert _{L^{\infty }\left( \tilde{%
\Gamma}\right) }\right) \sum_{j=1}^{n}\int_{2\mathcal{R}}\left\vert \nabla
\zeta \right\vert ^{2}w_{j}^{2\beta -2}.  \notag
\end{eqnarray}%
{\ \ Combining (\ref{xint004}), (\ref{xint005}), (\ref{xint0055}) and ( \ref%
{xint006}), and absorbing into the left yields} 
\begin{eqnarray*}
&&\sum_{j=1}^{n}\int_{2\mathcal{R}}\left\vert \nabla \zeta \right\vert
^{2}\left\vert \nabla _{\!\!\sqrt{\mathcal{A}},w}w\right\vert
^{2}w_{j}^{2\beta -2} \\
&\leq &CM_{0}^{2}\sum_{j=1}^{n}\int_{2\mathcal{R}}\left\vert \nabla \zeta
\right\vert ^{2}\left\vert \nabla _{\!\!\sqrt{\mathcal{A}},w}w_{j}^{\beta
-1}\right\vert ^{2} \\
&&+C\left( \left( A^{4}+A^{2}B^{2}\right) M_{0}^{2}+A^{2}\left\Vert
f\right\Vert _{L^{\infty }\left( \tilde{\Gamma}\right) }\right)
\sum_{j=1}^{n}\int_{2\mathcal{R}}\xi ^{2}w_{j}^{2\beta -2}.
\end{eqnarray*}%
Using this estimate in the first term on the right of ({\ref{xint003}}){\
finishes the proof of Lemma \ref{extraprev}. 
\endproof%
}

\begin{lemma}
{\ \label{extraprev2}Under the hypothesis of Theorem \ref{quasi}, if $w$ is
a smooth solution of (\ref{equation}) in $\Omega $, then for any $\beta \in 
\mathbb{N}$ and any box $\mathcal{R}\subset \Omega $ satisfying property (%
\ref{ax1c3}) in }the nondegeneracy Condition{\ \ref{hyp1}, 
\begin{eqnarray*}
&&\sum_{j=1}^n \int_{2\mathcal{R}}\zeta ^{2}\left\vert w_{j}\right\vert
^{2\beta }\left\vert \nabla _{\!\!\mathcal{A},w}w\right\vert ^{2} \\
&\leq &CB^{2}\sum_{j=1}^{n}\int_{2\mathcal{R}}\zeta ^{2}\left\vert \nabla
_{\!\!\mathcal{A},w}w_{j}^{\beta }\right\vert ^{2}+CB^{2} \int_{2\mathcal{R}%
}\left\vert \nabla _{\!\!\sqrt{\mathcal{A}}, w}\zeta \right\vert
^{2}\left\vert \nabla w\right\vert ^{2\beta }.
\end{eqnarray*}
Here $B=B_{\gamma }\left( 2\mathcal{R},M_{0}\right) $ is as in Definition %
\ref{subunitt} and $C$ depends on $M_0$ and $||f||_{L^\infty(\tilde{\Gamma}%
)} $. }
\end{lemma}

{\vspace*{-0.2in}%
\proof%
Integrating by parts, we have 
\begin{eqnarray}
&&\int_{2\mathcal{R}}\zeta ^{2}w_{j}^{2\beta }\left\vert \nabla _{\!\!%
\mathcal{A},w}w\right\vert ^{2}  \notag \\
&=&-\int_{2\mathcal{R}}w\mathop{\rm div}\left[ \zeta ^{2}w_{j}^{2\beta }%
\mathcal{A}\left( x,w\right) \nabla w\right]  \notag \\
&=&-2\int_{2\mathcal{R}}w\zeta w_{j}^{2\beta }\left( \nabla \zeta \right)
\cdot \mathcal{A}\left( x,w\right) \nabla w  \notag \\
&&-2\int_{2\mathcal{R}}w\zeta ^{2}w_{j}^{\beta }\left( \nabla w_{j}^{\beta
}\right) \cdot \mathcal{A}\left( x,w\right) \nabla w  \notag \\
&&+\int_{2\mathcal{R}}w\zeta ^{2}w_{j}^{2\beta }\big(\vec{\gamma}(x,w)\cdot
\nabla w+f(x,w)\big)  \notag \\
&=&I+II+III.  \label{xxtp200}
\end{eqnarray}%
By Schwarz's' inequality and since }$\vec{\gamma}${\ }is of subunit type
with respect to $\mathcal{A},${\ 
\begin{eqnarray*}
\left\vert I\right\vert &\leq &CM_{0}^{2}\int_{2\mathcal{R}}\left\vert
\nabla _{\!\!\sqrt{\mathcal{A}},w}\zeta \right\vert ^{2}w_{j}^{2\beta }+%
\frac{1}{6}\int_{2\mathcal{R}}\zeta ^{2}\left\vert \nabla _{\!\!\sqrt{%
\mathcal{A}},w}w\right\vert ^{2}w_{j}^{2\beta } \\
\left\vert II\right\vert &\leq &CM_{0}^{2}\int_{2\mathcal{R}}\zeta
^{2}\left\vert \nabla _{\!\!\sqrt{\mathcal{A}},w}w_{j}^{\beta }\right\vert
^{2}+\frac{1}{6}\int_{2\mathcal{R}}\zeta ^{2}\left\vert \nabla _{\!\!\sqrt{%
\mathcal{A}},w}w\right\vert ^{2}w_{j}^{2\beta } \\
\left\vert III\right\vert &\leq &CB^{2}M_{0}^{2}\left( 1+||f||_{L^{\infty }(%
\tilde{\Gamma})}^{2}\right) \int_{2\mathcal{R}}\zeta ^{2}w_{j}^{2\beta }+%
\frac{1}{6}\int_{2\mathcal{R}}\zeta ^{2}\left\vert \nabla _{\!\!\sqrt{%
\mathcal{A}},w}w\right\vert ^{2}w_{j}^{2\beta }.
\end{eqnarray*}%
Applying these estimates to (\ref{xxtp200}) and absorbing into the left
gives 
\begin{eqnarray*}
&&\sum_{j=1}^{n}\int_{2\mathcal{R}}\zeta ^{2}w_{j}^{2\beta }\left\vert
\nabla _{\!\!\sqrt{\mathcal{A}},w}w\right\vert ^{2} \\
&\leq &C\sum_{j=1}^{n}\int_{2\mathcal{R}}\zeta ^{2}\left\vert \nabla _{\!\!%
\mathcal{A},w}w_{j}^{\beta }\right\vert ^{2}+C\sum_{j=1}^{n}\int_{2\mathcal{R%
}}\left\vert \nabla _{\!\!\sqrt{\mathcal{A}},w}\zeta \right\vert
^{2}w_{j}^{2\beta }+CB^{2}\sum_{j=1}^{n}\int_{2\mathcal{R}}\zeta
^{2}w_{j}^{2\beta }
\end{eqnarray*}%
with $C$ depending on $M_{0}$ and $||f||_{L^{\infty }(\tilde{\Gamma})}$. To
obtain the conclusion of the lemma, we apply the Sobolev inequality (\ref%
{spoinc}) to the last term on the right and note that 
\begin{equation*}
CB^{2}\int_{2\mathcal{R}}\left\vert \nabla _{\!\!\sqrt{\mathcal{A}},w}\big(%
\zeta w_{j}^{\beta }\big)\right\vert ^{2}
\end{equation*}%
is bounded by the sum of the first two terms on the right. 
\endproof%
}

\begin{theorem}
{\ \label{extraintegrability}Under the hypothesis of Theorem \ref{apriorial}
, if $w$ is a smooth solution of (\ref{equation}) in $\Omega $, then for all
integers $\beta \geq 1$ and every open }$\Omega ^{\prime }$ with $\Omega
^{\prime }\Subset \Omega ${, there exists a positive constant 
\begin{equation*}
\mathcal{C}_{\beta }=\mathcal{C}_{\beta }\left( M_0,n,B,\Lambda ,\vec{k}
,\left\Vert f\right\Vert _{\mathcal{C}^{1}\left( \Gamma ^{\prime }\right)
},\left\Vert \vec{\gamma}\right\Vert _{\mathcal{C}^{1}\left( \Gamma ^{\prime
}\right) }, \Omega,\mathop{\rm dist}\left( \Omega ^{\prime },\partial \Omega
\right) \right)
\end{equation*}
such that%
\begin{equation*}
\sum_{j=1}^{n}\int_{\Omega ^{\prime }} w_{j}^{2\beta
}+\sum_{j=1}^{n}\int_{\Omega ^{\prime }}\left\vert \nabla _{\!\!\sqrt{%
\mathcal{A}} ,w}w_{j}^{\beta }\right\vert ^{2}\leq \mathcal{C}_{\beta }.
\end{equation*}
Here }$M_{0}=\left\Vert w\right\Vert _{L^{\infty }\left( \Omega ^{\prime
}\right) }$, $B$ denotes the constants $B_{\mathcal{A}}\left( \Omega
^{\prime },M_{0}\right) ,B_{\gamma }\left( \Omega ^{\prime },M_{0}\right)$
in (\ref{xtra}) and (\ref{gxtra}), and $\Gamma ^{\prime }=\Omega ^{\prime
}\times \left[ -M_{0},M_{0}\right]$.
\end{theorem}

{\vspace*{-0.2in}%
\proof%
Let $\mathcal{R}$ be a box satisfying property (\ref{ax1c3}) in }the
nondegeneracy Condition{\ \ref{hyp1} and such that $2\mathcal{R}\Subset
\Omega $. From Lemma \ref{Moser}, 
\begin{eqnarray}
&&\sum_{j=1}^{n}\int_{2\mathcal{R}}\left\vert \zeta \nabla _{\!\!\sqrt{%
\mathcal{A}},w}\,w_{j}^{\beta }\right\vert ^{2}  \notag \\
&\leq &C\beta \sum_{j=1}^{n}\left\vert \int_{2\mathcal{R}}\zeta ^{2}\left( 
\mathcal{L}_{w}w_{j}\right) \left( w_{j}^{2\beta -1}\right) \right\vert
+C\sum_{j=1}^{n}\int_{2\mathcal{R}}\left\vert \nabla _{\!\!\sqrt{\mathcal{A}}%
,w}\,\zeta \right\vert ^{2}w_{j}^{2\beta }.  \label{mmoo}
\end{eqnarray}%
Differentiating (\ref{equation}) with respect to the $j^{th}\,$variable, we
obtain 
\begin{equation}
-\mathcal{L}_{w}w_{j}=\mathop{\rm div}\left\{ \partial _{j}\mathbf{A}\left(
x\right) \right\} \nabla w+\left( \partial _{j}\vec{\gamma}\right) \cdot
\nabla w+\vec{\gamma}\cdot \nabla w_{j}+\partial _{j}f  \label{nonhom}
\end{equation}%
for all $j$, where $\mathcal{L}_{w}$ is the linear operator (\ref{operatorw}%
), $\vec{\gamma}=\vec{\gamma}\left( x,w\right) $ and $\partial
_{j}f=\partial \lbrack f(x,w)]$. Hence 
\begin{eqnarray*}
&&\sum_{j=1}^{n}\left\vert \int_{2\mathcal{R}}\zeta ^{2}\left( \mathcal{L}%
_{w}w_{j}\right) \left( w_{j}^{2\beta -1}\right) \right\vert \\
&=&\sum_{j=1}^{n}\left\vert \int_{2\mathcal{R}}\zeta ^{2}%
{\Big (}%
\mathop{\rm div}\left\{ \partial _{j}\mathbf{A}\left( x\right) \right\}
\nabla w+\left( \partial _{j}\vec{\gamma}\right) \cdot \nabla w+\vec{\gamma}%
\cdot \nabla w_{j}+\left( \partial _{j}f\right) 
{\Big )}%
\left( w_{j}^{2\beta -1}\right) \right\vert \\
&\leq &\sum_{j=1}^{n}\left\vert \int_{2\mathcal{R}}\zeta ^{2}%
{\Big (}%
\mathop{\rm div}\left\{ \partial _{j}\mathbf{A}\left( x\right) \right\}
\nabla w%
{\Big )}%
\left( w_{j}^{2\beta -1}\right) \right\vert \\
&&+\sum_{j=1}^{n}\left\vert \int_{2\mathcal{R}}\zeta ^{2}%
{\Big (}%
\left( \partial _{j}\vec{\gamma}\right) \cdot \nabla w+\left( \partial
_{j}f\right) 
{\Big )}%
\left( w_{j}^{2\beta -1}\right) \right\vert +\sum_{j=1}^{n}\left\vert
\sum_{i=1}^{n}\int_{2\mathcal{R}}\zeta ^{2}\gamma ^{i}w_{ij}w_{j}^{2\beta
-1}\right\vert
\end{eqnarray*}%
\begin{equation}
=I+II+III.  \label{extint00}
\end{equation}%
From (\ref{wirtm}), (\ref{xtra}) and the inequality 
\begin{equation*}
\sqrt{k^{\ast }\left( x,w\right) }\left\vert \partial _{j}u\right\vert \leq
\left\vert \nabla _{\!\!\sqrt{\mathcal{A}},w}u\right\vert ,\qquad 1\leq
j\leq n,
\end{equation*}%
where $u$ is any smooth function, we get (with $B_{\mathcal{A}}$ as in (\ref%
{xtra})) 
\begin{eqnarray}
\left\vert \left\{ \partial _{j}\mathbf{A}\left( x\right) \right\} \nabla
u\right\vert &=&\left\vert \left\{ \mathcal{A}_{j}+w_{j}\mathcal{A}%
_{z}\right\} \nabla u\right\vert  \notag \\
&\leq &CB_{\mathcal{A}}\left\vert \nabla _{\!\!\sqrt{\mathcal{A}}%
,w}u\right\vert +CB_{\mathcal{A}}\sqrt{k^{\ast }\left( x,w\right) }%
\left\vert w_{j}\right\vert \left\vert \nabla _{\!\!\sqrt{\mathcal{A}}%
,w}u\right\vert  \notag \\
&\leq &CB_{\mathcal{A}}\left\vert \nabla _{\!\!\sqrt{\mathcal{A}}%
,w}u\right\vert +CB_{\mathcal{A}}\left\vert \nabla _{\!\!\sqrt{\mathcal{A}}%
,w}w\right\vert \left\vert \nabla _{\!\!\sqrt{\mathcal{A}},w}u\right\vert .
\label{wirtX}
\end{eqnarray}%
Writing 
\begin{equation*}
\nabla \left( \zeta ^{2}w_{j}^{2\beta -1}\right) =\frac{2\beta -1}{\beta }%
\zeta ^{2}w_{j}^{\beta -1}\nabla w_{j}^{\beta }+w_{j}^{2\beta -1}\nabla
\zeta ^{2},
\end{equation*}%
integrating by parts, and using $\frac{2\beta -1}{\beta }\leq 2$ and (\ref%
{wirtX} ), we obtain 
\begin{eqnarray}
I &=&\sum_{j=1}^{n}\left\vert \int_{2\mathcal{R}}\left( \nabla w\right)
\cdot \left\{ \partial _{j}\mathbf{A}\left( x\right) \right\} \nabla \left(
\zeta ^{2}w_{j}^{2\beta -1}\right) \right\vert  \notag \\
&\leq &C\sum_{j=1}^{n}\left\vert \int_{2\mathcal{R}}\zeta ^{2}w_{j}^{\beta
-1}\left( \nabla w\right) \cdot \left\{ \partial _{j}\mathbf{A}\left(
x\right) \right\} \nabla w_{j}^{\beta }\right\vert  \notag \\
&&+\sum_{j=1}^{n}\left\vert \int_{2\mathcal{R}}w_{j}^{2\beta -1}\left(
\nabla w\right) \cdot \left\{ \partial _{j}\mathbf{A}\left( x\right)
\right\} \nabla \zeta ^{2}\right\vert  \notag \\
&\leq &CB_{\mathcal{A}}\sum_{j=1}^{n}\int_{2\mathcal{R}}\zeta ^{2}\left\vert
w_{j}\right\vert ^{\beta -1}\left\vert \nabla w\right\vert \left\vert \nabla
_{\!\!\sqrt{\mathcal{A}},w}w_{j}^{\beta }\right\vert \left\{ 1+\left\vert
\nabla _{\!\!\sqrt{\mathcal{A}},w}w\right\vert \right\}  \notag \\
&&+CB_{\mathcal{A}}\sum_{j=1}^{n}\int_{2\mathcal{R}}\zeta \left\vert
w_{j}\right\vert ^{2\beta -1}\left\vert \nabla w\right\vert \left\vert
\nabla _{\!\!\sqrt{\mathcal{A}},w}\zeta \right\vert \left\{ 1+\left\vert
\nabla _{\!\!\sqrt{\mathcal{A}},w}w\right\vert \right\}  \notag \\
&\leq &\alpha \sum_{j=1}^{n}\int_{2\mathcal{R}}\zeta ^{2}\left\vert \nabla
_{\!\!\sqrt{\mathcal{A}},w}w_{j}^{\beta }\right\vert ^{2}+C\int_{2\mathcal{R}%
}\left\vert \nabla w\right\vert ^{2\beta }\left\vert \nabla _{\!\!\sqrt{%
\mathcal{A}},w}\zeta \right\vert ^{2}  \label{xxiiI} \\
&&+\frac{CB_{\mathcal{A}}^{2}}{\alpha }\int_{2\mathcal{R}}\zeta
^{2}\left\vert \nabla w\right\vert ^{2\beta }+\frac{CB_{\mathcal{A}}^{2}}{%
\alpha }\int_{2\mathcal{R}}\zeta ^{2}\left\vert \nabla w\right\vert ^{2\beta
}\left\vert \nabla _{\!\!\sqrt{\mathcal{A}},w}w\right\vert ^{2}.  \notag
\end{eqnarray}%
}

From the chain rule, the super subordination Condition (\ref{gxtra}) for $%
\gamma ${, and Young's inequality, 
\begin{eqnarray}
II &\leq &\sum_{j=1}^{n}\int_{2\mathcal{R}}\zeta ^{2}\left( \left\vert
f_{j}\right\vert +\left\vert w_{j}\right\vert \left\vert f_{z}\right\vert
+\sum_{i=1}^{n}\left\vert \gamma _{j}^{i}\right\vert \left\vert
w_{i}\right\vert \right) \left\vert w_{j}\right\vert ^{2\beta -1}  \notag \\
&&+\sum_{j=1}^{n}\left\vert \int_{2\mathcal{R}}\sum_{i=1}^{n}\zeta
^{2}w_{i}\left( \gamma _{z}^{i}\right) w_{j}^{2\beta }\right\vert  \notag \\
&\leq &\mathcal{C}_{a}\int_{2\mathcal{R}}\zeta ^{2}\left( \left\vert \nabla
w\right\vert ^{2\beta -1}+\left\vert \nabla w\right\vert ^{2\beta }\right)
+C\sum_{j=1}^{n}\int_{2\mathcal{R}}\zeta ^{2}w_{j}^{2\beta }\left\vert
\nabla _{\!\!\sqrt{\mathcal{A}},w}w\right\vert ^{2}  \notag \\
&\leq &\mathcal{C}_{a}\int_{2\mathcal{R}}\zeta ^{2}\left\vert \nabla
w\right\vert ^{2\beta }+C\int_{2\mathcal{R}}\zeta ^{2}\left\vert \nabla
w\right\vert ^{2\beta }\left\vert \nabla _{\!\!\sqrt{\mathcal{A}}%
,w}w\right\vert ^{2}+\mathcal{C}_{a},  \label{xxiiII}
\end{eqnarray}%
where } $\mathcal{C}_{a}=\mathcal{C}_{a}\left( B_{\gamma },\Lambda
,\left\Vert f\right\Vert _{\mathcal{C}^{1}\left( \tilde{\Gamma}\right)
,}\left\Vert \vec{\gamma}\right\Vert _{\mathcal{C}^{1}\left( \tilde{\Gamma}%
\right) }\right) $ with $B_{\gamma }$ as in (\ref{gxtra}).

Since $\vec{\gamma}$ {\ }is of subunit type with respect to $\mathcal{A}$, 
\begin{eqnarray}
III &=&\frac{1}{\beta }\sum_{j=1}^{n}\left\vert \sum_{i=1}^{n}\int_{2%
\mathcal{R}}\left( \zeta \gamma ^{i}\partial _{i}\left( w_{j}^{\beta
}\right) \right) \left( \zeta w_{j}^{\beta }\right) \right\vert  \notag \\
&\leq &\alpha \sum_{j=1}^{n}\int_{2\mathcal{R}}\zeta ^{2}\left\vert \nabla
_{\!\!\sqrt{\mathcal{A}},w}w_{j}^{\beta }\right\vert ^{2}+\frac{B_{\gamma
}^{2}}{\alpha }\sum_{j=1}^{n}\int_{2\mathcal{R}}\zeta ^{2}w_{j}^{2\beta }.
\label{xxiiIII}
\end{eqnarray}%
Using (\ref{xxiiI}), (\ref{xxiiII}) and (\ref{xxiiIII}) in (\ref{extint00})
yields 
\begin{eqnarray*}
&&\sum_{j=1}^{n}\left\vert \int_{2\mathcal{R}}\zeta ^{2}\left( \mathcal{L}%
_{w}w_{j}\right) w_{j}^{2\beta -1}\right\vert \\
&\leq &C\alpha \sum_{j=1}^{n}\int_{2\mathcal{R}}\zeta ^{2}\left\vert \nabla
_{\!\!\sqrt{\mathcal{A}},w}w_{j}^{\beta }\right\vert ^{2}+C\int_{2\mathcal{R}%
}\left\vert \nabla w\right\vert ^{2\beta }\left\vert \nabla _{\!\!\sqrt{%
\mathcal{A}},w}\zeta \right\vert ^{2} \\
&&+\frac{\mathcal{C}_{a}B^{2}}{\alpha }\int_{2\mathcal{R}}\zeta
^{2}\left\vert \nabla w\right\vert ^{2\beta }+\frac{CB_{\mathcal{A}}^{2}}{%
\alpha }\int_{2\mathcal{R}}\zeta ^{2}\left\vert \nabla w\right\vert ^{2\beta
}\left\vert \nabla _{\!\!\sqrt{\mathcal{A}},w}w\right\vert ^{2}+\mathcal{C}%
_{a}.
\end{eqnarray*}%
Substituting on the right of (\ref{mmoo}) and absorbing into the left, we
get 
\begin{eqnarray}
\sum_{j=1}^{n}\int_{2\mathcal{R}}\zeta ^{2}\left\vert \nabla _{\!\!\sqrt{%
\mathcal{A}},w}w_{j}^{\beta }\right\vert ^{2} &\leq &C(\beta +1)\int_{2%
\mathcal{R}}\left\vert \nabla w\right\vert ^{2\beta }\left\vert \nabla _{\!\!%
\sqrt{\mathcal{A}},w}\zeta \right\vert ^{2}  \label{xint02} \\
&&+\,\mathcal{C}_{a}\beta B^{2}\int_{2\mathcal{R}}\zeta ^{2}\left\vert
\nabla w\right\vert ^{2\beta }  \notag \\
&&+\,C\beta B_{\mathcal{A}}^{2}\int_{2\mathcal{R}}\zeta ^{2}\left\vert
\nabla w\right\vert ^{2\beta }\left\vert \nabla _{\!\!\sqrt{\mathcal{A}}%
,w}w\right\vert ^{2}+\beta \mathcal{C}_{a}.  \notag
\end{eqnarray}%
Now we will further assume that for fixed constants $C$ and $\beta ,$ the
constant $B_{\mathcal{A}}$ is small enough so that 
\begin{equation}
C\beta B_{\mathcal{A}}^{2}B^{2}\leq \frac{1}{2}.  \label{double-xtra}
\end{equation}%
We will show later that this assumption causes no loss of generality.
Applying Lemma \ref{extraprev2} to the third term on the right of (\ref%
{xint02}), and absorbing into the left using (\ref{double-xtra}), we obtain 
\begin{eqnarray*}
\sum_{j=1}^{n}\int_{2\mathcal{R}}\zeta ^{2}\left\vert \nabla _{\!\!\sqrt{%
\mathcal{A}},w}w_{j}^{\beta }\right\vert ^{2} &\leq &C(\beta +1)\int_{2%
\mathcal{R}}\left\vert \nabla w\right\vert ^{2\beta }\left\vert \nabla _{\!\!%
\sqrt{\mathcal{A}},w}\zeta \right\vert ^{2}+\mathcal{C}_{a}\beta B^{2}\int_{2%
\mathcal{R}}\zeta ^{2}\left\vert \nabla w\right\vert ^{2\beta } \\
&&+C\beta B_{\mathcal{A}}^{2}B^{2}\int_{2\mathcal{R}}\left\vert \nabla _{\!\!%
\sqrt{\mathcal{A}},w}\zeta \right\vert ^{2}\left\vert \nabla w\right\vert
^{2\beta }+\beta \mathcal{C}_{a}
\end{eqnarray*}%
where $C$ depends on $M_{0}$ and $||f||_{L^{\infty }(\tilde{\Gamma})}$. In
turn, {applying the Sobolev inequality (\ref{spoinc}) to each $\zeta
w_{j}^{\beta }$ in the second term on the right, and applying Lemma \ref%
{extraprev} to the first and third terms on the right (and the similar term
arising from the Sobolev inequality), }we obtain 
\begin{eqnarray}
&&\sum_{j=1}^{n}\int_{2\mathcal{R}}\zeta ^{2}\left\vert \nabla _{\!\!\sqrt{%
\mathcal{A}},w}w_{j}^{\beta }\right\vert ^{2}  \notag \\
&\leq &\beta \mathcal{C}_{a}+\mathcal{C}_{b}\sum_{j=1}^{n}\int_{2\mathcal{R}%
}\left\vert \nabla \zeta \right\vert ^{2}\left\vert \nabla _{\!\!\sqrt{%
\mathcal{A}},w}w_{j}^{\beta -1}\right\vert ^{2}+\mathcal{C}%
_{b}\sum_{j=1}^{n}\int_{2\mathcal{R}}\xi ^{2}w_{j}^{2\beta -2},
\label{xint-induct}
\end{eqnarray}%
after absorbing into the left, by taking $r^{1}$ small enough depending on $%
\beta $, $B$, $\Lambda $, $\left\Vert f\right\Vert _{\mathcal{C}^{1}\left( 
\tilde{\Gamma}\right) }$ and $\left\Vert \vec{\gamma}\right\Vert _{\mathcal{C%
}^{1}\left( \tilde{\Gamma}\right) }$; here it is important to note that the
constants multiplying $r^{1}$ do not depend on the size of $\mathcal{R}$.
Also, 
\begin{equation*}
\mathcal{C}_{b}=\mathcal{C}\left( M_{0},\vec{k},\mathcal{R},\beta
,n,B,\Lambda ,\left\Vert f\right\Vert _{\mathcal{C}^{1}\left( \tilde{\Gamma}%
\right) },\left\Vert \vec{\gamma}\right\Vert _{\mathcal{C}^{1}\left( \tilde{%
\Gamma}\right) }\right) .
\end{equation*}

The conclusion of the theorem will now follow by induction from (\ref%
{xint-induct}) and the Sobolev inequality. Indeed, if $\beta =1,$ 
\begin{equation*}
\sum_{j=1}^{n}\int_{2\mathcal{R}}\zeta ^{2}\left\vert \nabla _{\!\!\sqrt{%
\mathcal{A}},w}w_{j}\right\vert ^{2}\leq \beta\mathcal{C}_{a}+\mathcal{C}%
_{b} \int_{2\mathcal{R}}\xi ^{2}\leq \mathcal{C}_{b}.
\end{equation*}
By the Sobolev inequality (\ref{spoinc}) {and Lemma \ref{extraprev}}, 
\begin{equation}
\sum_{j=1}^{n}\int_{2\mathcal{R}}\zeta ^{2} w_{j}^{2}+ \sum_{j=1}^{n}\int_{2%
\mathcal{R}}\zeta ^{2}\left\vert \nabla _{\!\! \sqrt{\mathcal{A}}%
,w}w_{j}\right\vert ^{2}\leq \mathcal{C}_{b}.  \label{xint-beta-0}
\end{equation}
Since $\Omega ^{\prime }\Subset \Omega $, there are a finite number of
rectangles $\left\{ \mathcal{R}_{i}\right\}$ satisfying (\ref{xint-beta-0})
and such that $\Omega ^{\prime }\subset \bigcup \mathcal{R} _{i}$. Choosing $%
\mathcal{C}_1$ as in Theorem \ref{extraintegrability} in case $\beta =1$, it
follows that{\ 
\begin{equation*}
\sum_{j=1}^{n}\int_{\Omega ^{\prime }} w_{j}^{2}+\sum_{j=1}^{n} \int_{\Omega
^{\prime }}\left\vert \nabla _{\!\!\sqrt{\mathcal{A} },w}w_{j}\right\vert
^{2}\leq \mathcal{C}_{1}.
\end{equation*}
}Suppose now that we have shown that {\ 
\begin{equation}
\sum_{j=1}^{n}\int_{\Omega ^{\prime }} w_{j}^{2\beta}+
\sum_{j=1}^{n}\int_{\Omega ^{\prime }}\left\vert \nabla _{\!\!\sqrt{\mathcal{%
A}},w}w_{j}^{\beta }\right\vert ^{2}\leq \mathcal{C}_{\beta }.
\label{xint-induct-M}
\end{equation}
for }$\beta =1,2,\dots ,M$. Given $x\in \Omega ^{\prime }$, let $\mathcal{R}$
be a{\ box satisfying property (\ref{ax1c3}) in }Condition{\ \ref{hyp1} and
such that $x\in \frac{1}{3}\mathcal{R}\subset 2\mathcal{R}\subset \Omega
^{\prime }\Subset \Omega $. }Then, from (\ref{xint-induct}), 
\begin{eqnarray*}
\sum_{j=1}^{n}\int_{2\mathcal{R}}\zeta ^{2}\left\vert \nabla _{\!\!\sqrt{%
\mathcal{A}},w}w_{j}^{M+1}\right\vert ^{2} \leq (M+1)\mathcal{C}_{a} &+& 
\mathcal{C} _{b}\sum_{j=1}^{n} \int_{2\mathcal{R}}\left\vert \nabla \zeta
\right\vert ^{2}\left\vert \nabla _{\!\!\sqrt{\mathcal{A}},w}w_{j}^{M}
\right\vert ^{2} \\
&+& \mathcal{C}_{b}\int_{2\mathcal{R}}\xi ^{2}\left\vert \nabla w\right\vert
^{2M}
\end{eqnarray*}
\begin{equation*}
\leq (M+1)\mathcal{C}_{a}+A^{2}\mathcal{C}_{b}\sum_{j=1}^{n}\int_{\Omega
^{\prime }}\left\vert \nabla _{\!\!\sqrt{\mathcal{A}},w}w_{j}^{M}
\right\vert ^{2}+ \mathcal{C}_{b}\int_{\Omega ^{\prime }}\left\vert \nabla
w\right\vert ^{2M} \leq \mathcal{C}_{M+1}.
\end{equation*}
An application of the Sobolev inequality,{\ Lemma \ref{extraprev}}, and a
covering argument complete the induction.

It remains to show that the extra assumption (\ref{double-xtra}) imposes no
loss of generality. For a constant $M\geq 1,$ let $u\left( x\right)
=Mw\left( x\right) $ and 
\begin{eqnarray*}
\widetilde{\mathcal{A}}\left( x,q\right) &=&\mathcal{A}\left( x,q/M\right)
,\qquad \qquad \widetilde{k^{\ast }}\left( x,q\right) =k^{\ast }\left(
x,q/M\right) , \\
\widetilde{\vec{\gamma}}\left( x,q\right) &=&\vec{\gamma}\left( x,q/M\right)
,\qquad \qquad \widetilde{f}\left( x,q\right) =Mf\left( x,q/M\right) .
\end{eqnarray*}%
Then $u$ satisfies 
\begin{eqnarray*}
&&\mathop{\rm div}\widetilde{\mathcal{A}}\left( x,u\right) \nabla u+%
\widetilde{\vec{\gamma}}\left( x,u\right) \cdot \nabla u+\widetilde{f}\left(
x,u\right) \\
&=&\mathop{\rm div}\widetilde{\mathcal{A}}\left( x,Mw\left( x\right) \right)
\nabla Mw\left( x\right) +\widetilde{\vec{\gamma}}\left( x,Mw\left( x\right)
\right) \cdot \nabla Mw\left( x\right) +\widetilde{f}\left( x,Mw\left(
x\right) \right) \\
&=&M\left[ \mathop{\rm div}\mathcal{A}\left( x,w\left( x\right) \right)
\nabla w\left( x\right) +\vec{\gamma}\left( x,w\left( x\right) \right) \cdot
\nabla w\left( x\right) +f\left( x,w\left( x\right) \right) \right] =0.
\end{eqnarray*}%
On the other hand, by (\ref{xtra}), 
\begin{eqnarray*}
\left\vert \partial _{q}\widetilde{\mathcal{A}}\left( x,q\right) \xi
\right\vert ^{2} &=&\left\vert \partial _{q}\mathcal{A}\left( x,q/M\right)
\xi \right\vert ^{2}=M^{-2}\left\vert \mathcal{A}_{z}\left( x,q/M\right) \xi
\right\vert ^{2} \\
&\leq & M^{-2}B_{\mathcal{A}}^{2}\,k^{\ast }\left( x,q/M\right) \,\xi ^{t}%
\mathcal{A}\left( x,q/M\right) \xi \\
&=&B_{\mathcal{A}^{2}}M^{-2} \widetilde{k^{\ast }}\left( x,q\right) \,\xi
^{t}\widetilde{\mathcal{A}}\left( x,q\right) \xi \\
&=&\left( \widetilde{B}_{\widetilde{\mathcal{A}}}\right) ^{2}\,\widetilde{%
k^{\ast }}\left( x,q\right) \,\xi ^{t}\widetilde{\mathcal{A}}\left(
x,q\right) \xi
\end{eqnarray*}%
with $\widetilde{B}_{\widetilde{\mathcal{A}}}=B_{\mathcal{A}}M^{-1}$. Since $%
\gamma $ is of subunit type (Definition \ref{subunitt}), 
\begin{equation*}
\left( \sum_{i=1}^{n}\widetilde{\gamma }^{i}\left( x,q\right) \xi
_{i}\right) ^{2}=\left( \sum_{i=1}^{n}\gamma ^{i}\left( x,q/M\right) \xi
_{i}\right) ^{2}\leq B_{\gamma }^{2}~\xi ^{t}\mathcal{A}\left( x,q/M\right)
\xi = B_{\gamma }^{2}~\xi ^{t}\widetilde{\mathcal{A}}\left( x,q\right)\xi.
\end{equation*}%
Also, by (\ref{gxtra}),%
\begin{eqnarray*}
\left\vert \partial _{q}\widetilde{\vec{\gamma}}\left( x,q\right) \cdot \xi
\right\vert ^{2} &=&M^{-2}\left\vert \vec{\gamma}_{z}\left( x,q/M\right)
\cdot \xi \right\vert ^{2}\leq M^{-2}B_{\gamma }^{2}~k^{\ast }\left(
x,q/M\right) \xi ^{t}\mathcal{A}\left( x,q/M\right) \xi \\
&=&\left( B_{\gamma }M^{-1}\right) ^{2}~\widetilde{k^{\ast }}\left(
x,q\right) \xi ^{t}\widetilde{\mathcal{A}}\left( x,q\right) \xi.
\end{eqnarray*}%
Hence $u\left( x\right) $ satisfies the equation 
\begin{equation*}
\mathop{\rm div}\widetilde{\mathcal{A}}\left( x,u\right) \nabla u+\widetilde{%
\vec{\gamma}}\left( x,u\right) \cdot \nabla u+\widetilde{f}\left( x,u\right)
=0
\end{equation*}%
with the corresponding constants in Condition (\ref{hyp2}), namely 
\begin{equation*}
\widetilde{B}_{\widetilde{\mathcal{A}}}= B_{\mathcal{A}}M^{-1}\qquad \text{%
and}\qquad \widetilde{B}_{\gamma }=B_{\gamma }.
\end{equation*}%
Hence, taking $M$ large enough and letting $\widetilde{B}=\max \{B_{\mathcal{%
A}}M^{-1},B_{\gamma }\}$, we have that 
\begin{equation*}
C\beta \left( \widetilde{B}_{\widetilde{\mathcal{A}}}\right) ^{2}\widetilde{B%
}^{2}=C\beta \left( \frac{B_{\mathcal{A}}}{M}\right) ^{2}\max \left\{ \left( 
\frac{B_{\mathcal{A}}}{M}\right) ^{2},B_{\gamma }^{2}\right\} \leq \frac{1}{2%
},
\end{equation*}%
so the extra assumption (\ref{double-xtra}) holds for the operator $%
\mathop{\rm div}\widetilde{\mathcal{A}}\left( x,\cdot \right) \nabla +%
\widetilde{\vec{\gamma}}\left( x,\cdot \right) \cdot \nabla +\widetilde{f}%
\left( x,\cdot \right) $. By the previous calculations, there is a constant 
\begin{equation*}
\widetilde{\mathcal{C}}_{\beta }=\widetilde{\mathcal{C}}_{\beta }\left( 
\widetilde{M_{0}},n,\widetilde{B},\Lambda ,\widetilde{\vec{k}},\left\Vert 
\widetilde{f}\right\Vert _{\mathcal{C}^{1}\left( \widetilde{\Gamma }^{\prime
}\right) },\left\Vert \widetilde{\vec{\gamma}}\right\Vert _{\mathcal{C}%
^{1}\left( \widetilde{\Gamma }^{\prime }\right) },\Omega ,\mathop{\rm dist}%
\left( \Omega ^{\prime },\partial \Omega \right) \right)
\end{equation*}%
such that 
\begin{equation*}
\sum_{j=1}^{n}\int_{\Omega ^{\prime }}u_{j}^{2\beta
}+\sum_{j=1}^{n}\int_{\Omega ^{\prime }}\left\vert \nabla _{\!\!\sqrt{%
\widetilde{\mathcal{A}}},u}u_{j}^{\beta }\right\vert ^{2}\leq \widetilde{%
\mathcal{C}}_{\beta };
\end{equation*}%
here $\widetilde{M_{0}}=\left\Vert u\right\Vert _{L^{\infty }\left( \Omega
^{\prime }\right) }=M\left\Vert w\right\Vert _{L^{\infty }\left( \Omega
^{\prime }\right) }$ {\ and} $\widetilde{\Gamma }^{\prime }=\Omega ^{\prime
}\times \left[ -\widetilde{M_{0}},\widetilde{M_{0}}\right] $. The general
result for $w$ follows from the identity $w=Mu$ and the definitions of $%
\widetilde{f}$ and $\widetilde{\vec{\gamma}}$ .{%
\endproof%
}

\subsection{Proof of Theorem \protect\ref{apriorial}}

{In this section we prove the a priori estimate in Theorem \ref{apriorial}
as a consequence of the higher integrability of $\nabla w$ established in
the previous section. Theorem \ref{apriorial} will be the main tool in the
proof of Theorem \ref{application}. }

{By Theorem \ref{quasi}, we only need to show that $\nabla w$ is locally
bounded in terms of appropriate parameters, i.e., we only need to show that
for every box $\mathcal{R}\subset 4\mathcal{R}\subset \Omega ^{\prime
}\Subset \Omega $, 
\begin{equation*}
\left\Vert \nabla w\right\Vert _{L^{\infty }\left( \mathcal{R}\right) }\leq 
\mathcal{C}\left( \left\Vert w\right\Vert _{L^{\infty }\left( \Omega
^{\prime }\right) },\mathcal{R},\vec{k},\mathcal{A},f,\vec{\gamma}\right) ;
\end{equation*}%
the dependence of }$\mathcal{C}${\ on its arguments will be made more
explicit as we proceed. By the Sobolev imbedding theorem, it is enough to
show that for some $\beta >n$, 
\begin{equation}
\left\Vert w_{ij}\right\Vert _{L^{\beta }\left( 2\mathcal{R}\right) }\leq 
\mathcal{C}_{\beta }\left( \left\Vert w\right\Vert _{L^{\infty }\left(
\Omega ^{\prime }\right) },\mathcal{R},\vec{k},\mathcal{A},f,\vec{\gamma}%
\right) ,\qquad 1\leq i,j\leq n.  \label{secondkbeta}
\end{equation}%
}

\vspace*{-0.2in}{Let $\mathcal{R}\subset {4\mathcal{R}\subset \Omega }%
^{\prime }$ be a box satisfying property (\ref{ax1c3}) in the nondegeneracy
Condition\ \ref{hyp1}. Since such boxes cover }$\Omega ^{\prime }$, there is
no loss of generality in adding this extra condition. {Applying the
Caccioppoli inequality for the ${\mathcal{A}}$-gradient, Lemma \ref{Moser},
to the smooth functions }$w_{ij}${$=\partial _{i}\partial _{j}w$, we get for 
$\beta \in \mathbb{N}$ that 
\begin{eqnarray*}
&&\sum_{i,j=1}^{n}\int_{2\mathcal{R}}\zeta ^{2}\left\vert \nabla _{\!\!\sqrt{%
\mathcal{A}},w}\,w_{ij}^{\beta }\right\vert ^{2} \\
&\leq &C\beta \sum_{i,j=1}^{n}\left\vert \int_{2\mathcal{R}}\zeta ^{2}\left( 
\mathcal{L}_{w}w_{ij}\right) w_{ij}^{2\beta -1}\right\vert
+C\sum_{i,j=1}^{n}\int_{2\mathcal{R}}\left\vert \nabla _{\!\!\sqrt{\mathcal{A%
}},w}\,\zeta \right\vert ^{2}w_{ij}^{2\beta }.
\end{eqnarray*}%
We estimate the first term on the right by Lemma \ref{secondcacc}, obtaining}
\begin{eqnarray*}
&&\sum_{i,j=1}^{n}\int_{2\mathcal{R}}\zeta ^{2}\left\vert \nabla _{\!\!\sqrt{%
\mathcal{A}},w}\,w_{ij}^{\beta }\right\vert ^{2} \\
&\leq &C\beta \sum_{i,j=1}^{n}\left\vert \int_{2\mathcal{R}}\left( \nabla
w_{j}\right) \cdot \mathbf{A}_{i}\nabla \zeta ^{2}w_{ij}^{2\beta
-1}\right\vert +C\beta \sum_{i,j=1}^{n}\left\vert \int_{2\mathcal{R}%
}w_{i}\left( \nabla w_{j}\right) \cdot \mathbf{A}_{z}\nabla \zeta
^{2}w_{ij}^{2\beta -1}\right\vert \\
&&+C\beta \sum_{i,j=1}^{n}\left\vert \int_{2\mathcal{R}}w_{ij}\left( \nabla
w\right) \cdot \mathbf{A}_{z}\nabla \zeta ^{2}w_{ij}^{2\beta -1}\right\vert +%
\frac{\mathcal{C}_{0}\beta }{\alpha }\int_{2\mathcal{R}}\zeta ^{2}\left\vert
\nabla ^{2}w\right\vert ^{2\beta } \\
&&+C\beta \alpha \sum_{i,j=1}^{n}\int_{2\mathcal{R}}\zeta ^{2}\left\vert 
\vec{\gamma}\cdot \nabla w_{ij}^{\beta }\right\vert ^{2}+C\beta \alpha
\int_{2\mathcal{R}}\zeta ^{2}\left\vert \vec{\gamma}_{z}\cdot \nabla
w\right\vert ^{2}\left\vert \nabla ^{2}w\right\vert ^{2\beta } \\
&&+C\beta \sum_{i,j=1}^{n}\left\vert \int_{2\mathcal{R}}\zeta ^{2}\left(
\left( \partial _{j}\vec{\gamma}\right) \cdot \nabla w_{i}\right)
w_{ij}^{2\beta -1}\right\vert +C\beta \sum_{i,j=1}^{n}\left\vert \int_{2%
\mathcal{R}}w_{i}\left( \nabla w\right) \cdot \mathbf{A}_{jz}\nabla \zeta
^{2}w_{ij}^{2\beta -1}\right\vert \\
&&+C\beta \sum_{i,j=1}^{n}\left\vert \int_{2\mathcal{R}}w_{i}w_{j}\left(
\nabla w\right) \cdot \mathbf{A}_{zz}\nabla \zeta ^{2}w_{ij}^{2\beta
-1}\right\vert +\mathcal{C}_{0}\beta \int_{2\mathcal{R}}\zeta ^{2}\left\vert
\nabla w\right\vert ^{6\beta } \\
&&+C\sum_{i,j=1}^{n}\int_{2\mathcal{R}}\left\vert \nabla _{\!\!\sqrt{%
\mathcal{A}},w}\,\zeta \right\vert ^{2}w_{ij}^{2\beta }+\mathcal{C}_{0}\beta
\end{eqnarray*}%
\vspace*{-0.2in}%
\begin{eqnarray}
&=&I+II+\cdots +VIII+IX+\mathcal{C}_{0}\beta \int_{2\mathcal{R}}\zeta
^{2}\left\vert \nabla w\right\vert ^{6\beta }  \label{snd-0} \\
&&+C\sum_{i,j=1}^{n}\int_{2\mathcal{R}}\left\vert \nabla _{\!\!\sqrt{%
\mathcal{A}},w}\,\zeta \right\vert ^{2}w_{ij}^{2\beta }+\mathcal{C}_{0}\beta
,  \notag
\end{eqnarray}%
where $0<\alpha <1$ is arbitrary, $\vec{\gamma}=\vec{\gamma}\left(
x,w\right) $, $f=f\left( x,w\right) $, ${\mathbf{A}}=\mathcal{A}\left(
x,w\right) $, $\mathbf{A}_{i}=\mathcal{A}_{i}\left( x,w\right) $, etc., and 
\begin{equation*}
\mathcal{C}_{0}=\mathcal{C}_{0}\left( \left\Vert \mathcal{A}\right\Vert _{%
\mathcal{C}^{3}\left( \Gamma ^{\prime }\right) },\left\Vert f\right\Vert _{%
\mathcal{C}^{2}\left( \Gamma ^{\prime }\right) },\left\Vert \vec{\gamma}%
\right\Vert _{\mathcal{C}^{2}\left( \Gamma ^{\prime }\right) }\right)
\end{equation*}%
with $\Gamma ^{\prime }=\Omega ^{\prime }\times \left[ -M_{0},M_{0}\right] $
and $M_{0}=\left\Vert w\right\Vert _{L^{\infty }\left( \Omega ^{\prime
}\right) }$.

Note that 
\begin{equation*}
\partial _{i}\mathbf{A}\left( x\right) =\partial _{i}\mathcal{A}\left(
x,w\right) =\mathcal{A}_{i}\left( x,w\right) +w_{i}\mathcal{A}_{z}\left(
x,w\right) =\mathbf{A}_{i}+w_{i}\mathbf{A}_{z}.
\end{equation*}%
{If $u$ is any smooth function, from (\ref{wirtm}), (\ref{xtra}), (\ref%
{xxtra}), (\ref{gxtra}) and since }$\vec{\gamma}${\ }is of subunit type with
respect to $\mathcal{A},$\ it follows that 
\begin{eqnarray}
\left\vert \mathbf{A}_{i} \nabla u\right\vert ^{2} &\leq &B_{\mathcal{A}%
}^{2}\left\vert \nabla _{\!\!\sqrt{\mathcal{A}}, w}u\right\vert ^{2},
\label{snd-1} \\
\left\vert w_{i}\mathbf{A}_{z} \nabla u\right\vert ^{2} &\leq &B_{\mathcal{A}%
}^{2}\left\vert w_{i}\right\vert ^{2}k^{\ast }\left( x,w\right)
\,\sum_{j=1}^{n}k^{j}\left( x,w\right) \,\left( \partial _{j}u\right) ^{2} 
\notag \\
&\leq &B_{\mathcal{A}}^{2}\left\vert \nabla _{\!\!\sqrt{\mathcal{A}},
w}w\right\vert ^{2}\left\vert \nabla _{\!\!\sqrt{\mathcal{A}},
w}u\right\vert ^{2},  \label{snd-2} \\
\sum_{i=1}^{n}\left\vert \mathbf{A}_{iz}\nabla u\right\vert ^{2}+\left\vert 
\mathbf{A}_{zz}\nabla u\right\vert ^{2} &\leq &\left( B_{\mathcal{A}%
}^{\prime }\right) ^{2}\left\vert \nabla _{\!\!\sqrt{\mathcal{A}},
w}u\right\vert ^{2},  \label{snd-3} \\
\left\vert \vec{\gamma}\cdot \nabla u\right\vert ^{2} &\leq &B_{\gamma
}^{2}\left\vert \nabla _{\!\!\sqrt{\mathcal{A}},w}u\right\vert ^{2}\qquad 
\text{and}  \label{snd-4} \\
\sum_{i=1}^{n}\left\vert \vec{\gamma}_{i}\cdot \nabla u\right\vert
^{2}+\left\vert \vec{\gamma}_{z}\cdot \nabla u\right\vert ^{2} &\leq
&B_{\gamma }^{2}k^{\ast }\left( x,w\right) \left\vert \nabla _{\!\!\sqrt{%
\mathcal{A} },w}u\right\vert ^{2}.  \label{snd-5}
\end{eqnarray}

We will now apply {these estimates together with the Schwarz and triangle
inequalities to treat each term of (\ref{snd-0}). } We will incorporate the
constants $B_{\mathcal{A}}, B_\gamma,$ etc. in our generic constant $C$.

By definition of $I$ and ({\ref{snd-1}}), 
\begin{eqnarray*}
I &\leq &C\beta \sum_{i,j=1}^{n}\left\vert \int_{2\mathcal{R}}\zeta
w_{ij}^{2\beta -1}\left( \nabla w_{j}\right) \cdot \mathbf{A}_{i}\left(
\nabla \zeta \right) \right\vert +C\beta \sum_{i,j=1}^{n}\left\vert \int_{2 
\mathcal{R}}\zeta ^{2}w_{ij}^{\beta -1}\left( \nabla w_{j}\right) \cdot 
\mathbf{A}_{i}\nabla w_{ij}^{\beta }\right\vert \\
&\leq &\alpha \sum_{i,j=1}^{n}\int_{2\mathcal{R}}\zeta ^{2}\left\vert \nabla
_{\!\!\sqrt{\mathcal{A}},w}w_{ij}^{\beta }\right\vert ^{2}+C\beta
\sum_{i,j=1}^{n}\int_{2\mathcal{R}}\left\vert \nabla _{\!\!\sqrt{\mathcal{A}}
,w}\zeta \right\vert ^{2} w_{ij}^{2\beta }+\frac{C\beta ^{2}}{\alpha }%
\sum_{i,j=1}^{n}\int_{2\mathcal{R}}\zeta ^{2} w_{ij}^{2\beta }.
\end{eqnarray*}
Similarly, using (\ref{snd-2}), 
\begin{eqnarray*}
II &\leq &C\beta \sum_{i,j=1}^{n}\left\vert \int_{2\mathcal{R}}\zeta
w_{i}w_{ij}^{2\beta -1}\left( \nabla w_{j}\right) \cdot \mathbf{A}%
_{z}\left(\nabla \zeta \right) \right\vert +C\beta
\sum_{i,j=1}^{n}\left\vert \int_{2 \mathcal{R}}\zeta ^{2}w_{ij}^{\beta
-1}w_{i}\left( \nabla w_{j}\right) \cdot \mathbf{A}_{z}\nabla w_{ij}^{\beta
}\right\vert \\
&\leq &\alpha \sum_{i,j=1}^{n}\int_{2\mathcal{R}}\zeta ^{2}\left\vert \nabla
_{\!\!\sqrt{\mathcal{A}},w}w_{ij}^{\beta }\right\vert ^{2}+C\beta
\sum_{i,j=1}^{n}\int_{2\mathcal{R}}\left\vert \nabla _{\!\!\sqrt{\mathcal{A}}
,w}\zeta \right\vert ^{2} w_{ij}^{2\beta } \\
&&+\frac{C\beta ^{2}}{\alpha }\sum_{i,j=1}^{n}\int_{2\mathcal{R}}\zeta
^{2}\left\vert \nabla _{\!\!\sqrt{\mathcal{A}},w}w\right\vert ^{2}
w_{ij}^{2\beta }.
\end{eqnarray*}
Treating $III$ analogously, we get 
\begin{eqnarray}
&&I+II+III+IV  \notag \\
&\leq &\alpha \sum_{i,j=1}^{n}\int_{2\mathcal{R}}\zeta ^{2}\left\vert \nabla
_{\!\!\sqrt{\mathcal{A}},w}w_{ij}^{\beta }\right\vert ^{2}+C\beta
\sum_{i,j=1}^{n}\int_{2\mathcal{R}}\left\vert \nabla _{\!\!\sqrt{\mathcal{A}}
,w}\zeta \right\vert ^{2}\left\vert w_{ij}\right\vert ^{2\beta }
\label{snd-I-IV} \\
&&+\frac{C\beta ^{2}}{\alpha }\sum_{i,j=1}^{n}\int_{2\mathcal{R}} \zeta ^{2}
w_{ij}^{2\beta }+\frac{C\beta ^{2}}{\alpha } \sum_{i,j=1}^{n}\int_{2\mathcal{%
R}}\zeta ^{2}\left\vert \nabla _{\!\! \sqrt{\mathcal{A}}S,w}w\right\vert
^{2} w_{ij}^{2\beta }.  \notag
\end{eqnarray}

By (\ref{snd-4}) and (\ref{snd-5}), 
\begin{equation}
V+VI\leq C\beta \alpha \sum_{i,j=1}^{n}\int_{2\mathcal{R}}\zeta
^{2}\left\vert \nabla _{\!\!\sqrt{\mathcal{A}},w}w_{ij}^{\beta }\right\vert
^{2}+C\beta \alpha \sum_{i,j=1}^{n}\int_{2\mathcal{R}}\zeta ^{2}\left\vert
\nabla _{\!\!\sqrt{\mathcal{A}},w}w\right\vert ^{2} w_{ij}^{2\beta }.
\label{snd-V-VI}
\end{equation}

Now, using the identity $\partial _{j}\vec{\gamma} =\vec{\gamma}_{j}+w_{j}%
\vec{\gamma}_{z}$, 
\begin{eqnarray*}
VII &\leq &C\beta \sum_{i,j=1}^{n}\left\vert \int_{2\mathcal{R}}\zeta
^{2}\left( \vec{\gamma}_{j}\cdot \nabla w_{i}\right) w_{ij}^{2\beta
-1}\right\vert +C\beta \sum_{i,j=1}^{n}\left\vert \int_{2\mathcal{R}}\zeta
^{2}w_{j}\left( \vec{\gamma}_{z}\cdot \nabla w_{i}\right) w_{ij}^{2\beta
-1}\right\vert \\
&=&VII_{1}+VII_{2}.
\end{eqnarray*}
We have 
\begin{eqnarray*}
VII_{1} &\leq &C\beta \sum_{i,j=1}^{n} \int_{2\mathcal{R}} \zeta^{2}
\left\Vert \vec{\gamma} \right\Vert_{C^1(2\mathcal{R})} w_{ij}^{2\beta} \\
&\leq &\mathcal{C}_0 \beta \sum_{i,j=1}^{n}\int_{2\mathcal{R}}\zeta ^{2}
w_{ij}^{2\beta }.
\end{eqnarray*}
Now, integrating by parts, 
\begin{eqnarray*}
VII_{2} &=&C\beta \sum_{i,j=1}^{n}\left\vert \int_{2\mathcal{R}}\nabla
w\cdot \partial _{i}\left( \zeta ^{2}\vec{\gamma}_{z}w_{j}w_{ij}^{2\beta
-1}\right) \right\vert \\
&\leq &C\beta \sum_{i,j=1}^{n}\left\vert \int_{2\mathcal{R}}\zeta w_{j}\zeta
_{i}w_{ij}^{2\beta -1}\left( \nabla w\right) \cdot \vec{\gamma}%
_{z}\right\vert +C\beta \sum_{i,j=1}^{n}\left\vert \int_{2\mathcal{R}}\zeta
^{2}w_{j}w_{ij}^{2\beta -1}\left( \nabla w\right) \cdot \left( \partial _{i}%
\vec{\gamma}_{z}\right) \right\vert \\
&&+C\beta \sum_{i,j=1}^{n}\left\vert \int_{2\mathcal{R}}\zeta ^{2}\left(
\nabla w\right) \cdot \vec{\gamma}_{z}w_{ij}^{2\beta }\right\vert +C\beta
\sum_{i,j=1}^{n}\left\vert \int_{2\mathcal{R}}\zeta ^{2}w_{j}w_{ij}^{\beta
-1}\left( \nabla w\right) \cdot \vec{\gamma}_{z}\left( \partial
_{i}w_{ij}^{\beta }\right) \right\vert \\
&=&VII_{2,1}+VII_{2,2}+VII_{2,3}+VII_{2,4}.
\end{eqnarray*}
Then 
\begin{eqnarray}
VII_{2,1} &=&C\beta \sum_{i,j=1}^{n}\left\vert \int_{2\mathcal{R}}\zeta
w_{j}w_{ij}^{2\beta -1}\zeta _{i}\vec{\gamma}_{z}\cdot \left( \nabla
w\right) \right\vert  \notag \\
&\leq &\frac{C\beta }{\alpha }\sum_{i,j=1}^{n}\int_{2\mathcal{R}}\zeta ^{2}
w_{ij}^{2\beta }+C\alpha \beta \sum_{i,j=1}^{n}\int_{2\mathcal{R}%
}w_{j}^{2}\zeta _{i}^{2} w_{ij}^{2\beta -2}\left\vert \vec{\gamma}_{z}\cdot
\nabla w\right\vert ^{2}.  \label{VII-2-1-0}
\end{eqnarray}
Assume for the moment that $\beta>1$. Then by using the super subordination
Condition \ref{hyp2} for $\gamma $, and Young's inequality in the form 
\begin{equation*}
\int \left\vert fg\right\vert \leq \frac{\beta -1}{\beta }\int \left\vert
f\right\vert ^{\frac{2\beta }{2\beta -2}}+\frac{1}{\beta}\int \left\vert
g\right\vert^{\beta },
\end{equation*}
the second term on the right of (\ref{VII-2-1-0}) is bounded by 
\begin{eqnarray*}
&& C\alpha \beta \sum_{i,j=1}^{n}\int_{2 \mathcal{R}}w_{j}^{2}
w_{ij}^{2\beta -2}\zeta _{i}^{2}k^{\ast }\left\vert \nabla _{\!\!\sqrt{%
\mathcal{A}},w}w \right\vert ^{2} \\
&\leq &C\alpha \beta \sum_{i,j=1}^{n}\int_{2\mathcal{R}} w_{ij}^{2\beta
-2}\left\vert \nabla _{\!\!\sqrt{\mathcal{A}},w} \zeta \right\vert
^{2}\left\vert \nabla w\right\vert ^{4} \\
&\leq &C\alpha \beta \sum_{i,j=1}^{n}\int_{2\mathcal{R}} w_{ij}^{2\beta
}\left\vert \nabla _{\!\!\sqrt{\mathcal{A}},w}\zeta \right\vert ^{2}+C\alpha
\beta A^{2} \sum_{i,j=1}^{n}\int_{2 \mathcal{R}}\xi ^{2}\left\vert \nabla
w\right\vert ^{4\beta },
\end{eqnarray*}
where we used that $\left\vert \nabla _{\!\!\sqrt{\mathcal{A}}, w}\zeta
\right\vert ^{2}\leq C A^{2}\xi ^{2}$. If on the other hand $\beta =1$, the
estimation of the second term in (\ref{VII-2-1-0}) is simpler; without using
Young's inequality, we can estimate it by just the second term above. Thus, 
\begin{eqnarray*}
VII_{2,1} &\leq &\frac{C\beta }{\alpha }\sum_{i,j=1}^{n}\int_{2\mathcal{R}
}\zeta ^{2} w_{ij}^{2\beta }+C\alpha \beta \sum_{i,j=1}^{n}\int_{2\mathcal{R}%
} w_{ij}^{2\beta} \left\vert \nabla _{\!\!\sqrt{\mathcal{A}},w}\zeta
\right\vert ^{2} \\
&&+C\alpha \beta A^{2} \sum_{i,j=1}^{n}\int_{2\mathcal{R}}\xi ^{2}\left\vert
\nabla w\right\vert ^{4\beta }.
\end{eqnarray*}
Similarly, 
\begin{eqnarray*}
VII_{2,2} &\leq &\frac{C\beta }{\alpha }\sum_{i,j=1}^{n}\int_{2\mathcal{R}
}\zeta ^{2} w_{ij}^{2\beta }+ C\alpha \beta \sum_{i,j=1}^{n}\int_{2\mathcal{R%
}}\zeta ^{2} w_{j}^{2\beta-2}\left\vert \nabla w\right\vert ^4 \\
&\leq &\frac{C\beta }{\alpha }\sum_{i,j=1}^{n}\int_{2\mathcal{R}}\zeta ^{2}
w_{ij}^{2\beta } +\mathcal{C}_0 \alpha \beta \int_{2\mathcal{R}}\zeta
^{2}\left\vert \nabla w\right\vert ^{4\beta },
\end{eqnarray*}
and 
\begin{equation*}
VII_{2,3}\leq \frac{C\beta }{\alpha }\sum_{i,j=1}^{n}\int_{2 \mathcal{R}%
}\zeta ^{2} w_{ij}^{2\beta }+ C\alpha \beta \sum_{i,j=1}^{n}\int_{2\mathcal{R%
}}\zeta ^{2}\left\vert \nabla _{\!\! \sqrt{\mathcal{A}},w}w\right\vert ^{2}
w_{ij}^{2\beta }.
\end{equation*}
Now, from (\ref{snd-5}), 
\begin{eqnarray*}
\left\vert \left( \nabla w\right) \cdot \vec{\gamma}_{z}\left( \partial
_{i}w_{ij}^{\beta }\right) \right\vert &\leq &B_{\gamma }\left\vert \nabla
_{\!\!\sqrt{\mathcal{A}},w}w\right\vert \sqrt{k^{\ast }\left( x,w\right) }
\left\vert \partial _{i}w_{ij}^{\beta }\right\vert \\
&\leq &B_{\gamma }\left\vert \nabla _{\!\!\sqrt{\mathcal{A}}, w}w\right\vert
\left\vert \nabla _{\!\!\sqrt{\mathcal{A}},w} w_{ij}^{\beta }\right\vert .
\end{eqnarray*}
Thus, 
\begin{eqnarray*}
VII_{2,4} &\leq &\frac{C\beta }{\alpha }\sum_{i,j=1}^{n} \int_{2\mathcal{R}%
}\zeta ^{2} w_{j}^{2} w_{ij}^{2\beta -2} \left\vert \nabla _{\!\!\sqrt{%
\mathcal{A}} ,w}w\right\vert ^{2}+C\alpha \beta \sum_{i,j=1}^{n}\int_{2%
\mathcal{R}}\zeta ^{2}\left\vert \nabla _{\!\!\sqrt{\mathcal{A}}%
,w}w_{ij}^{\beta } \right\vert ^{2} \\
&\leq &\frac{C\beta }{\alpha }\sum_{i,j=1}^{n}\int_{2\mathcal{\ R}}\zeta
^{2} w_{ij}^{2\beta } +\frac{\mathcal{C}_0\beta }{\alpha }\int_{2\mathcal{R}%
}\zeta ^{2}\left\vert \nabla w\right\vert ^{4\beta }+C\alpha \beta
\sum_{i,j=1}^{n}\int_{2\mathcal{\ R}}\zeta ^{2}\left\vert \nabla _{\!\!\sqrt{%
\mathcal{A}},w} w_{ij}^{\beta }\right\vert ^{2}.
\end{eqnarray*}
Assembling all these estimates, we obtain 
\begin{eqnarray}
VII &\leq &\frac{\mathcal{C}_0\beta }{\alpha } \sum_{i,j=1}^{n} \int_{2%
\mathcal{R}} \zeta ^{2} w_{ij}^{2\beta }+ C \alpha\beta
\sum_{i,j=1}^{n}\int_{2\mathcal{R}} w_{ij}^{2\beta } \left\vert \nabla _{\!\!%
\sqrt{\mathcal{A}},w}\zeta \right\vert ^{2}  \label{snd-VII} \\
&&+C\alpha \beta \sum_{i,j=1}^{n}\int_{2\mathcal{R}}\zeta ^{2}\left\vert
\nabla _{\!\!\sqrt{\mathcal{A}},w}w\right\vert ^{2} w_{ij}^{2\beta }  \notag
\\
&&+\frac{\mathcal{C}_0 A^{2}\beta }{\alpha }\int_{2\mathcal{R} }\xi
^{2}\left\vert \nabla w\right\vert ^{4\beta }+C\alpha \beta
\sum_{i,j=1}^{n}\int_{2\mathcal{R}}\zeta ^{2}\left\vert \nabla _{\!\! \sqrt{%
\mathcal{A}},w}w_{ij}^{\beta }\right\vert ^{2} .  \notag
\end{eqnarray}

By (\ref{snd-3}), 
\begin{eqnarray}
VIII &\leq &C\beta \sum_{i,j=1}^{n}\left\vert \int_{2\mathcal{R}}\zeta
w_{i}w_{ij}^{2\beta -1}\left( \nabla w\right) \cdot \mathbf{A}_{jz}\left(
\nabla \zeta \right) \right\vert  \notag \\
&&+C\beta \sum_{i,j=1}^{n}\left\vert \int_{2\mathcal{R}}\zeta
^{2}w_{i}\left( \nabla w\right) w_{ij}^{\beta -1}\cdot \mathbf{A}_{jz}\nabla
w_{ij}^{\beta }\right\vert  \notag \\
&\leq &C\beta \sum_{i,j=1}^{n}\int_{2\mathcal{R}} w_{ij}^{2\beta }\left\vert
\nabla _{\!\!\sqrt{\mathcal{A}},w}\zeta \right\vert ^{2}+C\alpha \beta
\sum_{i,j=1}^{n}\int_{2\mathcal{R}}\zeta ^{2}\left\vert \nabla _{\!\!\sqrt{%
\mathcal{A}},w}w_{ij}^{\beta }\right\vert ^{2}  \notag \\
&&+\frac{C\left( B_{\mathcal{A}}^{\prime }\right) ^{2}\beta }{\alpha }
\sum_{i,j=1}^{n}\int_{2\mathcal{R}}\zeta ^{2}w_{i}^{2} w_{ij}^{2\beta
-2}\left\vert \nabla w\right\vert ^{2}  \notag \\
&\leq &C\beta \sum_{i,j=1}^{n}\int_{2\mathcal{R}} w_{ij}^{2\beta }\left\vert
\nabla _{\!\!\sqrt{\mathcal{A}},w}\zeta \right\vert ^{2}+\frac{C\left( B_{%
\mathcal{A}}^{\prime }\right) ^{2}\beta }{\alpha }\sum_{i,j=1}^{n}\int_{2%
\mathcal{R}}\zeta ^{2} w_{ij}^{2\beta }  \label{snd-VIII} \\
&&+\frac{\mathcal{C}_0\left( B_{\mathcal{A}}^{\prime }\right) ^{2}\beta }{%
\alpha }\int_{2 \mathcal{R}}\zeta ^{2}\left\vert \nabla w\right\vert
^{4\beta }+C\alpha \beta \sum_{i,j=1}^{n}\int_{2\mathcal{R}}\zeta
^{2}\left\vert \nabla _{\!\!\sqrt{\mathcal{A}},w}w_{ij}^{\beta }\right\vert
^{2}.  \notag
\end{eqnarray}
Similarly, 
\begin{eqnarray}
IX &\leq &C\beta \sum_{i,j=1}^{n}\int_{2\mathcal{R}} w_{ij}^{2\beta
}\left\vert \nabla _{\!\!\sqrt{\mathcal{A}},w}\zeta \right\vert ^{2}+\frac{%
\mathcal{C}_0\left( B_{\mathcal{A}}^{\prime }\right) ^{2}\beta }{\alpha }%
\sum_{i,j=1}^{n}\int_{2\mathcal{R}}\zeta ^{2} w_{ij}^{2\beta }
\label{snd-IX} \\
&&+\frac{C\left( B_{\mathcal{A}}^{\prime }\right) ^{2}\beta }{\alpha }%
\int_{2 \mathcal{R}}\zeta ^{2}\left\vert \nabla w\right\vert ^{6\beta
}+C\alpha \beta \sum_{i,j=1}^{n}\int_{2\mathcal{R}}\zeta ^{2}\left\vert
\nabla _{\!\!\sqrt{\mathcal{A}},w}w_{ij}^{\beta }\right\vert ^{2}.  \notag
\end{eqnarray}
Using estimates (\ref{snd-I-IV}), (\ref{snd-V-VI}), (\ref{snd-VII}), (\ref%
{snd-VIII}) and (\ref{snd-IX}) in (\ref{snd-0}), and absorbing into the
left, we obtain 
\begin{eqnarray}
\sum_{i,j=1}^{n}\int_{2\mathcal{R}}\zeta ^{2}\left\vert \nabla _{\!\! \sqrt{%
\mathcal{A}},w}\,w_{ij}^{\beta }\right\vert ^{2} &\leq &C\beta
\sum_{i,j=1}^{n}\int_{2\mathcal{R}}\left\vert \nabla _{\!\!\sqrt{\mathcal{A}}
,w}\zeta \right\vert ^{2} w_{ij}^{2\beta }  \label{tta} \\
&&+\mathcal{C}_0 \beta ^{2}\sum_{i,j=1}^{n}\int_{2\mathcal{R}}\zeta ^{2}
w_{ij}^{2\beta }  \notag \\
&&+C\beta ^{2}\sum_{i,j=1}^{n}\int_{2\mathcal{R}}\zeta ^{2}\left\vert \nabla
_{\!\!\sqrt{\mathcal{A}},w}w\right\vert ^{2} w_{ij}^{2\beta }  \notag \\
&&+\mathcal{C}_0 \beta ^{2}\int_{2\mathcal{R}}\xi ^{2}\left\vert \nabla
w\right\vert ^{6\beta }+\mathcal{C}_0 A^{2}\beta .  \notag
\end{eqnarray}

{We now estimate the third term on the right of (\ref{tta}) proceeding as in
the proof of Lemma \ref{extraprev2}. Integrating by parts, 
\begin{eqnarray}
\sum_{i,j=1}^{n}\int_{2\mathcal{R}}\zeta ^{2}\left\vert \nabla _{\!\! \sqrt{%
\mathcal{A}},w}w\right\vert ^{2} w_{ij}^{2\beta } &=&-\sum_{i,j=1}^{n}\int_{2%
\mathcal{R}}w\mathop{\rm div} \left\{\zeta ^{2}w_{ij}^{2\beta }\mathcal{A}%
\left( x,w\right) \nabla w\right\}  \notag \\
&\leq &M_{0}\sum_{i,j=1}^{n}\left\vert \int_{2\mathcal{R}} w_{ij}^{2\beta
}\left( \nabla \zeta ^{2}\right) \cdot \mathcal{A} \left( x,w\right) \nabla
w\right\vert  \notag \\
&&+M_{0}\sum_{i,j=1}^{n}\left\vert \int_{2\mathcal{R}}\zeta ^{2}\left(
\nabla w_{ij}^{2\beta }\right) \cdot \mathcal{A}\left( x,w\right) \nabla
w\right\vert  \notag \\
&&+M_{0}\sum_{i,j=1}^{n}\int_{2\mathcal{R}}\zeta ^{2} w_{ij}^{2\beta
}\left\vert \vec{\gamma}\cdot \nabla w+f\right\vert .  \notag
\end{eqnarray}
By Schwarz's inequality, the identity }$\nabla w_{ij}^{2\beta
}=2w_{ij}^{\beta }\nabla w_{ij}^{\beta }${, and (\ref{snd-4}), we obtain
after absorbing into the left, 
\begin{eqnarray*}
&&\sum_{i,j=1}^{n}\int_{2\mathcal{R}}\zeta ^{2}\left\vert \nabla _{\!\! 
\sqrt{\mathcal{A}},w}w\right\vert ^{2} w_{ij}^{2\beta }\leq C(M_{0}^{2}+1)
\left( \left\Vert f\right\Vert _{L^{\infty }\left( \Gamma ^{\prime }\right)
}+B_{\gamma}^{2}\right) \sum_{i,j=1}^{n}\int_{2\mathcal{R}}\zeta ^{2}
w_{ij}^{2\beta } \\
&&+CM_{0}^{2}\sum_{i,j=1}^{n}\int_{2\mathcal{R}}\left\vert \nabla _{\!\! 
\sqrt{\mathcal{A}},w}\zeta \right\vert ^{2} w_{ij}^{2\beta}+
CM_{0}^{2}\sum_{i,j=1}^{n}\int_{2\mathcal{R}}\zeta ^{2}\left\vert \nabla
_{\!\!\sqrt{\mathcal{A}},w}\,w_{ij}^{\beta }\right\vert ^{2}.
\end{eqnarray*}
Using this on the right of (\ref{tta}) gives} 
\begin{eqnarray}
\sum_{i,j=1}^{n}\int_{2\mathcal{R}}\zeta ^{2}\left\vert \nabla _{\!\! \sqrt{%
\mathcal{A}},w}\,w_{ij}^{\beta }\right\vert ^{2} &\leq &C\beta ^{2}\left(
M_{0}^{2}+1\right) \sum_{i,j=1}^{n}\int_{2\mathcal{R}} \left\vert \nabla
_{\!\!\sqrt{\mathcal{A}},w}\zeta \right\vert ^{2} w_{ij}^{2\beta }
\label{tta1} \\
&&+ \mathcal{C}_1 \sum_{i,j=1}^{n}\int_{2\mathcal{R}}\zeta ^{2}
w_{ij}^{2\beta }  \notag \\
&&+C\beta ^{2}M_{0}^{2}\sum_{i,j=1}^{n}\int_{2\mathcal{R}}\zeta
^{2}\left\vert \nabla _{\!\!\sqrt{\mathcal{A}},w}\,w_{ij}^{\beta }
\right\vert ^{2}  \notag \\
&&+ \mathcal{C}_0 B^2\beta ^{2}\int_{2\mathcal{R}}\xi ^{2}\left\vert \nabla
w\right\vert ^{6\beta }+ \mathcal{C}_0 A^{2}\beta ,  \notag
\end{eqnarray}
with 
\begin{equation*}
\mathcal{C}_{1}=\mathcal{C}_{1}\left( M_0, B,\beta,\left\Vert \mathcal{A}
\right\Vert _{\mathcal{C}^{3}\left( \Gamma ^{\prime }\right) },\left\Vert
f\right\Vert _{\mathcal{C}^{2}\left( \Gamma ^{\prime }\right) },\left\Vert 
\vec{\gamma}\right\Vert _{\mathcal{C}^{2}\left( \Gamma ^{\prime }\right)
}\right) .
\end{equation*}
{By the Sobolev inequality (\ref{spoinc}) and the product rule, 
\begin{equation*}
\int_{2\mathcal{R}}\zeta ^{2} w_{ij}^{2\beta }\leq C\left( r^{1}\right)
^{2}\int_{2\mathcal{R}}\zeta ^{2}\left\vert \nabla _{\!\!\sqrt{\mathcal{A}}%
,w}w_{ij}^{\beta }\right\vert ^{2}+C\left( r^{1}\right) ^{2}\int_{2\mathcal{R%
}}\left\vert \nabla _{\!\!\sqrt{\mathcal{A}},w}\zeta \right\vert ^{2}\,
w_{ij}^{2\beta }.
\end{equation*}
Applying this to the second term on the right of (\ref{tta1}), and taking }$%
r^1$ small enough depending on $M_0, \beta $, $B$, $\left\Vert \mathcal{A}%
\right\Vert _{\mathcal{C}^{3}\left( \Gamma ^{\prime }\right) }$, $\left\Vert
f\right\Vert _{\mathcal{C}^{2}\left( \Gamma ^{\prime }\right) }$ and $%
\left\Vert \vec{\gamma}\right\Vert _{\mathcal{C}^{2}\left( \Gamma ^{\prime
}\right) }$ (in order to absorb the term resulting from the first term on
the right of the last estimate), we get 
\begin{equation*}
\sum_{i,j=1}^{n}\int_{2\mathcal{R}}\zeta ^{2}\left\vert \nabla _{\!\! \sqrt{%
\mathcal{A}},w}\,w_{ij}^{\beta }\right\vert ^{2} \leq \mathcal{C}
_{1}\sum_{i,j=1}^{n}\int_{2\mathcal{R}}\left\vert \nabla _{\!\!\sqrt{%
\mathcal{A} } ,w}\zeta \right\vert ^{2} w_{ij}^{2\beta }
\end{equation*}
\begin{equation*}
+C\beta ^{2}M_{0}^{2}\sum_{i,j=1}^{n}\int_{2\mathcal{R}}\zeta ^{2}\left\vert
\nabla _{\!\!\sqrt{\mathcal{A}},w}\,w_{ij}^{\beta }\right\vert ^{2} +%
\mathcal{C}_{0}B^{2}\beta ^{2}\int_{2\mathcal{R}}\xi ^{2}\left\vert \nabla
w\right\vert ^{6\beta }+\mathcal{C}_{0}A^{2}\beta .
\end{equation*}

{Now restrict }$1\leq \beta \leq n+1$, and assume that $M_{0}$ is small
enough so that 
\begin{equation}
C\beta ^{2}M_{0}^{2}\leq C\left( n+1\right) ^{2}M_{0}^{2}\leq 1/2.
\label{triple-xtra}
\end{equation}
Then the second term on the right above may be absorbed into the left side
to obtain 
\begin{eqnarray}
\sum_{i,j=1}^{n}\int_{2\mathcal{R}}\zeta ^{2}\left\vert \nabla _{\!\! \sqrt{%
\mathcal{A}},w}\,w_{ij}^{\beta }\right\vert ^{2} &\leq &\mathcal{C}
_{1}\sum_{i,j=1}^{n}\int_{2\mathcal{R}}\left\vert \nabla _{\!\!\sqrt{%
\mathcal{A}} ,w}\zeta \right\vert ^{2} w_{ij}^{2\beta }  \label{tta2} \\
&&+\mathcal{C}_{0}B^{2}\beta ^{2}\int_{2\mathcal{R}}\xi ^{2}\left\vert
\nabla w\right\vert ^{6\beta }+\mathcal{C}_{0}A^{2}\beta .  \notag
\end{eqnarray}
{For any $q\geq 1,$ 
\begin{equation*}
\left\Vert \chi \nabla \mathbf{A}\right\Vert _{L^{q}}\leq C\left( \int_{2 
\mathcal{R}}\sum_{i=1}^n \left\vert \mathbf{A}_{i}\right\vert
^{q}+\left\vert \nabla w\right\vert ^{q}\left\vert \mathbf{A}_{z}\right\vert
^{q}\right) ^{\frac{1}{ q}}\leq C\left\Vert \mathcal{A}\right\Vert _{%
\mathcal{C}^{1}\left( \Gamma ^{\prime }\right) }\left( 1+\left\Vert \nabla
w\right\Vert _{L^{q}\left( 2\mathcal{R}\right) }\right).
\end{equation*}
Then from Lemma \ref{lemma-interpolation} applied to the function $u = \xi
w_ij$, choosing $q=n+1$, there exists $1<p=p\left( n\right) <2$ such that 
\begin{eqnarray*}
&&\sum_{i,j=1}^{n}\int_{2\mathcal{R}}\left\vert \nabla _{\!\!\sqrt{\mathcal{A%
}},w}\,\zeta \right\vert ^{2} w_{ij}^{2\beta } = \sum_{i,j=1}^{n} \int_{2%
\mathcal{R}}\left\vert \nabla _{\!\!\sqrt{\mathcal{A}} ,w}\,\zeta
\right\vert ^{2} \left(\xi w_{ij}\right)^{2\beta } \\
&\leq& \varepsilon ^{-1}\mathcal{C}\left( n,A,\left\Vert \nabla w\right\Vert
_{L^{n+1}\left( 2\mathcal{R}\right) }\right) \sum_{i,j=1}^{n}\left\Vert \xi
w_{ij}^{\beta }\right\Vert _{L^{p}}^{2} +C\varepsilon \sum_{i,j=1}^{n}\int_{2%
\mathcal{R}}\zeta ^{2}\left\vert \nabla _{\!\!\sqrt{\mathcal{A}}%
,w}\,w_{ij}^{\beta }\right\vert ^{2}
\end{eqnarray*}
for all }$1\leq \beta \leq n+1$. {On the other hand, by Theorem \ref%
{extraintegrability}, 
\begin{equation}
\left\Vert \nabla w\right\Vert _{L^{n+1}\left( 2\mathcal{R}\right)
}+\left\Vert \nabla w\right\Vert _{L^{6\beta }\left( 2\mathcal{R}\right)
}\leq \mathcal{C}_{2}  \label{ttw}
\end{equation}
where }$\mathcal{C}_{2}=\mathcal{C}_{2}\left( M_0,n,B,\Lambda ,\vec{k}
,\left\Vert f\right\Vert _{\mathcal{C}^{1}\left( \Gamma ^{\prime }\right)
},\left\Vert \vec{\gamma}\right\Vert _{\mathcal{C}^{1}\left( \Gamma ^{\prime
}\right) }, \Omega ,\mathop{\rm dist}\left( \Omega ^{\prime },\partial
\Omega \right) \right) $. {Using these estimates in (\ref{tta2}) and
absorbing into the left gives} 
\begin{equation}
\sum_{i,j=1}^{n}\int_{2\mathcal{R}}\zeta ^{2}\left\vert \nabla _{\!\! \sqrt{%
\mathcal{A}},w}\,w_{ij}^{\beta }\right\vert ^{2}\leq \mathcal{C}
_{3}\sum_{i,j=1}^{n}\left\Vert \xi w_{ij}^{\beta }\right\Vert _{L^{p}}^{2}+ 
\mathcal{C}_{3}.  \label{tttc}
\end{equation}
{\ with 
\begin{equation*}
\mathcal{C}_{3}=\mathcal{C}_{3}\left( M_0, n,B,\Lambda ,\vec{k},\mathcal{R},
\left\Vert \mathcal{A}\right \Vert _{\mathcal{C}^{3}\left( \Gamma ^{\prime
}\right) },\left\Vert f\right\Vert _{\mathcal{C}^{2}\left( \Gamma ^{\prime
}\right) },\left\Vert \vec{\gamma}\right\Vert _{\mathcal{C} ^{2}\left(
\Gamma ^{\prime }\right) },\Omega ,\mathop{\rm dist}\left( \Omega ^{\prime
},\partial \Omega \right) \right) ,
\end{equation*}%
where we have used Remark \ref{AdependsonR} to substitute the dependence on }%
$A$ by dependence on $\vec{k}$, $\mathcal{R}$, $M_{0}$.

Choosing {$\beta =1$ in (\ref{tta2}) and applying Lemma \ref{minkgrad} and (%
\ref{ttw}) gives} 
\begin{eqnarray*}
\sum_{i,j=1}^{n}\int_{2\mathcal{R}}\zeta ^{2}\left\vert \nabla _{\!\! \sqrt{%
\mathcal{A}},w}\,w_{ij}\right\vert ^{2} &\leq &\mathcal{C}_{1}
\sum_{i,j=1}^{n}\int_{2\mathcal{R}}\left\vert \nabla _{\!\!\sqrt{\mathcal{A}}
,w}\zeta \right\vert ^{2} w_{ij}^{2}+\mathcal{C}_{0}B^{2} \int_{2\mathcal{R}%
}\xi ^{2}\left\vert \nabla w\right\vert ^{6}+\mathcal{C}_{0}A^{2} \\
&\leq &\mathcal{C}_{1}C_{1}A^{2}\Lambda \sum_{i=1}^{n}\int_{2\mathcal{R}}\xi
^{2}\left\vert \nabla _{\!\!\sqrt{\mathcal{A}},w}w_{i}\right\vert ^{2} +%
\mathcal{C}_{3}.
\end{eqnarray*}
{Estimating the first term on the right by Theorem \ref{extraintegrability},
we then get} 
\begin{equation}
\sum_{i,j=1}^{n}\int_{2\mathcal{R}}\zeta ^{2}\left\vert \nabla _{\!\! \sqrt{%
\mathcal{A}},w}\,w_{ij}\right\vert ^{2}\leq \mathcal{C}_{3}.  \label{tttd}
\end{equation}

{To finish the proof, we iterate (\ref{tttc}) in a similar fashion as in the
proof of Theorem \ref{extraintegrability}, using (\ref{tttd}) to start the
iteration. We omit the details. As a result, we obtain 
\begin{equation*}
\sum_{i,j=1}^{n}\int_{\Omega ^{\prime }}\left\vert \,w_{ij}\right\vert
^{n+1}\leq \mathcal{C}^{\ast }\left( M_0, n,B,\Lambda ,\vec{k},\left\Vert 
\mathcal{A}\right\Vert _{\mathcal{C}^{3}\left( \Gamma ^{\prime }\right)
},\left\Vert f\right\Vert _{\mathcal{C}^{2}\left( \Gamma ^{\prime }\right)
},\left\Vert \vec{\gamma}\right\Vert _{\mathcal{C}^{2}\left( \Gamma ^{\prime
}\right) }, \mathop{\rm dist}\left( \Omega ^{\prime },\partial \Omega
\right) ,\Omega \right).
\end{equation*}
From the Sobolev embedding theorem it follows that 
\begin{equation*}
\left\Vert \nabla w\right\Vert _{L^{\infty }\left( \Omega^{\prime} \right)
}\,\leq \mathcal{C}^{\ast },
\end{equation*}
as desired.}

It remains to prove that assumption (\ref{triple-xtra}) does not result in a
loss of generality.{\ This will be accomplished by a change of variables as
at the end of the proof of Theorem \ref{extraintegrability}. In fact,
letting }$u\left( x\right) =w\left( x\right)/N$, $u$ satisfies the equation 
\begin{equation}
\mathop{\rm div}\widetilde{\mathcal{A}}\left( x,u\right) \nabla u+\widetilde{%
\vec{\gamma}}\left( x,u\right) \cdot \nabla u+\widetilde{f}\left( x,u\right)
=0  \label{eq2}
\end{equation}
in $\Omega $, where 
\begin{equation*}
\widetilde{\mathcal{A}}\left( x,q\right) =\mathcal{A}\left( x,Nq\right)
,\qquad \widetilde{\vec{\gamma}}\left( x,q\right) =\vec{\gamma}\left(
x,Nq\right) ,\qquad \widetilde{f}\left( x,q\right) =\frac{1}{N}f\left(
x,Nq\right) .
\end{equation*}
Moreover, 
\begin{equation*}
\widetilde{M}_{0}=\left\Vert u\right\Vert _{L^{\infty }\left( \Omega
^{\prime }\right) }=\frac{M_{0}}{N}.
\end{equation*}
Hence, for $N$ big enough, the analogue of (\ref{triple-xtra}) holds for $u$%
. The result for $w$ then follows from the identity $w=Nu$.{\ 
\endproof%
}

\section{{\label{section-Proof}Proof of the Hypoellipticity Theorem}}

{In this section we prove our main results Theorems \ref{application} and %
\ref{DP}. To do so, we will apply Theorem 15.19 of \cite{GiTr}. For easy
reference we now state a special version of this theorem suitable for our
needs. }

\begin{theorem}[Theorem 15.19 in \protect\cite{GiTr}]
{\ \label{th15.19} Let $\Omega $ be a bounded domain in $\mathbb{R}^{n}$
satisfying an exterior sphere condition at each point of $\partial \Omega $.
Let $\mathcal{M}$ be a divergence structure operator, 
\begin{equation*}
\mathcal{M}\left( w\right) =\mathop{\rm div}\mathcal{S}\left( x,w,\nabla
w\right) +\mathcal{T}\left( x,w,\nabla w\right) ,
\end{equation*}
where $\mathcal{S}=\left( \mathcal{S}^{1},\mathcal{S}^{2},\dots ,\mathcal{S}
^{n}\right) $, $\mathcal{S}^{i}\in \mathcal{C}^{1+\delta }\left( \Omega
\times \mathbb{R}\times \mathbb{R}^{n}\right) $ for all $i$, $\mathcal{T} $ $%
\in \mathcal{C}^{\delta }\left( \Omega \times \mathbb{R}\times \mathbb{R}
^{n}\right) $ for some $0<\delta <1$, and where there exist positive
constants $a_{0}$, $b_{0}$, $b_{1}$, $c_{0}$ and $d_{0}$ such that for all $%
\xi \in \mathbb{R}^{n}$ and all $\left( x,z,p\right) \in \Omega \times 
\mathbb{R}\times \mathbb{R}^{n}$, } 
\begin{eqnarray}
\xi ^{t}\left\{ \nabla _{p}\mathcal{S}\left( x,z,p\right) \right\} \xi &\geq
&c_{0}\left\vert \xi \right\vert ^{2}  \label{15.59-1} \\
\left\vert \nabla _{p}\mathcal{S}\left( x,z,p\right) \right\vert &\leq &d_{0}
\label{15.64-1} \\
\left( 1+\left\vert p\right\vert \right) \left\vert \partial _{z}\mathcal{S}
\right\vert +\left\vert \nabla _{x}\mathcal{S}\right\vert +\left\vert 
\mathcal{T}\right\vert &\leq &d_{0}\left( 1+\left\vert p\right\vert \right)
^{2}  \label{15.66-1} \\
p\cdot \mathcal{S}\left( x,z,p\right) &\geq &a_{0}\left\vert p\right\vert
^{2}  \label{10.23-a} \\
\mathcal{T}\left( x,z,p\right) \mathop{\rm sign}z &\leq &b_{0}\left\vert
p\right\vert +b_{1}.  \label{10.23-b}
\end{eqnarray}
{Then for any function $\varphi \in \mathcal{C}^{0}\left( \partial \Omega
\right) $, there exists a solution $w\in \mathcal{C}^{2+\delta }\left(
\Omega \right) \bigcap \mathcal{C}^{0}\left( \overline{\Omega}\right) $ of
the Dirichlet problem 
\begin{eqnarray*}
\mathcal{M}\left( w\right) &=&0\qquad \text{in }\Omega \\
w &=&\varphi \qquad \text{on }\partial \Omega .
\end{eqnarray*}
}
\end{theorem}

\begin{remark}
{\ Theorem 15.19 in \cite{GiTr} is established only for $a_{0}=1$. Its
generalization to any positive constant $a_{0}$ is straightforward. See the
proof of Theorem 10.9 in \cite{GiTr} for details. }
\end{remark}

To prove Theorem \ref{DP}, we will apply Theorem \ref{th15.19} to a family
of truncations of the operator $\mathcal{Q}$ defined in (\ref{equation}).
For each $M>0$, let $\chi _{M}\in \mathcal{C}^{\infty }\left( \mathbb{R}%
\right) $ satisfy 
\begin{equation}
\chi _{M}\left( z\right) =\left\{ 
\begin{array}{ll}
z & \qquad \text{if }\left\vert z\right\vert \leq M \\ 
&  \\ 
\frac{3}{2}M & \qquad \text{if }z\geq 2M \\ 
&  \\ 
-\frac{3}{2}M & \qquad \text{if }z<2M%
\end{array}%
\right. ,\qquad \left\vert \frac{d}{dz}\chi _{M}\left( z\right) \right\vert
\leq 1.  \label{chi_M}
\end{equation}%
Define 
\begin{equation*}
\mathcal{A}^{M}\left( x,z\right) =\mathcal{A}\left( x,\chi _{M}\left(
z\right) \right) ,\quad \vec{\gamma}^{M}\left( x,z\right) =\vec{\gamma}%
\left( x,\chi _{M}\left( z\right) \right) ,\quad f^{M}\left( x,z\right)
=f\left( x,\chi _{M}\left( z\right) \right) ,
\end{equation*}%
and set 
\begin{equation*}
\mathcal{Q}^{M}w\left( x\right) =\mathop{\rm div}\mathcal{A}^{M}\left(
x,w\left( x\right) \right) \nabla w\left( x\right) +\vec{\gamma}^{M}\left(
x,w\left( x\right) \right) \cdot \nabla w\left( x\right) +f^{M}\left(
x,w\left( x\right) \right) .
\end{equation*}%
Note that if $\left\vert w\left( x\right) \right\vert \leq M$ then $\mathcal{%
Q}^{M}w\left( x\right) =\mathcal{Q}w\left( x\right) $.

\begin{proposition}
\label{proptouse} For each $\varepsilon ,M>0,${\ the} operators $\mathcal{Q}%
_{\varepsilon }^{M}=\mathcal{Q}^{M}+\varepsilon \mathbf{\Delta }$ satisfy
the hypothesis of Theorem \ref{th15.19} with{\ 
\begin{eqnarray*}
a_{0} &=&c_{0}=\varepsilon \\
b_{0} &=&\left\Vert \vec{\gamma}\right\Vert _{L^{\infty }\left( \tilde{\Gamma%
}\right) },\qquad b_{1}=\left\Vert f\right\Vert _{L^{\infty }\left( \tilde{%
\Gamma}\right) } \\
d_{0} &=&\left\Vert \nabla _{\left( x,z\right) }\mathcal{A}\right\Vert
_{L^{\infty }\left( \tilde{\Gamma}\right) }+\left\Vert \vec{\gamma}%
\right\Vert _{L^{\infty }\left( \tilde{\Gamma}\right) }+\left\Vert
f\right\Vert _{L^{\infty }\left( \tilde{\Gamma}\right) }+\varepsilon ,
\end{eqnarray*}%
where $\tilde{\Gamma}=\Omega \times \left[ -\frac{3}{2}M,\frac{3}{2}M\right]
.$ Moreover, since $\mathcal{A}^{M}$, $\vec{\gamma}^{M}$ and $f^{M}$ are
smooth functions, the value of ${\delta }$ in Theorem \ref{th15.19} can be
any value $0<\delta <1$}.
\end{proposition}

\begin{proof}
With the notation of Theorem \ref{th15.19}, and $\mathcal{A}\left(
x,z\right) =\left\{ a^{ij}\left( x,z\right) \right\} _{i,j=1}^{n}$, we have 
\begin{equation*}
\mathcal{Q}_{\varepsilon}^{M}\left( w\right) = \mathcal{M}\left( w\right) =%
\mathop{\rm div}\mathcal{S}\left( x,w,\nabla w\right) +\mathcal{T}\left(
x,w,\nabla w\right)
\end{equation*}
with 
\begin{eqnarray}
\mathcal{S}^{i}\left( x,z,p\right) &=&\sum_{j=1}^{n}\left( a^{ij}\left(
x,\chi _{M}\left( z\right) \right) +\varepsilon \delta _{ij}\right)
p_{j},\qquad i=1,\dots ,n,  \label{Sform} \\
\mathcal{T}\left( x,z,p\right) &=&\vec{\gamma}\left( x,\chi _{M}\left(
z\right) \right) \cdot p+f\left( x,\chi _{M}\left( z\right) \right) .
\label{Tform}
\end{eqnarray}
Here $\delta _{ij}=1$ if $i=j$ and $\delta _{ij}=0$ otherwise. Let us now
verify (\ref{15.64-1})--(\ref{10.23-b}). \newline
\textbf{(\ref{15.59-1}).} By (\ref{Sform}), 
\begin{eqnarray*}
\xi ^{t}\left\{ \nabla _{p}\mathcal{S}\left( x,z,p\right) \right\} \xi
&=&\sum_{i,\ell =1}^{n}\partial _{p_{\ell }}\left( \sum_{j=1}^{n}\left(
a^{ij}\left( x,\chi _{M}\left( z\right) \right) + \varepsilon \delta _{ij}
\right) p_{j}\right) \xi _{\ell }\xi _{i} \\
&=&\sum_{i,\ell =1}^{n}\left( a^{i\ell }\left( x,\chi _{M}\left( z\right)
\right) + \varepsilon \delta _{i\ell }\right) \xi _{\ell }\xi _{i} \\
&=&\xi ^{t}\mathcal{A}\left( x,\chi _{M}\left( z\right) \right) \xi
+\varepsilon \left\vert \xi \right\vert ^{2} \\
&\geq &\varepsilon \left\vert \xi \right\vert ^{2}.
\end{eqnarray*}
Hence, (\ref{15.59-1}) holds with $c_{0}=\varepsilon $ . \newline
\textbf{(\ref{15.64-1}).} From (\ref{Sform}), 
\begin{eqnarray*}
\left\vert \partial _{p_{j}}\mathcal{S}^{i}\left( x,z,p\right) \right\vert
&=&\left\vert a^{ij}\left( x,\chi _{M}\left( z\right) \right) +\varepsilon
\delta _{ij}\right\vert \\
&\leq &\left\vert a^{ij}\left( x,\chi _{M}\left( z\right) \right)
\right\vert +\varepsilon \\
&\leq &\left\Vert \mathcal{A}\right\Vert _{L^{\infty }\left( \tilde{\Gamma}
\right) }+\varepsilon ,
\end{eqnarray*}
where $\tilde{\Gamma}=\Omega \times \left[ -\frac{3}{2}M,\frac{3}{2}M \right]
$. Thus, (\ref{15.64-1}) holds with $d_{0}=\left\Vert \mathcal{A}
\right\Vert _{L^{\infty }\left( \tilde{\Gamma}\right) }+\varepsilon $. 
\newline
\textbf{(\ref{15.66-1}).} From (\ref{Sform}) and (\ref{Tform}), 
\begin{eqnarray*}
&&\left( 1+\left\vert p\right\vert \right) \left\vert \partial _{z}\mathcal{%
S }\right\vert +\left\vert \nabla _{x}\mathcal{S}\right\vert +\left\vert 
\mathcal{T}\right\vert \\
&=&\left( 1+\left\vert p\right\vert \right) \left\vert \tfrac{d}{dz}\chi
_{M}\left( z\right) \right\vert \left\vert \mathcal{A}_{z}\left( x,\chi
_{M}\left( z\right) \right) p\right\vert +\left\vert \nabla _{x}\mathcal{A}
\left( x,\chi _{M}\left( z\right) \right)p \right\vert \\
&&+\left\vert \vec{\gamma}\left( x,\chi _{M}\left( z\right) \right) \cdot
p+f\left( x,\chi _{M}\left( z\right) \right) \right\vert \\
&\leq &\left( \left\Vert \nabla _{\left( x,z\right) }\mathcal{A}\right\Vert
_{L^{\infty }\left( \tilde{\Gamma}\right) }+\left\Vert \vec{\gamma}
\right\Vert _{L^{\infty }\left( \tilde{\Gamma}\right) }+\left\Vert
f\right\Vert _{L^{\infty }\left( \tilde{\Gamma}\right) }\right) \left(
1+\left\vert p\right\vert \right) ^{2},
\end{eqnarray*}
where we used that $\left\vert \frac{d}{dz}\chi _{M}\left( z\right)
\right\vert \leq 1$. Thus (\ref{15.66-1}) holds with 
\begin{equation*}
d_{0}=\left\Vert \nabla _{\left( x,z\right) }\mathcal{A}\right\Vert
_{L^{\infty }\left( \tilde{\Gamma}\right) }+\left\Vert \vec{\gamma}
\right\Vert _{L^{\infty }\left( \tilde{\Gamma}\right) }+\left\Vert
f\right\Vert _{L^{\infty }\left( \tilde{\Gamma}\right) }.
\end{equation*}
\textbf{(\ref{10.23-a}).} By (\ref{Sform}) it follows that 
\begin{equation*}
p\cdot \mathcal{S}\left( x,z,p\right) =p\cdot \left( \mathcal{A}\left(
x,\chi _{M}\left( z\right) \right) +\mathbf{I}\varepsilon \right) p\geq
\varepsilon \left\vert p\right\vert ^{2}.
\end{equation*}
Thus (\ref{10.23-a}) holds with $a_{0}=\varepsilon $.\newline
\textbf{(\ref{10.23-b}).} By (\ref{Tform}), 
\begin{eqnarray*}
\mathcal{T}\left( x,z,p\right) \,\mathop{\rm sign}z &=&\left( \vec{ \gamma}%
\left( x,\chi _{M}\left( z\right) \right) \cdot p\,+f\left( x,\chi
_{M}\left( z\right) \right) \right) \,\mathop{\rm sign}z \\
&\leq &\left\Vert \vec{\gamma}\right\Vert _{L^{\infty }\left( \tilde{\Gamma}
\right) }~\left\vert p\right\vert +\left\Vert f\right\Vert _{L^{\infty
}\left( \tilde{\Gamma}\right) }.
\end{eqnarray*}
\newline
Hence (\ref{10.23-b}) holds with $b_{0}=\left\Vert \vec{\gamma}\right\Vert
_{L^{\infty }\left( \tilde{\Gamma}\right) }$ and $b_{1}=\left\Vert
f\right\Vert _{L^{\infty }\left( \tilde{\Gamma}\right) }$.
\end{proof}

\subsection{Proof of Theorem \protect\ref{DP}}

Let $\Omega \Subset \widetilde{\Omega }$ be a strongly convex domain as in
the hypotheses of Theorem \ref{DP}. Then $\Omega $ satisfies an exterior {%
sphere condition at each point of $\partial \Omega $}. Given a continuous
function $\varphi $ on $\partial \Omega $, let $M_{0}=\sup_{\partial \Omega
}\left\vert \varphi \right\vert $. {By Theorem \ref{th15.19} and Proposition %
\ref{proptouse} with $M=M_{0}$, for all $\varepsilon >0$ there exists $%
w^{\varepsilon }\in \mathcal{C}^{2+\delta }\left( \Omega \right) \bigcap 
\mathcal{C}^{0}\left( \overline{{\Omega }}\right) $ , $0<\delta <1$, such
that $w^{\varepsilon }$ is a solution of the Dirichlet problem} 
\begin{equation*}
\left\{ 
\begin{array}{rcll}
\mathcal{Q}_{\varepsilon }^{M_{0}}w & = & 0\quad & \text{in }\Omega \\ 
w & = & \varphi & \text{on }\partial \Omega .%
\end{array}%
\right.
\end{equation*}%
{\ The solution $w^{\varepsilon }$ also depends on $M_{0}$, which is fixed.
The smoothness assumptions on $\mathcal{A}$, $\vec{\gamma}$ and $f$ and a
standard bootstrapping argument (see, e.g., Theorems 6.2 and 6.3 in \cite{Ev}%
) imply that }$w^{\varepsilon }${$\in \mathcal{\ C}^{\infty }\left( \Omega
\right) \bigcap \mathcal{C}^{0}\left( \overline{\Omega }\right) $. Since the
operators $\mathcal{Q}_{\varepsilon }^{M_{0}}$ satisfy the hypotheses of
Theorem \ref{maxp} in the Appendix (the maximum principle), we have 
\begin{equation}
\left\Vert w^{\varepsilon }\right\Vert _{L^{\infty }\left( \Omega \right)
}\leq \left\Vert w^{\varepsilon }\right\Vert _{L^{\infty }\left( \partial
\Omega \right) }=\left\Vert \varphi \right\Vert _{L^{\infty }\left( \partial
\Omega \right) }=M_{0}.  \label{mp000}
\end{equation}%
Thus the functions }$w^{\varepsilon }${\ are uniformly bounded by $M_{0}$ in 
}$\overline{{\Omega }}${\ for all $\varepsilon >0$. From the definition of }$%
\mathcal{Q}_{\varepsilon }^{M_{0}}$ it follows that{\ } $\mathcal{Q}%
w^{\varepsilon }+\varepsilon \triangle w^{\varepsilon }=\mathcal{Q}%
_{\varepsilon }^{M_{0}}w^{\varepsilon }=0$ in $\Omega $.

{By (\ref{hellip}), the coefficients of $\mathcal{Q}+\varepsilon \mathbf{I}$%
, namely the entries of $\mathcal{A}_{\varepsilon }=\mathcal{A+\varepsilon }%
\mathbf{I}$, satisfy 
\begin{equation*}
\sum_{i=1}^{n}\left( k^{i}\left( x,z\right) +\varepsilon \right) \xi
_{i}^{2}\leq \xi ^{t}\mathcal{A}_{\varepsilon }\left( x,z\right) \xi \leq
\Lambda \sum_{i=1}^{n}\left( k^{i}\left( x,z\right) +\varepsilon \right) \xi
_{i}^{2}
\end{equation*}%
for all $\xi \in \mathbb{R}^{n}$ and $\left( x,z\right) \in \Gamma =\Omega
\times \mathbb{R}$. That is, $\mathcal{A}_{\varepsilon }$} satisfies the
diagonal condition for diagonal entries $k^{i}+\varepsilon $. Next, by (\ref%
{wirtm}), {\ 
\begin{equation*}
\sum_{i=1}^{n}\left\vert \partial _{i}\mathcal{A}_{\varepsilon }\left(
x,z\right) \xi \right\vert ^{2}+\left\vert \partial _{z}\mathcal{A}%
_{\varepsilon }\left( x,z\right) \xi \right\vert ^{2}\leq B_{\mathcal{A}%
}^{2}~\xi ^{t}\mathcal{A}\xi \leq B_{\mathcal{A}}^{2}~\xi ^{t}\mathcal{A}%
_{\varepsilon }\xi
\end{equation*}%
for all $\xi \in \mathbb{R}^{n}$, $\left( x,z\right) \in \Gamma
_{M_{0}}^{\prime }$. Hence }$\mathcal{A}_{\varepsilon }$ is subordinate,
with the same constant\ $B_{\mathcal{A}}$ as for $\mathcal{A}$ in $\Gamma
_{M_{0}}^{\prime }$. We also have by (\ref{xtra}) and (\ref{xxtra}) that {\ 
\begin{eqnarray*}
\left\vert \partial _{z}\mathcal{A}_{\varepsilon }\left( x,z\right) \xi
\right\vert ^{2} &=&\left\vert \partial _{z}\mathcal{A}\left( x,z\right) \xi
\right\vert ^{2}\leq B_{\mathcal{A}}^{2}~k^{\ast }\left( x,z\right) \,\xi
^{t}\mathcal{A}\left( x,z\right) \xi \\
&\leq &CB_{\mathcal{A}}^{2}~\left( k^{\ast }\left( x,z\right) +\varepsilon
\right) \,\xi ^{t}\mathcal{A}_{\varepsilon }\left( x,z\right) \xi
\end{eqnarray*}%
and} 
\begin{eqnarray*}
\sum_{i=1}^{n}\left\vert \partial _{i}\partial _{z}\mathcal{A}_{\varepsilon
}\left( x,z\right) \xi \right\vert ^{2}+\left\vert \partial _{z}^{2}\mathcal{%
\ A}_{\varepsilon }\left( x,z\right) \xi \right\vert ^{2} &\leq &\left( B_{%
\mathcal{A}}^{\prime }\right) ^{2}\,\xi ^{t}\mathcal{A}\left( x,z\right) \xi
\\
&\leq &\left( B_{\mathcal{A}}^{\prime }\right) ^{2}\,\xi ^{t}\mathcal{A}%
_{\varepsilon }\left( x,z\right) \xi
\end{eqnarray*}%
{for all $\xi \in \mathbb{R}^{n}$, $\left( x,z\right) \in \Gamma
_{M_{0}}^{\prime }$.}

Thus $\mathcal{A}_{\varepsilon }$ satisfies the super subordination
condition with the same constants as $\mathcal{A}$. Hence $\mathcal{Q}%
+\varepsilon $ satisfies the hypotheses of Theorem \ref{apriorial} uniformly
in $\varepsilon $, $0\leq \varepsilon \leq 1${. Applying Theorem \ref%
{apriorial} to the solutions $w^{\varepsilon }$, it follows that for any
multi-index $\vec{\alpha}$ of nonnegative integers, the family $\left\{ D^{%
\vec{\alpha}}w^{\varepsilon },0<\varepsilon \leq 1\right\} $ is
equicontinuous and uniformly bounded in any subdomain }${\Omega }${$^{\prime
}\Subset $}${\Omega }${. By the Arzela-Ascoli theorem, there is a
subsequence $\left\{ w^{\varepsilon _{i}}\right\} $ with $\varepsilon
_{i}\rightarrow 0$ which converges in $\mathcal{C}^{\infty }\left( \Omega
\right) $ to a function $w^{0}\in \mathcal{C}^{\infty }\left( \Omega \right) 
$, i.e., $D^{\vec{\alpha}}w^{\varepsilon _{i}}\,$converges to $D^{\vec{\alpha%
}}w^{0}\,$uniformly on compact subsets of $\Omega $, for all multi-indexes $%
\vec{\alpha}$. We will show that $w^{0}$ is a solution of the Dirichlet
problem (\ref{dirchlet}). }

{Since $w^{0}$ and all its derivatives are uniform limits of $w^{\varepsilon
_{i}}$ and their corresponding derivatives in compact subsets of $\Omega $,
then $\left\vert \bigtriangleup w^{0}\right\vert <\infty $ in $\Omega $, and
for all $x\in \Omega $, 
\begin{eqnarray*}
\mathcal{Q}w^{0}\left( x\right) &=&\mathop{\rm div}\mathcal{A}\left(
x,w^{0}\right) \nabla w^{0}+\vec{\gamma}\left( x,w^{0}\right) \cdot \nabla
w^{0}+f\left( x,w^{0}\right) \\
&=&\lim_{i\rightarrow \infty }\mathop{\rm div}\mathcal{A}\left(
x,w^{\varepsilon _{i}}\right) \nabla w^{\varepsilon _{i}}+\vec{\gamma}\left(
x,w^{\varepsilon _{i}}\right) \cdot \nabla w^{\varepsilon _{i}}+f\left(
x,w^{\varepsilon _{i}}\right) \\
&=&\lim_{i\rightarrow \infty }\mathcal{Q}_{\varepsilon
^{i}}^{M_{0}}w^{\varepsilon _{i}}\left( x\right) -\varepsilon
_{i}\bigtriangleup w^{\varepsilon _{i}}\left( x\right) \\
&=&-\varepsilon _{i}\lim_{i\rightarrow \infty }\bigtriangleup w^{\varepsilon
_{i}}\left( x\right) =0.
\end{eqnarray*}%
Therefore }$w^{0}\in \mathcal{C}^{\infty }\left( \Omega \right) $ is a
strong solution of the differential equation in the Dirichlet problem (\ref%
{dirchlet}). Define $w^{0}=\varphi $ on $\partial \Omega $ and recall that $%
w^{\varepsilon }=\varphi $ on $\partial \Omega $ if $\varepsilon >0$. To
finish the proof of the theorem we must check that $w^{0}\in \mathcal{C}^{0}(%
\overline{\Omega })$.

Let $\omega \left( r\right) $ be the modulus of continuity of ${\varphi }$
on $\partial \Omega $: 
\begin{equation*}
\omega \left( r\right) =\sup_{x,y\in \partial \Omega ,\left\vert
x-y\right\vert \leq r}\left\vert {\varphi }\left( x\right) -{\varphi }\left(
y\right) \right\vert .
\end{equation*}%
Then $\omega $ is continuous and $\omega \left( 0\right) =0$. By Lemma \ref%
{concave-majorant}, taking a bigger $\omega $ if necessary, we may assume
that $\omega $ is also concave and strictly increasing in $\left[ 0,%
\mathop{\rm diam}\Omega \right] $, and $\mathcal{C}^{2}$ in $\left( 0,%
\mathop{\rm diam}\Omega \right] $.

By our hypothesis on $\gamma $, there exists $\eta _{0}>0$ such that $\gamma
\left( x,z\right) =0$ if $x\in \Omega $, $\mathop{\rm dist}\left( x,\partial
\Omega \right) <\eta _{0}$ and $\left\vert z\right\vert \leq M_{0}$. {Let $%
x_{0}$ be an arbitrary point on $\partial \Omega $, and let $h\left(
x\right) $ be the barrier function for $\omega $ at $x_{0}$ provided by
Lemma \ref{barrier} in the Appendix, with $\Phi $ and $\Omega $ there chosen
to be $\Omega $ and $\tilde{\Omega}$ respectively, and with $\nu =2M_{0}$, }$%
m_{0}=2M_{0}$, $\eta =\eta _{0}$ and $K=\left\Vert f\right\Vert _{L^{\infty
}\left( \tilde{\Gamma}\right) }$, where ${\tilde{\Gamma}}${$=\overline{%
\Omega }\times \left[ -2M_{0},2M_{0}\right] $. {Thus } there is a
neighborhood $\mathcal{N}$ of $x_{0}$ with $\mathcal{N}\subset
\{|x-x_{0}|<\eta _{0}\}$ and a function $h\in \mathcal{C}^{\infty }\left( 
\mathcal{N}\right) \bigcap \mathcal{C}^{0}\left( \overline{\mathcal{N}}%
\right) $ such that\ 
\begin{eqnarray}
h\left( x\right) &\leq &-\omega \left( \left\vert x-x_{0}\right\vert \right)
,  \label{mp001} \\
\mathop{\rm div}\mathcal{A}\left( x,h\left( x\right) +m\right) \nabla h
&\geq &\left\Vert f\right\Vert _{L^{\infty }\left( \tilde{\Gamma}\right) },
\label{mp002} \\
\bigtriangleup h &=&\sum_{i=1}^{n}\partial _{i}^{2}h>0  \label{mp003}
\end{eqnarray}%
for all $x\in \Omega \bigcap \mathcal{N}$ and }$\left\vert {m}\right\vert
\leq 2M_{0}${. Moreover, 
\begin{eqnarray}
h\left( x\right) &\leq &-2M_{0}\text{ \qquad if }x\in \partial \mathcal{N}%
\bigcap \Omega ,  \label{mp004} \\
h\left( x_{0}\right) &=&0.  \label{mp005}
\end{eqnarray}%
Now, by (\ref{mp001}) and the continuity of $h$ on $\overline{\mathcal{N}}$, 
\begin{equation}
h\left( x\right) \leq -\omega \left( \left\vert x-x_{0}\right\vert \right)
\leq {\varphi }\left( x\right) -{\varphi }\left( x_{0}\right)
=w^{\varepsilon }\left( x\right) -{\varphi }\left( x_{0}\right) \qquad \text{%
if }x\in \overline{\mathcal{N}}\bigcap \partial \Omega ,\varepsilon >0.
\label{out}
\end{equation}%
By (\ref{mp004}) and (\ref{mp000}), 
\begin{equation*}
h\left( x\right) \leq -2M_{0}\leq w^{\varepsilon }\left( x\right) -{\varphi }%
\left( x_{0}\right) ,\qquad \text{if }x\in \partial \mathcal{N}\bigcap
\Omega ,\varepsilon >0.
\end{equation*}%
Therefore, 
\begin{equation}
h\left( x\right) +\varphi \left( x_{0}\right) \leq w^{\varepsilon }\left(
x\right) \qquad \text{if }x\in \partial \mathcal{N},\varepsilon >0.
\label{bdryh}
\end{equation}%
\ On the other hand, since $w^{\varepsilon }$ is a solution of $\mathcal{Q}%
_{\varepsilon }^{M_{0}}w^{\varepsilon }=0$ and }$\mathcal{N}\subset \left\{
\left\vert x-x_{0}\right\vert <\eta _{0}\right\} $, {\ 
\begin{equation*}
\mathop{\rm div}\mathcal{A}^{\varepsilon }\left( x,w^{\varepsilon }\right)
\nabla w^{\varepsilon }=-f\left( x,w^{\varepsilon }\right) \leq \left\Vert
f\right\Vert _{L^{\infty }\left( \tilde{\Gamma}\right) }\qquad \text{in }%
\mathcal{N}\bigcap \Omega ,
\end{equation*}%
where the last inequality follows from (\ref{mp000}). Thus, letting $%
\mathcal{L}_{\varepsilon }$ be the quasilinear operator }$\mathcal{L}%
_{\varepsilon }=\mathop{\rm div}\mathcal{A}^{\varepsilon }\left( x,\cdot
\right) \nabla $, we have by (\ref{mp002}) and (\ref{bdryh}) that{\ 
\begin{eqnarray*}
\mathcal{L}_{\varepsilon }\left( h+\varphi \left( x_{0}\right) \right) &\geq
&\mathcal{L}_{\varepsilon }w^{\varepsilon }\qquad \text{in }\mathcal{N}%
\bigcap \Omega , \\
h\left( x\right) +\varphi \left( x_{0}\right) &<&w^{\varepsilon }\left(
x\right) \qquad \text{if }x\in \partial \left( \mathcal{N}\bigcap \Omega
\right) .
\end{eqnarray*}%
From the comparison principle Lemma \ref{comparison} applied to $\mathcal{L}%
_{\varepsilon }$ and the functions $w^{\varepsilon }$, $h+\varphi \left(
x_{0}\right) $ in $\mathcal{N}\bigcap \Omega $, we obtain 
\begin{equation}
h\left( x\right) \leq w^{\varepsilon }\left( x\right) -\varphi \left(
x_{0}\right) \qquad \text{if }x\in \mathcal{N}\bigcap \Omega .
\label{mpalleps}
\end{equation}%
Since $h$ is continuous in $\overline{\mathcal{N}}$ and $h\left(
x_{0}\right) =0$, given any $\sigma >0$, there exists $\delta _{0}>0$
independent of $\varepsilon $ such that 
\begin{equation*}
-\sigma <w^{\varepsilon }\left( x\right) -\varphi \left( x_{0}\right) \qquad 
\text{if }x\in \mathcal{N}\bigcap \left\{ \left\vert x-x_{0}\right\vert
<\delta _{0}\right\} \bigcap \Omega ,\quad \varepsilon >0.
\end{equation*}%
Proceeding in a similar fashion for the function $\varphi \left(
x_{0}\right) -w^{\varepsilon }\left( x\right) $, we obtain 
\begin{equation}
\left\vert w^{\varepsilon }\left( x\right) -\varphi \left( x_{0}\right)
\right\vert <\sigma \qquad \text{if }x\in \mathcal{N}\bigcap \left\{
\left\vert x-x_{0}\right\vert <\delta _{0}\right\} \bigcap \Omega ,\quad
\varepsilon >0.  \label{close}
\end{equation}%
}

Let us now show that $w^{0}$ is continuous on $\overline{\Omega }$. Suppose
not, and let $x_{0}$ be a point of discontinuity of $w^{0}$ in $\overline{%
\Omega }$. Since $w^{0}$ is smooth in $\Omega $, $x_{0}$ must lie on $%
\partial \Omega $. Since $w^{0}=\varphi $ on $\partial \Omega $ and $\varphi 
$ is continuous by hypothesis, there exist points $\{x_{k}\}_{k=1}^{\infty
}\subset \Omega $ and $\tilde{\sigma}>0$ such that $x_{k}\rightarrow x_{0}$
and $|w^{0}(x_{k})-\varphi (x_{0})|\geq \tilde{\sigma}$ for all $k$. By (\ref%
{close}) with $\sigma =\tilde{\sigma}/2$, there exists $\delta >0$
independent of $\varepsilon $ such that $|w^{\varepsilon }(x)-\varphi
(x_{0})|\leq \tilde{\sigma}/2$ if $x\in \Omega $ and $|x-x_{0}|<\delta $.
Choose $x=x_{k_{0}}$ for $k_{0}$ so large that $|x_{k_{0}}-x_{0}|<\delta $.
Then $|w^{\varepsilon }(x_{k_{0}})-\varphi (x_{0})|\leq \tilde{\sigma}/2$
for all $\varepsilon $. However, $w^{\varepsilon }(x_{k_{0}})\rightarrow
w^{0}(x_{k_{0}})$ as $\varepsilon \rightarrow 0$, so $|w^{0}(x_{k_{0}})-%
\varphi (x_{0})|\leq \tilde{\sigma}/2$, which is a contradiction. Hence $%
w^{0}\in \mathcal{C}^{\infty }\left( \Omega \right) \bigcap \mathcal{C}%
^{0}\left( \overline{\Omega }\right) $ and $w^{0}$ is a solution to the
Dirichlet problem (\ref{dirchlet}).

When $\gamma \equiv 0$, uniqueness follows by Lemma \ref{comparison} in the
Appendix. This finishes the proof of Theorem \ref{DP}. \endproof

\subsection{Proof of Theorem \protect\ref{application}}

Under the hypotheses of Theorem \ref{application}, let $w$ be a continuous
weak solution of 
\begin{equation*}
\mathop{\rm div}\mathcal{A}\left( x,w\right) \nabla w+f\left( x,w\right)
=0\qquad \text{in }\Omega .
\end{equation*}%
Given $\bar{x}\in \Omega $, let $\Phi $ be the ball centered at $\bar{x}$
with radius $r=\frac{1}{2}\mathop{\rm dist}\left( \bar{x},\partial \Omega
\right) $. Then $\Phi $ is strongly convex. By Theorem \ref{DP}, there is a
continuous strong solution $u$ of the Dirichlet problem 
\begin{equation*}
\left\{ 
\begin{array}{rcll}
\mathcal{Q}u & = & 0\quad & \text{in }\Phi \\ 
u & = & w & \text{on }\partial \Phi .%
\end{array}%
\right.
\end{equation*}%
{Moreover, }$u\in \mathcal{C}^{0}\left( \overline{\Phi }\right) \bigcap 
\mathcal{C}^{\infty }\left( \Phi \right) $. By restricting $\bar{x}$ to a
compact set $\Omega ^{\prime }\subset \Omega $, the convex character $%
\lambda _{0}$ of $\Phi $ is bounded below away from zero, the bound
depending on $\mathop{\rm dist}(\Omega ^{\prime },\partial \Omega )$, and
hence the constants $\mathcal{C}_{N}$ controlling the derivatives are
independent of $\lambda _{0}$. By the uniqueness part of the comparison
principle, Lemma \ref{comparison}, it follows that $u=w$ in $\overline{\Phi }
$ and therefore $w$ is smooth inside $\Phi $ with control on all its
derivatives in compact subsets of $\Phi $ . {\ This finishes the proof of
Theorem \ref{application}.\endproof}

\section{Appendix}

This Appendix is divided into four subsections in which we give some
technical details about facts that we used earlier: degenerate Sobolev
spaces and weak solutions, a maximum principle, a comparison principle, and
barriers for the Dirichlet problem.

\label{deg-sob-section}

\subsection{Degenerate Sobolev Spaces and Weak Solutions}

\subsubsection{The weak degenerate Sobolev space $H_{\mathcal{X}%
}^{1,2}(\Omega) \label{Section-Weak-Sol}$}

We first describe the degenerate Sobolev spaces used in the paper, beginning
with a standard definition.

\begin{definition}[Weak $X$ derivative]
Let $X$ be a locally Lipschitz vector field on $\Omega $, i.e., $X=\mathbf{v}%
\cdot \nabla $ with $\mathbf{v}\in \mathop{\rm Lip_{loc}}\left( \Omega
\right) $, the class of locally Lipschitz continuous $\mathbb{R}^{n}-$valued
functions on $\Omega $. $X$ is initially defined on real-valued functions $%
w\in \mathop{\rm Lip_{loc}}(\Omega )$ by $Xw=\mathbf{v}\cdot \nabla w$. We
say that a locally integrable function $g$ is the weak derivative $Xw$ of a
locally integrable function $w$ if 
\begin{equation}
\int_{\Omega }g\varphi =-\int_{\Omega }wX^{\prime }\varphi =-\int_{\Omega
}w\nabla \cdot \left( \mathbf{v}\varphi \right) \quad \text{for all $\varphi
\in \mathop{\rm Lip_{0}}\left( \Omega \right) $.}  \label{weakdef}
\end{equation}
\end{definition}

The weak derivative $Xw$ is clearly unique if it exists, and $Xw$ exists and
coincides with $\mathbf{v}\cdot \nabla w$ if $w\in \mathop{\rm Lip_{loc}}%
\left( \Omega \right) $.

\begin{definition}[Weak degenerate Sobolev space]
Let $\mathcal{X}=\left\{ X_{j}\right\} _{j=1}^{m}$ where $X_{j}=\mathbf{v}_j$
are $\mathop{\rm Lip_{loc}}(\Omega)$ vector fields on $\Omega \subset 
\mathbb{R}^{n}.$ The degenerate Sobolev space $H_{\mathcal{X}}^{1,2}\left(
\Omega \right) $ is defined as the inner product space consisting of all $%
w\in L^{2}\left( \Omega \right) $ whose weak derivatives $X_{j}w$ are also
in $L^{2}\left( \Omega \right) $. The inner product in $H_{\mathcal{X}%
}^{1,2}\left( \Omega \right) $ is defined by 
\begin{equation}
\left\langle w,v\right\rangle _{\mathcal{X}}=\int_{\Omega
}wv~dx+\int_{\Omega }\mathcal{X}w\cdot \mathcal{X}v~dx,  \label{innpdt}
\end{equation}
where we denote $\mathcal{X}w= \left(X_1w,\dots,X_mw\right)$, and the norm
is 
\begin{equation*}
\left\Vert w\right\Vert _{H_{\mathcal{X}}^{1,2}\left( \Omega \right)
}=\left\langle w,w\right\rangle _{\mathcal{X}}^{1/2} = \left(\left\Vert
w\right\Vert _{L^{2}\left( \Omega \right) }^{2}+\left\Vert \mathcal{X}%
w\right\Vert _{L^{2}\left( \Omega \right) }^{2}\right)^{1/2}.
\end{equation*}
\end{definition}

We now define what we mean by $\nabla w$ if $w \in H_{\mathcal{X}%
}^{1,2}\left( \Omega \right) $ for a collection $\mathcal{X} =
\{X_j\}_{j=1}^m = \{\mathbf{v}_j\cdot \nabla\}_{j=1}^m$ of $%
\mathop{\rm
Lip_{loc}}(\Omega)$ vector fields. For such $w$ and $\mathcal{X}$, there is
a sequence $\{w_k\}_{k=1}^\infty$ of $\mathop{\rm Lip}(\Omega)$ functions
and a vector $\vec{W}(x) \in \mathbb{R}^n$ satisfying $\mathbf{v}_j\cdot 
\vec{W} \in L^2(\Omega)$ for all $j$ and 
\begin{equation}  \label{density}
||w_k -w||_{L^2(\Omega)} + \sum_j ||X_jw_k -\mathbf{v}_j\cdot \vec{W}
||_{L^2(\Omega)} \rightarrow 0\quad \text{as $k \to \infty$}.
\end{equation}
This is proved in \cite{SaW3} (see also \cite{FSS}, \cite{GN}) in case all $%
\mathbf{v}_j \in \mathop{\rm Lip}(\Omega)$ but remains true if all $\mathbf{v%
}_j \in \mathop{\rm Lip_{loc}}(\Omega)$ by examining the proof in \cite{SaW3}%
.

Then 
\begin{equation}
\mathcal{X}w=\left( X_{1}w,\dots ,X_{m}w\right) =\left( \mathbf{v}_{1}\cdot 
\vec{W},\dots \mathbf{v}_{m}\cdot \vec{W}\right)  \label{Xform}
\end{equation}%
since for all $\varphi \in \mathop{\rm Lip_{0}}\left( \Omega \right) $, 
\begin{eqnarray*}
\int_{\Omega }w\,\nabla \cdot (\mathbf{v}_{j}\varphi ) &=&\lim_{k\rightarrow
\infty }\int_{\Omega }w_{k}\nabla \cdot (\mathbf{v}_{j}\varphi ) \\
&=&\lim_{k\rightarrow \infty }\int_{\Omega }(\mathbf{v}_{j}\cdot \nabla
w_{k})\varphi =\lim_{k\rightarrow \infty }\int_{\Omega }(X_{j}w_{k})\varphi
\\
&=&-\int_{\Omega }\left( \mathbf{v}_{j}\cdot \vec{W}\right) \varphi .
\end{eqnarray*}%
Moreover, if $\{w_{k}^{\prime }\}$ and $\vec{W^{\prime }}$ are another such
sequence and vector for the same $w$, it follows similarly that $%
\int_{\Omega }w\,\nabla \cdot (\mathbf{v}_{j}\varphi )=-\int_{\Omega }\left( 
\mathbf{v}_{j}\cdot \vec{W^{\prime }}\right) \varphi $. Hence 
\begin{equation*}
\int_{\Omega }\left( \mathbf{v}_{j}\cdot \vec{W}\right) \varphi
=\int_{\Omega }\left( \mathbf{v}_{j}\cdot \vec{W^{\prime }}\right) \varphi
\end{equation*}%
for all $\varphi \in \mathop{\rm Lip_{0}}\left( \Omega \right) $, so that 
\begin{equation}
\mathbf{v}_{j}\cdot \vec{W}=\mathbf{v}_{j}\cdot \vec{W^{\prime }}\quad \text{%
for all $j$}.  \label{uniquegrad}
\end{equation}%
In this sense, $\vec{W}$ is unique, i.e., $\vec{W}$ is uniquely determined
by $w$ up to its dot product with each vector $\mathbf{v}_{j}$. We will
often abuse notation by writing $\vec{W}=\nabla w$. Any particular $\vec{W}$
as above with be called a \emph{representative} of $\nabla w$. Then $X_{j}w=%
\mathbf{v}_{j}\cdot \nabla w,j=1,\dots ,m,$ for all $w\in H_{\mathcal{X}%
}^{1,2}\left( \Omega \right) .$ Furthermore, by (\ref{density}), the
sequence $\{w_{k}\}$ above satisfies 
\begin{equation*}
||w_{k}-w||_{H_{\mathcal{X}}^{1,2}\left( \Omega \right) }\rightarrow 0\quad 
\text{as $k\rightarrow \infty $.}
\end{equation*}

Suppose that $\Omega ^{\prime }\subset \Omega $ and $M>0$, and let $\mathcal{%
X}=\{X_{j}\}=\{\mathbf{v}_{j}\cdot \nabla \}$ and $H_{\mathcal{X}%
,0}^{1,2}\left( \Omega \right) $ be as above. We claim that if $\mathcal{A}%
(x,z)$ and $\vec{\gamma}(x,z)$ satisfy 
\begin{equation}
\xi \cdot \mathcal{A}(x,z)\xi \leq c\sum_{j}(\mathbf{v}_{j}(x)\cdot \xi
)^{2}\quad \text{and}\quad (\vec{\gamma}(x,z)\cdot \xi )^{2}\leq c\sum_{j}(%
\mathbf{v}_{j}(x)\cdot \xi )^{2}  \label{subu}
\end{equation}%
for all $(x,z,\xi )\in \Omega ^{\prime }\times (-M,M)\times \mathbb{R}^{n}$,
then $\sqrt{\mathcal{A}(x,z)}\nabla w$ and $\vec{\gamma}(x,z)\cdot \nabla w$
are well-defined for any $(x,z)\in \Omega ^{\prime }\times (-M,M)$ and any $%
w\in H_{\mathcal{X},0}^{1,2}\left( \Omega \right) ,$ i.e., that if $\vec{W}$
and $\vec{W^{\prime }}$ are any two representatives of $\nabla w$, then $%
\sqrt{\mathcal{A}(x,z)}\vec{W}(x)=\sqrt{\mathcal{A}(x,z)}\vec{W^{\prime }}%
(x) $ and $\vec{\gamma}(x,z)\cdot \vec{W}(x)=\vec{\gamma}(x,z)\cdot \vec{%
W^{\prime }}(x)$. In fact, since $\mathbf{v}_{j}\cdot (\vec{W}-\vec{%
W^{\prime }})=\mathbf{v}_{j}\cdot \vec{W}-\mathbf{v}_{j}\cdot \vec{W^{\prime
}}=0$ for all $j$ by (\ref{uniquegrad}), this follows immediately from (\ref%
{subu}) by choosing $\xi =\vec{W}(x)-\vec{W^{\prime }}(x)$. \newline

Let {$\vec{k}\left( x,z\right) =$}$\left( k^{i}\left( x,z\right) \right)
_{i=1,\cdots ,n}$ and $\mathcal{A}\left( x,z\right) $ be as in Theorem \ref%
{DP}, that is, with $\Gamma =\widetilde{\Omega} \times \mathbb{R}$,

\begin{enumerate}
\item {$\vec{k}\left( x,z\right) \in \mathcal{C}^2(\Gamma)$ and satisfies }%
the nondegeneracy Condition{\ \ref{hyp1} in $\Gamma$, }

\item {$\mathcal{A} \in \mathcal{C}^\infty(\Gamma)$ and satisfies the
diagonal }Condition{\ \ref{hellipp} in $\Gamma $, }

\item {$\mathcal{A}$ satisfies }the super subordination Condition{\ \ref%
{hyp2} in $\Gamma $.}
\end{enumerate}

The particular vector fields that we will use are 
\begin{equation}
\partial _{1},\sqrt{k^{2}(x,0)}\,\partial _{2},\dots ,\sqrt{k^{n}(x,0)}%
\,\partial _{n}.  \label{ourvf}
\end{equation}%
We claim that since $k^{2}(x,0),\dots k^{n}(x,0)\in \mathcal{C}^{2}(\Omega )$
and are nonnegative, the Wirtinger inequality (\ref{genWirt}) implies that
the vector fields (\ref{ourvf}) belong to $\mathop{\rm Lip_{loc}}(\Omega )$.
To see why, fix $i$ and denote $k^{i}(x,0)=k(x)$. For a Euclidean ball $%
D\Subset \Omega $, $\varepsilon >0$ and all $x_{1},x_{2}\in D$, we have 
\begin{eqnarray*}
\left\vert \sqrt{k(x_{1})+\varepsilon }-\sqrt{k(x_{2})+\varepsilon }%
\right\vert &\leq &\left\Vert \nabla \sqrt{k+\varepsilon }\right\Vert
_{L^{\infty }(D)}|x_{1}-x_{2}| \\
&=&\left\Vert \frac{\nabla k}{\sqrt{k+\varepsilon }}\right\Vert _{L^{\infty
}(D)}|x_{1}-x_{2}| \\
&\leq &C_{D}\left\Vert \frac{\sqrt{k}}{\sqrt{k+\varepsilon }}\right\Vert
_{L^{\infty }(D)}|x_{1}-x_{2}|\quad \mbox{by (\ref{genWirt}),}
\end{eqnarray*}%
where $C_{D}$ depends on $k$ and dist$(D,\partial \Omega )$. Hence 
\begin{equation*}
\left\vert \sqrt{k(x_{1})+\varepsilon }-\sqrt{k(x_{2})+\varepsilon }%
\right\vert \leq C_{D}|x_{1}-x_{2}|,\quad x_{1},x_{2}\in D,
\end{equation*}%
uniformly in $\varepsilon $. Letting $\varepsilon \rightarrow 0$ and using
the Heine-Borel theorem to cover any compact subset of $\Omega $ by a finite
number of balls proves our claim.

\subsubsection{$\mathcal{X}$-weak solutions of quasilinear equations\label%
{Subsubsection-weak-solution}}

Here we make precise the notion of a ``weak solution'' of the quasilinear
equation (\ref{equation}). For $k^{i}\left( x,z\right) $ as in the
hypotheses of our main results Theorems \ref{application} and \ref{DP}, we
let $\mathcal{X} =\left\{ X_{j}\right\} _{j=1}^{n}$ with $X_{j}=\sqrt{%
k^{i}\left( x,0\right) }\frac{\partial }{\partial x_{j}}$. An analogous
definition can be given for any collection of locally Lipschitz vector
fields.

\begin{definition}
\label{Xweak}A function $w\in H_{\mathcal{X}}^{1,2}\left( \Omega \right)
\bigcap L_{\mathop{\rm loc}}^{\infty }\left( \Omega \right) $ is a weak
solution of 
\begin{equation}
\mathcal{Q}w=\mathop{\rm div}\mathcal{A}\left( x,w\right) \nabla w+\vec{
\gamma}\left( x,w\right) \cdot \nabla w+f\left( x,w\right) =0\qquad \text{in 
}\Omega  \label{equation-weak}
\end{equation}
if for all $u\in \mathop{\rm Lip}_0(\Omega)$, 
\begin{equation}
\int_{\Omega }\left( \nabla u\right) ^{t}\mathcal{A}\left( x,w\right) \nabla
w~dx=\int_{\Omega }u~\vec{\gamma}\left( x,w\right) \cdot \nabla
w~dx+\int_{\Omega }u~f\left( x,w\right) ~dx.  \label{weak-eq}
\end{equation}
Given $w_{0}$, $w_{1}\in H_{\mathcal{X}}^{1,2}\left( \Omega \right) \bigcap
L_{\mathop{\rm loc}}^{\infty }\left( \Omega \right),$ we say that $\mathcal{Q%
}w_{1}\geq \mathcal{Q}w_{0}$ in $\Omega$ if 
\begin{eqnarray*}
&&\int_{\Omega }\left( \nabla u\right) ^{t}\mathcal{A}\left( x,w_{1}\right)
\nabla w_{1}~dx-\int_{\Omega }u~\vec{\gamma}\left( x,w_{1}\right) \cdot
\nabla w_{1}~dx-\int_{\Omega }u~f\left( x,w_{1}\right) ~dx \\
&\leq &\int_{\Omega }\left( \nabla u\right) ^{t}\mathcal{A}\left(
x,w_{0}\right) \nabla w_{0}~dx-\int_{\Omega }u~\vec{\gamma}\left(
x,w_{0}\right) \cdot \nabla w_{0}~dx-\int_{\Omega }u~f\left( x,w_{0}\right)
~dx
\end{eqnarray*}
for all $u\in \mathop{\rm Lip}_0(\Omega)$.
\end{definition}

To show that the integrals in (\ref{weak-eq}) converge absolutely, note that
if $w\in $ $L_{\mathop{\rm loc}}^{\infty }\left( \Omega \right) $, then by
Lemma \ref{admisall} and the diagonal condition (\ref{hellip}), we have that
for all $x\in \Omega ^{\prime }\Subset \Omega,$ 
\begin{equation*}
\xi ^{t}\mathcal{A}\left( x,w\left( x\right) \right) \xi \approx \xi ^{t} 
\mathcal{A}\left( x,0\right) \xi \approx \sum_{i=1}^{n}k^{i}\left(
x,0\right) \xi _{i}^{2},
\end{equation*}
with constants which depend on $\mathcal{A}$, $\Omega ^{\prime }$ and $%
M_{0}=\left\Vert w\right\Vert _{L^{\infty }\left( \Omega ^{\prime }\right) }$%
. Hence, since $w\in H_{\mathcal{X}}^{1,2}\left( \Omega \right) \bigcap L_{%
\mathop{\rm loc}}^{\infty }\left( \Omega \right) $, $u\in \mathop{\rm Lip}%
_0(\Omega)$, and $f\left(x,z\right) $ is continuous, it follows that 
\begin{equation*}
\nabla _{\!\!\sqrt{\mathcal{A}},w}w,~\nabla _{\!\!\sqrt{\mathcal{A}},w}u,~%
\vec{\gamma }\left( x,w\left( x\right) \right) \cdot \nabla w,~-f\left(
x,w\left( x\right) \right) \in L_{\mathop{\rm loc}}^{2}\left( \Omega \right)
.
\end{equation*}
Consequently, each integral in (\ref{weak-eq}) converges absolutely, and the
same is true for the integrals in Definition \ref{Xweak}.

Alternately, we can make sense of weak solutions $w\in H_{\mathcal{X}
}^{1,2}\left( \Omega \right) $ which are not necessarily locally bounded by
assuming that the coefficient matrix is bounded in $z$ locally in $x$, that $%
\vec{\gamma}(x,z)$ is of subunit type globally in $z$ locally in $x$, and
that $\sup_z |f(x,z)|$ is locally integrable.

\subsection{{\label{maxpsection}A Maximum Principle}}

We will now prove the following result.

\begin{theorem}
{\ \label{maxp}Let }$\mathcal{A}$ satisfy (\ref{hellip}), $\vec{\gamma}${\ }
be of subunit type with respect to $\mathcal{A}${\ in $\Omega \times \mathbb{%
R}$, } and $f\,$ be a continuous function on $\Omega \times \mathbb{R}^{n}$
which satisfies $f(x,0)=0$ for $x\in \Omega $ and 
\begin{eqnarray}
f\left( x,z\right) \mathop{\rm sign}z &\leq &0\qquad \text{in }\Omega \times 
\mathbb{R},  \label{eff} \\
f\left( x,z_{1}\right) -f\left( x,z_{2}\right) &\leq &0\qquad \text{if }x\in
\Omega \text{ and }z_{1}\geq z_{2}.  \label{fnc}
\end{eqnarray}%
{If $w$ is a smooth function in }$\Omega $ which is continuous on $\overline{%
\Omega }$ and satisfies 
\begin{equation}
\mathop{\rm div}\mathcal{A}\left( x,w\right) \nabla w+\vec{\gamma}\left(
x,w\right) \cdot \nabla w+f\left( x,w\right) \geq 0\qquad \text{in }\Omega
\label{subso}
\end{equation}%
in the weak sense, i.e., satisfies 
\begin{equation}
\int \nabla {\varphi }\cdot \mathcal{A}\left( x,w\right) \nabla w\leq \int {%
\varphi }\,\vec{\gamma}\left( x,w\right) \cdot \nabla w+\int {\varphi }%
\,f\left( x,w\right)  \label{subsso}
\end{equation}%
for all nonnegative ${\varphi \in }\mathop{\rm Lip_{0}}\left( \Omega \right) 
$, then 
\begin{equation*}
\sup_{\Omega }w\leq \sup_{\partial \Omega }w^{+}.
\end{equation*}%
On the other hand, if the opposite inequality holds in (\ref{subso}), i.e.,
if 
\begin{equation}
\mathop{\rm div}\mathcal{A}\left( x,w\right) \nabla w+\vec{\gamma}\left(
x,w\right) \cdot \nabla w+f\left( x,w\right) \leq 0\qquad \text{in }\Omega
\label{superso}
\end{equation}%
in the weak sense, then 
\begin{equation}
\inf_{\Omega }w\geq \inf_{\partial \Omega }-(w^{-})\quad \left(
\,=-\sup_{\partial \Omega }w^{-}\right) ,  \label{lowerbound}
\end{equation}%
where $w^{-}:=-w$ if $w\leq 0$ and $w^{-}:=0$ if $w>0$. In particular, if $w$
is a weak solution of $\mathop{\rm
div}\mathcal{A}\left( x,w\right) \nabla w+\vec{\gamma}\left( x,w\right)
\cdot \nabla w+f\left( x,w\right) =0$ in $\Omega $, then $\sup_{\Omega
}|w|\leq \sup_{\partial \Omega }|w|$.
\end{theorem}

{\ 
\proof%
Assume first that $w$ satisfies (\ref{subso}), and recall that $w$ is smooth
by assumption. Let $\omega _{0}=\sup_{\partial \Omega }w^{+}$ and 
\begin{equation*}
v_{\tau }\left( x\right) =\left( w(x)-\omega _{0}-\tau \right) ^{+},\qquad
\tau >0,\,x\in \Omega .
\end{equation*}%
Then $v_{\tau }$ is nonnegative and Lipschitz continuous in $\Omega $, and $%
v_{\tau }$ has compact support in $\Omega $ since $w$ is continuous on $%
\overline{\Omega }$ by hypothesis and $\tau >0$. Also, for any $x$, $v_{\tau
}(x)>0$ if and only if $w(x)>\omega _{0}+\tau $. Let $\Phi _{\tau }=\{x\in
\Omega :v_{\tau }(x)>0\}.$ Then $v_{\tau }=\chi _{\Phi _{\tau }}(w-\omega
_{0}-\tau )$ on $\Omega $ and $\nabla v_{\tau }=\chi _{\Phi _{\tau }}\nabla
w $ a.e. on $\Omega $. By choosing ${\varphi =}v_{\tau }$ in (\ref{subsso}),
we obtain 
\begin{equation*}
\int \nabla v_{\tau }\cdot \mathcal{A}\left( x,w\right) \nabla w\leq \int
v_{\tau }\vec{\gamma}\left( x,w\right) \cdot \nabla w+\int v_{\tau }f\left(
x,w\right) .
\end{equation*}%
In this inequality, the right-hand side satisfies 
\begin{equation*}
\int_{\Phi _{\tau }}v_{\tau }\vec{\gamma}\left( x,w\right) \cdot \nabla
w+\int_{\Phi _{\tau }}v_{\tau }f\left( x,w\right) =\int_{\Phi _{\tau
}}v_{\tau }\vec{\gamma}\left( x,w\right) \cdot \nabla v_{\tau }+\int_{\Phi
_{\tau }}v_{\tau }f\left( x,w\right) ,
\end{equation*}%
while the left-hand side is $\int_{\Phi _{\tau }}\nabla v_{\tau }\cdot 
\mathcal{A}\left( x,w\right) \nabla v_{\tau }.$ Hence 
\begin{equation}
\int_{\Phi _{\tau }}\nabla v_{\tau }\cdot \mathcal{A}\left( x,w\right)
\nabla v_{\tau }\leq \int_{\Phi _{\tau }}v_{\tau }\vec{\gamma}\left(
x,w\right) \cdot \nabla v_{\tau }+\int_{\Phi _{\tau }}v_{\tau }f\left(
x,w\right) =I+II.  \label{7.12}
\end{equation}%
We have 
\begin{equation*}
II=\int_{\Phi _{\tau }}v_{\tau }f\left( x,v_{\tau }\right) +\int_{\Phi
_{\tau }}v_{\tau }\big[f\left( x,w\right) -f\left( x,v_{\tau }\right) \big]%
\leq 0\,+\,0=0
\end{equation*}%
by (\ref{eff}) and (\ref{fnc}) since $v_{\tau }>0$ and $w>v_{\tau }$ on $%
\Phi _{\tau }$ (note that $\omega \geq 0$). Recalling that $w$ is assumed to
be continuous on $\overline{\Omega }$ and so is bounded there, we may choose 
$M$ with $M>w$ on $\Omega $. Since $\vec{\gamma}$ is of subunit type with
respect to $\mathcal{A}$, Schwarz's inequality implies that 
\begin{equation*}
I\leq \frac{1}{4}\int_{\Phi _{\tau }}\nabla v_{\tau }\cdot \mathcal{A}\left(
x,w\right) \nabla v_{\tau }+4B_{\gamma }^{2}\int_{\Phi _{\tau }}v_{\tau
}^{2},
\end{equation*}%
where $B_{\gamma }=B_{\gamma }\left( \Omega _{\tau },M\right) $. From (\ref%
{7.12}) and the estimates for $I$ and $II$, we obtain 
\begin{equation}
\int_{\Phi _{\tau }}\nabla v_{\tau }\cdot \mathcal{A}\left( x,w\right)
\nabla v_{\tau }\leq C^{2}B_{\gamma }^{2}\int_{\Phi _{\tau }}v_{\tau }^{2}.
\label{7.13}
\end{equation}%
By the one-dimensional Sobolev estimate, 
\begin{equation*}
\int_{\Phi _{\tau }}v_{\tau }^{2}\leq C^{2}R^{2}\int_{\Phi _{\tau
}}\left\vert \partial _{1}v_{\tau }\right\vert ^{2},\quad R=\mathop{\rm diam}%
(\Omega ).
\end{equation*}%
Combining this with (\ref{7.13}) gives 
\begin{equation*}
\int_{\Phi _{\tau }}v_{\tau }^{2}\leq C^{2}R^{2}B_{\gamma }^{2}\int_{\Phi
_{\tau }}v_{\tau }^{2}.
\end{equation*}%
Thus, assuming that $R<(CB_{\gamma })^{-1}$, we obtain $\int_{\Phi _{\tau
}}v_{\tau }^{2}=0$. Then $\Phi _{\tau }$ is empty, i.e., $v_{\tau }=0$ on $%
\Omega $ and therefore $w\leq \omega _{0}+\tau $ on $\Omega $. }

To drop the restriction that $R<\big(CB_{\gamma }(\Omega _{\tau },M)\big)%
^{-1}$, let $N=CB_{\gamma }\mathop{\rm diam}(\Omega )+1$ and 
\begin{equation*}
\widetilde{\Omega }=\frac{\Omega }{N}=\left\{ \frac{x}{N}:x\in \Omega
\right\} .
\end{equation*}%
Also, for $x\in \widetilde{\Omega }$, let 
\begin{eqnarray*}
\widetilde{w}\left( x\right) &=&w\left( Nx\right) ,\qquad \widetilde{%
\mathcal{A}}\left( x,z\right) =N^{-2}\mathcal{A}\left( Nx,z\right) , \\
\widetilde{\vec{\gamma}}\left( x,z\right) &=&N^{-1}\vec{\gamma}\left(
Nx,z\right) ,\qquad \widetilde{f}\left( x,z\right) =f\left( Nx,z\right) .
\end{eqnarray*}%
Then if $x\in \widetilde{\Omega }$, 
\begin{eqnarray*}
&&\mathop{\rm div}\widetilde{\mathcal{A}}\left( x,\widetilde{w}\right)
\nabla \widetilde{w}+\widetilde{\vec{\gamma}}\left( x,\widetilde{w}\right)
\cdot \nabla \widetilde{w}+\widetilde{f}\left( x,\widetilde{w}\right) \\
&=&N^{-2}\mathop{\rm div}\big[\mathcal{A}\left( Nx,w\left( Nx\right) \right)
\nabla \big(w\left( Nx\right) \big)\big]+N^{-1}\vec{\gamma}\left( Nx,w\left(
Nx\right) \right) \cdot \nabla \big(w\left( Nx\right) \big) \\
&&+f\left( Nx,w\left( Nx\right) \right) \\
&=&\mathop{\rm div}\mathcal{A}\left( y,w\left( y\right) \right) \nabla
w\left( y\right) +\vec{\gamma}\left( y,w\left( y\right) \right) \cdot \nabla
w\left( y\right) +f\left( y,w\left( y\right) \right) \geq 0
\end{eqnarray*}%
by (\ref{subso}), where $y=Nx\in \Omega $. Thus $\widetilde{w}$ satisfies an
analogue of (\ref{subso}) in $\tilde{\Omega}$. Moreover, since $\vec{\gamma}$
is of subunit type with respect to $\mathcal{A}$ in $\Omega \times \mathbb{R}
$ (Definition \ref{subunitt}), it follows that $\widetilde{\vec{\gamma}}$ is
of subunit type with respect to $\widetilde{\mathcal{A}}$ in $\tilde{\Omega}%
\times \mathbb{R}$ with constant $\widetilde{B_{\gamma }}=B_{\gamma }$:
indeed, 
\begin{eqnarray*}
\left\vert \widetilde{\vec{\gamma}}\left( x,z\right) \cdot \xi \right\vert
^{2} &=&\left\vert N^{-1}\vec{\gamma}\left( Nx,z\right) \cdot \xi
\right\vert ^{2}\leq N^{-2}B_{\gamma }^{2}\xi ^{t}\mathcal{A}\left(
Nx,z\right) \xi \\
&=&B_{\gamma }^{2}\xi ^{t}\widetilde{\mathcal{A}}\left( x,z\right) \xi .
\end{eqnarray*}%
Also, 
\begin{eqnarray*}
\widetilde{f}\left( x,z\right) \mathop{\rm sign}z &=&f\left( Nx,z\right) %
\mathop{\rm sign}z\leq 0\qquad \text{in }\widetilde{\Omega }\times \mathbb{R}%
, \\
\widetilde{f}\left( x,z_{1}\right) -\widetilde{f}\left( x,z_{2}\right)
&=&f\left( Nx,z_{1}\right) -f\left( Nx,z_{2}\right) \leq 0\quad \text{ if }%
z_{1}\geq z_{2},\,x\in \widetilde{\Omega }
\end{eqnarray*}%
and 
\begin{equation*}
\mathop{\rm diam}\widetilde{\Omega }=\dfrac{\mathop{\rm diam}\Omega }{N}=%
\dfrac{\mathop{\rm diam}\Omega }{CB_{\gamma }\mathop{\rm diam}\Omega +1}<%
\dfrac{1}{CB_{\gamma }}=\dfrac{1}{C\widetilde{B_{\gamma }}}.
\end{equation*}%
Then $\widetilde{w}\leq \sup_{\partial \widetilde{\Omega }}(\widetilde{w}%
)^{+}+\tau $ in $\widetilde{\Omega }$ for all $\tau >0$, i.e., $w\leq
\sup_{\partial \Omega }w^{+}+\tau $ in $\Omega $ for all $\tau >0$. Letting $%
\tau \rightarrow 0$, we obtain $w\leq \sup_{\partial \Omega }w^{+}$ as
desired.

To prove (\ref{lowerbound}), let (\ref{superso}) hold and define $\widetilde{%
\mathcal{A}}(x,z)=\mathcal{A}(x,-z)$, $\widetilde{\vec{\gamma}}(x,z)=\vec{%
\gamma}(x,-z)$ and $\widetilde{f}(x,z)=-f(x,z)$ for $x\in \Omega $. Since $%
\mathcal{A}$ satisfies (\ref{hellip}) and $\vec{\gamma}$ is subunit with
respect to $\mathcal{A}$ in $\Omega \times \mathbb{R}$, it follows that $%
\widetilde{\mathcal{A}}$ satisfies (\ref{hellip}) and $\widetilde{\vec{\gamma%
}}$ is subunit with respect to $\widetilde{\mathcal{A}}$ in $\Omega \times 
\mathbb{R}$. From (\ref{eff}) and (\ref{fnc}) for $f$, we obtain (\ref{eff})
and (\ref{fnc}) for $\widetilde{f}$: 
\begin{eqnarray*}
\widetilde{f}\left( x,z\right) \mathop{\rm sign}z &=&-f\left( x,-z\right) %
\mathop{\rm sign}z \\
&=&f(x,-z)\mathop{\rm sign}(-z)\leq 0\qquad \text{in }\Omega \times \mathbb{R%
}, \\
\widetilde{f}\left( x,z_{1}\right) -\widetilde{f}\left( x,z_{2}\right)
&=&f\left( x,-z_{2}\right) -f\left( x,z_{1}\right) \leq 0\quad \text{ if }%
z_{1}\geq z_{2},\,x\in \Omega .
\end{eqnarray*}%
Now let $\widetilde{w}(x)=-w(x)$ and note that 
\begin{equation*}
\mathop{\rm div}\widetilde{\mathcal{A}}\left( x,\widetilde{w}(x)\right)
\nabla \widetilde{w}(x)+\widetilde{\vec{\gamma}}\left( x,\widetilde{w}%
(x)\right) \cdot \nabla \widetilde{w}(x)+\widetilde{f}\left( x,\widetilde{w}%
(x)\right) \,=
\end{equation*}%
\begin{equation*}
-\left[ \mathop{\rm div}\mathcal{A}\left( x,w(x)\right) \nabla w(x)+\vec{%
\gamma}\left( x,w(x)\right) \cdot \nabla w(x)+f\left( x,w(x)\right) \right]
\,\geq 0\quad \text{in $\Omega $.}
\end{equation*}%
By the previous case, $\sup_{\Omega }\widetilde{w}\leq \sup_{\partial \Omega
}\widetilde{w}^{+}$. Equivalently, 
\begin{equation*}
\sup_{x\in \Omega }(-w(x))\leq \sup_{x\in \partial \Omega
}(-w(x))^{+},\qquad \text{or}\qquad \inf_{x\in \Omega }w(x)\geq \inf_{x\in
\partial \Omega }-(w(x)^{-}),
\end{equation*}
which completes the proof of Theorem \ref{maxp}.{%
\endproof%
}

\subsection{{A Comparison Principle\label{compsect}}}

\begin{lemma}
{\ \label{comparison}Suppose that }$\mathcal{A}\left( x,z\right) $ satisfies
(\ref{hellip}) and (\ref{xtra}), and that $f$ is nonincreasing in $z$, i.e., 
\begin{equation}
f_{z}\left( x,z\right) \leq 0,\qquad \left( x,z\right) \in \Gamma .
\label{finc}
\end{equation}%
{Let $w_{0}$, $w_{1}\in $}$\left( H_{\mathcal{X}}^{1,2}\left( \Omega \right)
\bigcup \mathcal{C}^{\infty }\left( \Omega \right) \right) \bigcap \,${$%
\mathcal{C}^{0}\left( \overline{\Omega }\right) $\ satisfy $w_{0}+\kappa
\geq w_{1}$ on $\partial \Omega $ for some constant }$\kappa \geq 0${\ and 
\begin{equation}
\mathcal{P}\left( w_{1}\right) \geq \mathcal{P}\left( w_{0}\right) \,\text{
in }\Omega  \label{O}
\end{equation}%
(in the sense of Definition \ref{Xweak}), where 
\begin{equation*}
\mathcal{P}\left( w\right) =\mathop{\rm div}\mathcal{A}\left( x,w\right)
\nabla w+f\left( x,w\right) .
\end{equation*}%
Then $w_{0}+\kappa \geq w_{1}$ in $\Omega $. In particular, if $\mathcal{P}%
w_{0}=\mathcal{P}w_{1}\,$ in $\Omega $ and $w_{0}=w_{1}\,$on $\partial
\Omega $, then $w_{0}=w_{1}$ in $\overline{\Omega }$. }
\end{lemma}

{\ 
\proof%
First we will assume that $w_{0}$, $w_{1}\in $}$H_{\mathcal{X}}^{1,2}\left(
\Omega \right) \bigcap \,$$\mathcal{C}^{0}\left( \overline{\Omega }\right) $%
. {Given $\tau >0,$ {let }$u_{\tau }=w_{1}-w_{0}-\kappa -\tau $ and {\ }$%
u_{\tau }^{+}=\max \left\{ u_{\tau },0\right\} $. Then $u_{\tau }^{+}$ is a
nonnegative continuous function compactly supported in $\Omega $. Denote $K=%
\mathop{\rm supp}\left( u_{\tau }^{+}\right) $ and $\delta _{0}=%
\mathop{\rm
dist}\left( K,\partial \Omega \right) $. Let $\psi _{\delta }$ be a smooth
approximation of the identity with $\delta >0$, i.e., $\psi _{\delta }\left(
x\right) =\delta ^{-n}\psi \left( x/\delta \right) $ where $\psi \in 
\mathcal{C}_{0}^{\infty }\left( B_{1}\right) $ and $\int \psi dx=1$; here $%
B_{1}$ denotes the unit ball in $\mathbb{R}^{n}$. For $0<\varepsilon <1$ and 
$0<\delta <\delta _{0}/2$, set 
\begin{equation*}
{\varphi }_{\tau ,\varepsilon ,\delta }=\frac{u_{\tau }^{+}}{u_{\tau
}^{+}+\varepsilon }\ast \psi _{\delta }.
\end{equation*}%
Then $\varphi _{\tau ,\varepsilon ,\delta }$ is nonnegative, smooth and
compactly supported in $\Omega $. {From (\ref{O}), 
\begin{equation*}
\int_{\Omega }\left[ \mathcal{A}\left( x,w_{1}\right) \nabla w_{1}-\mathcal{A%
}\left( x,w_{0}\right) \nabla w_{0}\right] \,\nabla {\varphi }_{\tau
,\varepsilon ,\delta }-\int_{\Omega }\left[ f\left( x,w_{1}\right) -f\left(
x,w_{0}\right) \right] \,{\varphi }_{\tau ,\varepsilon ,\delta }\leq 0.
\end{equation*}%
}Since $w_{0},w_{1},\varphi _{\tau ,\varepsilon ,\delta }\in H_{\mathcal{X}%
}^{1,2}\left( \Omega \right) $ and $\varphi _{\tau ,\varepsilon ,\delta }$
is continuous and has compact support, all the integrals above are
absolutely convergent. By Proposition 1.2.2 in \cite{FSS}, ${\varphi }_{\tau
,\varepsilon ,\delta }\rightarrow u_{\tau }^{+}/$}$\left( {u_{\tau
}^{+}+\varepsilon }\right) ${\ in $H_{\mathcal{X}}^{1,2}\left( \Omega
^{\prime \prime }\right) $ as $\delta \rightarrow 0^{+}$ for any open $%
\Omega ^{\prime \prime }\Subset \Omega ^{\prime }$. Letting $\delta
\rightarrow 0^{+}$, we obtain{{\ 
\begin{eqnarray}
&&\varepsilon \int_{\Omega }\left[ \mathcal{A}\left( x,w_{1}\right) \nabla
w_{1}-\mathcal{A}\left( x,w_{0}\right) \nabla w_{0}\right] \,\frac{\nabla
u_{\tau }^{+}}{\left( u_{\tau }^{+}+\varepsilon \right) ^{2}}
\label{subsol2} \\
&&-\int_{\Omega }\left[ f\left( x,w_{1}\right) -f\left( x,w_{0}\right) %
\right] \,\frac{u_{\tau }^{+}}{u_{\tau }^{+}+\varepsilon }\leq 0,  \notag
\end{eqnarray}%
where we used that }}$\nabla \left( u_{\tau }^{+}/(u_{\tau }^{+}+\varepsilon
)\right) =\varepsilon \nabla u_{\tau }^{+}/$}$\left( u_{\tau
}^{+}+\varepsilon \right) ^{2}${. }

{Set {$u=w_{1}-w_{0}-\kappa $ and {$w_{t}=tw_{1}+\left( 1-t\right) w_{0}$, $%
0\leq t\leq 1$.} Then $\partial _{t}w_{t}=$}$u+\kappa $, and} {by (\ref{finc}%
), 
\begin{eqnarray}
\left[ f\left( x,w_{1}\right) -f\left( x,w_{0}\right) \right] \frac{u_{\tau
}^{+}}{u_{\tau }^{+}+\varepsilon } &=&\frac{u_{\tau }^{+}}{u_{\tau
}^{+}+\varepsilon }\int_{0}^{1}\partial _{t}\left[ f\left( x,w_{t}\right) %
\right] \,dt  \notag \\
&=&\frac{u_{\tau }^{+}\left( u+\kappa \right) }{u_{\tau }^{+}+\varepsilon }%
\int_{0}^{1}f_{z}\left( x,w_{t}\right) \,dt\leq 0,  \label{subsll022}
\end{eqnarray}%
where we used that }$u+\kappa \geq 0$ on the support of $u_{\tau }^{+}$;
recall that $\kappa \ge 0$ by hypothesis. Also,{{\ 
\begin{eqnarray}
&&\mathcal{A}\left( x,w_{1}\right) \nabla w_{1}-\mathcal{A}\left(
x,w_{0}\right) \nabla w_{0}=\int_{0}^{1}\partial _{t}\left\{ \mathcal{A}%
\left( x,w_{t}\right) \nabla w_{t}\right\} \,dt  \notag \\
&=&\left( u+\kappa \right) \,\left\{ \int_{0}^{1}\mathcal{A}_{z}\left(
x,w_{t}\right) \nabla w_{t}\,\,dt\right\} +\left\{ \int_{0}^{1}\mathcal{A}%
\left( x,w_{t}\right) \,dt\right\} \nabla u  \notag \\
&=&\left( u+\kappa \right) \,\vec{a}\left( x\right) +\mathbf{\mathbf{\tilde{A%
}}}\left( x\right) \nabla u,\qquad \text{where}  \label{subsll020}
\end{eqnarray}%
\begin{equation*}
\vec{a}\left( x\right) =\int_{0}^{1}\mathcal{A}_{z}\left( x,w_{t}\right)
\nabla w_{t}\,\,dt\qquad \text{and}\qquad \mathbf{\mathbf{\tilde{A}}}\left(
x\right) =\int_{0}^{1}\mathcal{A}\left( x,w_{t}\right) \,dt.
\end{equation*}%
Using (\ref{subsll020}) in (\ref{subsol2}),{\ using (\ref{subsll022}) to
omit the term in (\ref{subsol2}) which involves the difference of $f$
values, }and using the facts that $u=u_{\tau }^{+}+\tau $ and $\nabla
u=\nabla u_{\tau }^{+}$ on the support of $u_{\tau }^{+}$ yields}} 
\begin{equation}
\int_{\Omega }\mathbf{\mathbf{\tilde{A}}}\left( x\right) \nabla u_{\tau
}^{+}\cdot \frac{\nabla u_{\tau }^{+}}{\left( u_{\tau }^{+}+\varepsilon
\right) ^{2}}\leq -\int_{\Omega }\left[ \left( {{u_{\tau }^{+}+\tau }}%
\right) \,\vec{a}\left( x\right) \right] \,\frac{\nabla u_{\tau }^{+}}{%
\left( u_{\tau }^{+}+\varepsilon \right) ^{2}}.  \label{subsol05}
\end{equation}%
{{{\ } Now, by Schwarz's inequality and the definition of }}$\,\vec{a}\left(
x\right) ,$ 
\begin{eqnarray*}
&&\left\vert \int_{\Omega }\left[ \left( {{u_{\tau }^{+}+\tau }}\right) \,%
\vec{a}\left( x\right) \right] \,\frac{\nabla u_{\tau }^{+}}{\left( u_{\tau
}^{+}+\varepsilon \right) ^{2}}\right\vert \\
&=&\left\vert \int_{\Omega }\left[ \left( {{u_{\tau }^{+}+\tau }}\right)
\,\left\{ \int_{0}^{1}\nabla u_{\tau }^{+}\cdot \mathcal{A}_{z}\left(
x,w_{t}\right) \nabla w_{t}\,\,dt\right\} \right] \,\frac{1}{\left( u_{\tau
}^{+}+\varepsilon \right) ^{2}}\right\vert \\
&\leq &\alpha \int_{\Omega }\left\{ \int_{0}^{1}\dfrac{\left\vert \mathcal{A}%
_{z}\left( x,w_{t}\right) \nabla u_{\tau }^{+}\right\vert ^{2}}{k^{\ast
}\left( x,w_{t}\right) }\,dt\right\} \,\frac{1}{\left( u_{\tau
}^{+}+\varepsilon \right) ^{2}} \\
&&+\dfrac{C}{\alpha }\int_{\Omega }\,\left\{ \int_{0}^{1}k^{\ast }\left(
x,w_{t}\right) \left\vert \nabla w_{t}\right\vert ^{2}\,dt\right\} \,\frac{%
\left( {{u_{\tau }^{+}+\tau }}\right) ^{2}}{\left( u_{\tau }^{+}+\varepsilon
\right) ^{2}}.
\end{eqnarray*}%
{{In fact, in the last four integrations as well as those below, the domain
of integration can be restricted to the compact subset supp $u_{\tau }^{+}$
of $\Omega $. Then, since $k^{\ast }\left( x,w_{t}\right) \left\vert \nabla
w_{t}\right\vert ^{2}\leq \left\vert \nabla _{\!\!\sqrt{\mathcal{A}}%
,w_{t}}w_{t}\right\vert ^{2}$ due to (\ref{hellip}), by assuming that }}$%
\tau \leq \varepsilon ${\ and{\ applying{\ (\ref{xtra})} to the first term
on the right, we obtain }} 
\begin{eqnarray*}
&&\left\vert \int_{\Omega }\left[ \left( {{u_{\tau }^{+}+\tau }}\right) \,%
\vec{a}\left( x\right) \right] \,\frac{\nabla u_{\tau }^{+}}{\left( u_{\tau
}^{+}+\varepsilon \right) ^{2}}\right\vert \\
&\leq &\alpha B_{\mathcal{A}}^{2}\int_{\Omega }\dfrac{\left\{
\int_{0}^{1}\nabla u_{\tau }^{+}\cdot \mathcal{A}\left( x,w_{t}\right)
\nabla u_{\tau }^{+}\,\,dt\right\} }{\left( u_{\tau }^{+}+\varepsilon
\right) ^{2}}+\dfrac{C}{\alpha }\int_{\Omega }\left\{ \int_{0}^{1}\left\vert
\nabla _{\!\!\sqrt{\mathcal{A}},w_{t}}w_{t}\right\vert ^{2}\,\,dt\right\} \\
&\leq &\alpha B_{\mathcal{A}}^{2}\int_{\Omega }\dfrac{\nabla u_{\tau
}^{+}\cdot \mathbf{\mathbf{\tilde{A}}}\left( x\right) \nabla u_{\tau }^{+}}{%
\left( u_{\tau }^{+}+\varepsilon \right) ^{2}}+\frac{C}{\alpha },
\end{eqnarray*}%
{{where we used that $w_{t}\in $}}$H_{\mathcal{X}}^{1,2}\left( \Omega
\right) ${\ for $0\leq t\leq 1$, and hence the constant }$C$ is independent
of $\tau $ and $\varepsilon ${{. Taking $\alpha =1/2B_{\mathcal{A}^{2}}$,
combining with (\ref{subsol05}) and absorbing into the left gives 
\begin{equation}
\int_{\Omega }\mathbf{\mathbf{\tilde{A}}}\left( x\right) \nabla u_{\tau
}^{+}\cdot \frac{\nabla u_{\tau }^{+}}{\left( u_{\tau }^{+}+\varepsilon
\right) ^{2}}\leq CB_{\mathcal{A}}^{2},\qquad 0<\tau \leq \varepsilon .
\label{subsll23}
\end{equation}%
}}

From (\ref{hellip}) and the fact that $k^{1}(x,z)=1$, we have 
\begin{equation*}
\mathbf{\mathbf{\tilde{A}}}\left( x\right) \nabla u_{\tau }^{+}\cdot \nabla
u_{\tau }^{+}\geq \left( \partial _{1}u_{\tau }^{+}\right) ^{2}.
\end{equation*}%
Then, by (\ref{subsll23}) and the identity 
\begin{equation*}
\frac{\partial _{1}\left( u_{\tau }^{+}\right) }{u_{\tau }^{+}+\varepsilon }=%
\frac{\partial _{1}\left( \frac{u_{\tau }^{+}}{\varepsilon }+1\right) }{%
\frac{u_{\tau }^{+}}{\varepsilon }+1}=\partial _{1}\ln \left( \frac{u_{\tau
}^{+}}{\varepsilon }+1\right) ,
\end{equation*}%
it follows that 
\begin{equation*}
\int_{\Omega }\left\vert \partial _{1}\ln \left( \frac{u_{\tau }^{+}}{%
\varepsilon }+1\right) \right\vert ^{2}dx\leq CB_{\mathcal{A}}^{2},\qquad
0<\tau \leq \varepsilon .
\end{equation*}%
Applying the one-dimensional Sobolev inequality and letting $\tau
\rightarrow 0$ gives 
\begin{equation*}
\dfrac{1}{C\mathop{\rm diam}\Omega }\int_{\Omega }\left\vert \ln \left( 
\frac{u^{+}}{\varepsilon }+1\right) \right\vert ^{2}dx\leq \int_{\Omega
}\left\vert \partial _{1}\ln \left( \frac{u^{+}}{\varepsilon }+1\right)
\right\vert ^{2}dx\leq CB_{\mathcal{A}}^{2}
\end{equation*}%
{{uniformly in $\varepsilon >0$. Since $u$ is continuous in $\overline{%
\Omega }$, it follows that $u^{+}=0$ in $\Omega $, so that $u\leq 0$ in $%
\Omega $. Hence $w_{1}\leq w_{0}+\kappa $ in $\Omega $ as desired.}}

Now, for the general case, assume that {$w_{0}$, $w_{1}\in $}$\left( H_{%
\mathcal{X}}^{1,2}\left( \Omega \right) \bigcup \mathcal{C}^{\infty }\left(
\Omega \right) \right) \bigcap \,${$\mathcal{C}^{0}\left( \overline{\Omega }%
\right) $ and satisfy the hypotheses of the lemma. Consider a family} $%
\left\{ \Omega _{\varepsilon }\right\} _{\varepsilon >0}$ of open sets such
that $\Omega _{\varepsilon }\nearrow \Omega $ and%
\begin{equation*}
0<\inf \left\{ \mathop{\rm dist}\left( x,\partial \Omega \right) :x\in
\partial \Omega _{\varepsilon }\right\} <\varepsilon .
\end{equation*}%
Then $\Omega _{\varepsilon }\Subset \Omega $ for all $\varepsilon >0$, and
the function $\mu \left( \varepsilon \right) $ defined for $\varepsilon >0$
by $\mu \left( \varepsilon \right) =\max_{\partial \Omega _{\varepsilon
}}\left( w_{1}-w_{0}-\kappa \right) $ satisfies $\lim_{\epsilon \rightarrow
0}\mu \left( \epsilon \right) \leq 0$.

Since {$w_{0}$, $w_{1}\in $}$\left( H_{\mathcal{X}}^{1,2}\left( \Omega
\right) \bigcup \mathcal{C}^{\infty }\left( \Omega \right) \right) $ and $%
\Omega _{\varepsilon }\Subset \Omega $, it follows that {$w_{0}$, $w_{1}\in $%
}$H_{\mathcal{X}}^{1,2}\left( \Omega _{\varepsilon }\right) $ for each $%
\varepsilon >0$. Moreover, {$w_{0}+\kappa +\mu \left( \varepsilon \right)
\geq w_{1}$ on $\partial \Omega _{\varepsilon }$. By the previous case, 
\begin{equation*}
{w_{0}+\kappa +\mu \left( \varepsilon \right) \geq w_{1}\qquad }\text{in }%
\Omega _{\varepsilon }.
\end{equation*}%
The lemma now follows by letting $\varepsilon \rightarrow 0^{+}$.%
\endproof%
}

\subsection{Barriers for the Dirichlet problem\label{barriersect}}

In this section, we construct barrier functions for continuous weak
solutions of the Dirichlet problem in a smooth, strictly convex domain. An
interesting aspect of these barriers is that even though they are
specialized to a particular solution $w$, they depend only on the modulus
continuity of $w$.

\begin{lemma}
\label{barrier} Let $\Phi \Subset \Omega $ be a strongly convex domain. Let $%
\bar{r}=\mathop{\rm diam}\Phi $ and $\omega $ be a concave, strictly
increasing function with $\omega \in \mathcal{C}^{0}\left( \left[ 0,\bar{r}%
\right] \right) \bigcap \mathcal{C}^{2}\left( \left( 0,\bar{r}\right]
\right) $ and $\omega \left( 0\right) =0$. Suppose $\mathcal{A}$ satisfies (%
\ref{xtra}) and $\gamma $ is of subunit type with respect to $\mathcal{A}$.
For $m\in \mathbb{R}$, define a differential operator $\mathcal{L}_{m}$ by 
\begin{equation*}
\mathcal{L}_{m}h=\mathop{\rm div}\mathcal{A}\left( x,h\left( x\right)
+m\right) \nabla h.
\end{equation*}%
Then for every $x_{0}\in \partial \Phi $ and all positive constants $\eta $, 
$\nu $, $m_{0},K$, there exists a neighborhood $\mathcal{N}$of $x_{0}$ and a
function $h\in \mathcal{C}^{0}\left( \overline{\mathcal{N}}\right) \bigcap 
\mathcal{C}^{\infty }\left( \mathcal{N}\right) $ such that $\mathcal{N}%
\subset \left\{ \left\vert x-x_{0}\right\vert <\eta \right\} $ and 
\begin{equation*}
\left\{ 
\begin{array}{ll}
h\left( x_{0}\right) =0, &  \\ 
h\left( x\right) \leq -\omega \left( \left\vert x-x_{0}\right\vert \right) ,
& x\in \Phi \bigcap \mathcal{N}, \\ 
h\left( x\right) \leq -\nu , & x\in \Phi \bigcap \partial \mathcal{N}, \\ 
\mathcal{L}_{m}h\geq K, & x\in \Phi \bigcap \mathcal{N},\quad \left\vert
m\right\vert \leq m_{0}, \\ 
\bigtriangleup h=\sum_{i=1}^{n}\partial _{i}^{2}h>0, & x\in \Phi \bigcap 
\mathcal{N}.%
\end{array}%
\right.
\end{equation*}
\end{lemma}

\noindent {\bf Proof }%
Let $x_{0}\in \partial \Phi $. By translation, we may assume that $x_{0}=0.$
By using a rotation $\Theta =\left( \theta _{ij}\right) _{i,j=1}^{n}$, we
may represent $\partial \Phi $ locally as $y=\Theta x$ so that for $%
y^{\prime }=\left( y_{1},\dots ,y_{n-1}\right) $ and positive $\kappa _{0}$, 
$r_{0}$ depending on $\partial \Phi $, we have 
\begin{equation}
\kappa _{0}\left\vert y^{\prime }\right\vert ^{2}\leq y_{n},\qquad \left(
y^{\prime },y_{n}\right) \in \partial \Phi ,\qquad \left\vert y^{\prime
}\right\vert <r_{0}.  \label{bdry}
\end{equation}%
By hypothesis, $\omega \,$satisfies 
\begin{eqnarray}
\omega \left( r\right) &\geq &a_{0}\,r\qquad \text{if $r\in \left[ 0,\bar{r}%
\right] $, where $a_{0}=\frac{\omega \left( \bar{r}\right) }{\bar{r}}$},
\label{tang} \\
\liminf_{r\rightarrow 0^{+}}\omega ^{\prime }\left( r\right) &\geq &a_{0}>0,
\label{concav} \\
\omega ^{\prime \prime }\left( r\right) &\leq &0\qquad \text{if $r\in \left(
0,\bar{r}\right] .$}  \label{conse}
\end{eqnarray}%
For fixed $\alpha _{0}\in (0,1]$, set 
\begin{equation}
r_{0}=\left\{ 
\begin{array}{ll}
\omega ^{-1}\left( \alpha _{0}\right) & \qquad \text{if }\omega \left( \bar{r%
}\right) >\alpha _{0} \\ 
\bar{r} & \qquad \text{otherwise,}%
\end{array}%
\right.  \label{barrr0}
\end{equation}%
i.e., $r_{0}$ is the largest $r\in (0,\bar{r}]$ with $\omega \left( r\right)
\leq \alpha _{0}$. Letting $\psi \left( r\right) =\sqrt{\omega \left(
r\right) }$ for $0<r\leq r_{0}$, we have by (\ref{tang})--(\ref{conse}) and
since $\lim_{r\rightarrow 0^{+}}\omega \left( r\right) =0$ that 
\begin{eqnarray}
1 &\geq &\sqrt{\alpha _{0}}\geq \psi \left( r\right) \geq \omega \left(
r\right) ,  \label{psione} \\
-\psi ^{\prime \prime } &=&-\dfrac{2\omega ^{\prime \prime }\omega -\left(
\omega ^{\prime }\right) ^{2}}{4\omega ^{3/2}}\geq \frac{\left( \omega
^{\prime }\right) ^{2}}{4\omega ^{3/2}}=\frac{\left( \psi ^{\prime }\right)
^{2}}{\psi }>0,  \label{psign} \\
\psi \left( r\right) &\geq &a_{1}\sqrt{r},\qquad \text{where $a_{1}=\sqrt{%
a_{0}}$,}  \label{psiline} \\
\lim_{r\rightarrow 0+}\psi ^{\prime }\left( r\right) &=&\lim_{r\rightarrow
0+}\frac{\omega ^{\prime }\left( r\right) }{2\sqrt{\omega \left( r\right) }}%
=+\infty \text{.}  \label{blow}
\end{eqnarray}%
For $t>0$, let $\mathcal{N}_{t}=\left\{ y_{n}<t\right\} \bigcap \left\{
|y|<r_{0}\right\} $. Because of (\ref{hellip}) (recall that $k^{1}=1$) and
continuity of $\mathcal{A}$, there exist $1\leq \ell \leq n$ and $c_{1}>0$
such that for all small $t$, 
\begin{equation}
\vec{\theta}_{\ell }\mathcal{A}\left( y,z\right) \left( \vec{\theta}_{\ell
}\right) ^{\prime }\geq c_{1}>0,\qquad y\in \partial \Phi \bigcap \mathcal{N}%
_{t},\quad \quad \left\vert z\right\vert \leq 2m_{0}.  \label{ne}
\end{equation}%
Here $\vec{\theta}_{\ell }^{\prime }$ is the $\ell ^{\text{th}}$ column of $%
\Theta $. For $m_{1}>0$ and $0<t_{1}\leq 1$ to be determined, define 
\begin{equation}
h\left( y\right) =-2\psi \left( \sqrt{\rho y_{n}}\right) +m_{1}\frac{y_{\ell
}^{2}}{2}+\frac{1}{\ln y_{n}},\qquad y\in \mathcal{N}_{t_{1}},  \label{barrh}
\end{equation}%
where $\rho =\left( \kappa _{0}^{-\frac{1}{2}}+1\right) ^{2}$. For $t_{1}$
small enough $h$ is well-defined in $\mathcal{N}_{t_{1}}$ and extends
continuously to $\overline{\mathcal{N}_{t_{1}}}$ with 
\begin{equation}
h\in \mathcal{C}^{\infty }\left( \mathcal{N}_{t_{1}}\right) \bigcap \mathcal{%
\ C}^{0}\left( \overline{\mathcal{N}_{t_{1}}}\right) ,\qquad h\left(
0\right) =0.  \label{barr99}
\end{equation}%
From (\ref{bdry}), 
\begin{eqnarray}
\sqrt{\rho y_{n}} &=&\left( \kappa _{0}^{-\frac{1}{2}}+1\right) \sqrt{y_{n}}=%
\sqrt{\frac{y_{n}}{\kappa _{0}}}+\sqrt{y_{n}}  \notag \\
&\geq &\left\vert y^{\prime }\right\vert +y_{n}\geq \left\vert y\right\vert ,
\label{rho-y}
\end{eqnarray}%
where we used that $y_{n}\leq 1$ in $\mathcal{N}_{t_{1}}$.

Also, by taking $t_{1}$ small enough depending on $\kappa _{0}$, $a_{1}$, $%
m_{1} $, we obtain from (\ref{bdry}) and (\ref{psiline}) that 
\begin{eqnarray*}
m_{1}\frac{y_{\ell }^{2}}{2}-\psi \left( \sqrt{\rho y_{n}}\right) &\leq
&m_{1}\frac{y_{n}}{2\kappa _{0}}-a_{1}\sqrt[4]{\rho y_{n}} \\
&\leq &m_{1}\frac{\sqrt[4]{y_{n}}}{2\kappa _{0}}\left( t_{1}^{3/4}-\frac{
2\kappa _{0}a_{1}\sqrt{\kappa _{0}^{-\frac{1}{2}}+1}}{m_{1}}\right) \leq 0.
\end{eqnarray*}
Then, using (\ref{rho-y}), the fact that $\psi $ is increasing, the
definition of $\psi$ and (\ref{psione}), we get 
\begin{eqnarray}
h\left( y\right) +\omega \left( \left\vert y\right\vert \right) &=&-2\psi
\left( \sqrt{\rho y_{n}}\right) +m_{1}\frac{y_{\ell }^{2}}{2}+\frac{1}{\ln
y_{n}}+\omega \left( \left\vert y\right\vert \right)  \notag \\
&\leq &\omega \left( \left\vert y\right\vert \right) -\psi \left( \sqrt{\rho
y_{n}}\right)  \notag \\
&\leq &\psi ^{2}\left( \left\vert y\right\vert \right) -\psi \left(
\left\vert y\right\vert \right) <0.  \label{bigger}
\end{eqnarray}

Now set 
\begin{equation*}
G_{n}\left( y\right) =\left( \frac{\rho }{y_{n}}\right) ^{\frac{1}{2}}\psi
^{\prime }\left( \sqrt{\rho y_{n}}\right) +\frac{1}{\left( \ln y_{n}\right)
^{2}y_{n}}>0,
\end{equation*}
and write $h_{j}=\frac{\partial h}{\partial y_{i}}$, $h_{ij}=\frac{ \partial
^{2}h}{\partial y_{i}\partial y_{j}}$. If $\delta _{ij}$ denotes the
Kronecker delta, we have 
\begin{eqnarray}
h_{n}\left( y\right) &=&-G_{n}+\delta _{\ell n}m_{1}y_n,\qquad h_{\ell
}\left( y\right) =m_{1}y_{\ell }-\delta _{\ell n}G_{n}  \label{derh} \\
h_{i}\left( y\right) &=&0\qquad \text{for }i\neq n,\ell  \label{derhz} \\
h_{\ell \ell }\left( y\right) &=& m_{1}+ \delta_{\ell n} \dfrac{1}{2y_{n}}%
\left[ G_{n}\left( y\right) -\rho \psi ^{\prime \prime }\left( \sqrt{\rho
y_{n}}\right) \right] \\
&& +\delta_{\ell n}\frac{1}{\left( \ln y_{n}\right) ^{2}y_{n}^{2}}\left( 
\dfrac{2}{\ln y_{n}}+\dfrac{1}{2}\right)  \label{h11} \\
h_{nn}\left( y\right) &=&\dfrac{1}{2y_{n}}\left[ G_{n}\left( y\right) -\rho
\psi ^{\prime \prime }\left( \sqrt{\rho y_{n}}\right) \right]  \label{hnn} \\
&&+\frac{1}{\left( \ln y_{n}\right) ^{2}y_{n}^{2}}\left( \dfrac{2}{\ln y_{n}}
+\dfrac{1}{2}\right) +\delta _{\ell n}m_{1}  \notag \\
h_{ii}\left( y\right) &=&0 \quad \text{for }i\neq n,\ell ,\qquad
h_{ij}\left( y\right) =0 \quad \text{for }i\neq j.  \label{derhnn}
\end{eqnarray}
In particular, for $t_1$ small enough, 
\begin{equation}
\begin{array}{lll}
\bigtriangleup h\left( y\right) & = & \dfrac{1}{2y_{n}}\left[ G_{n}\left(
y\right) -\rho \psi ^{\prime \prime }\left( \sqrt{\rho y_{n}}\right) \right]
+\dfrac{1}{\left( \ln y_{n}\right) ^{2}y_{n}^{2}}\left( \dfrac{2}{\ln y_{n}}%
+ \dfrac{1}{2}\right) +m_{1}>0%
\end{array}
\label{lapp}
\end{equation}
in $\mathcal{N}_{t_{1}}$ because of (\ref{psign}). Moreover, by (\ref{psign}%
) and the formulas for derivatives of $h$, 
\begin{eqnarray}
\left\vert h_{\ell }\left( y\right) \right\vert ^{2} &\leq
&2m_{1}^{2}y_{\ell }^{2}+2\delta _{\ell n}G_{n}^{2}  \notag \\
&\leq & c\max \left\{ m_{1}y_{\ell }^{2}, \, \delta _{\ell n}\psi \left( 
\sqrt{ \rho y_{n}}\right) ,\frac{\delta _{\ell n}}{\left( \ln y_{n}\right)
^{2}} \right\} h_{\ell \ell }\left( y\right),  \label{hl-s}
\end{eqnarray}
\begin{eqnarray}
\left\vert h_{n}\left( y\right) \right\vert ^{2} &\leq &2G_{n}^{2}+2\delta
_{\ell n}m_{1}^{2}y_n^{2}  \notag \\
&\leq & c\max \left\{ \psi \left( \sqrt{\rho y_{n}}\right) ,\frac{1}{\left(
\ln y_{n}\right) ^{2}}, \delta _{\ell n}m_{1}y_n^{2}\right\} h_{nn}\left(
y\right) .  \label{hn-s}
\end{eqnarray}

Let $H(x)= h(\Theta x)$ and $\tilde{H}(x) = H(x) +m = h(\Theta x) +m$ for
fixed $m\in [-m_0,m_0]$. Using $\Theta^t$ to denote the transpose of $\Theta$%
, we then have $\nabla H(x)= \Theta^t (\nabla h)(\Theta x)$, $H_j(x) =
\sum_k \theta_{jk}h_k(\Theta x)$ and 
\begin{equation*}
H_{ij}(x) = \sum_k \theta_{jk} \sum_\lambda \theta_{i\lambda}
h_{k\lambda}(\Theta x) = \sum_k \theta_{jk} \theta_{ik} h_{kk}(\Theta x)
\end{equation*}
\begin{equation*}
= \theta_{j\ell}\theta_{i\ell} h_{\ell \ell}(\Theta x) +
\theta_{jn}\theta_{in} h_{nn}(\Theta x),
\end{equation*}
where only one of the last two terms appears in case $\ell =n$.

Setting $\mathcal{A}(x,\tilde{H}(x))=(a_{ij}(x))$ and letting $\left[ %
\mathop{\rm div}\mathcal{A}(x,\tilde{H}(x))\right] =\left[ \mathop{\rm div}%
(a_{ij}(x))\right] $ denote the vector whose components are the divergence
of the columns of $\big(a_{ij}(x)\big)$, we obtain 
\begin{eqnarray*}
\mathcal{L}_{m}h &=&\mathop{\rm div}\left[ \mathcal{A}\left( x,\tilde{H}%
(x)\right) \nabla H(x)\right] =\mathop{\rm div}\left(
\sum_{i}a_{ij}(x)H_{i}(x)\right) _{j=1,\dots ,n} \\
&=&\left[ \mathop{\rm div}(a_{ij}(x))\right] \cdot \nabla
H(x)+\sum_{i,j}a_{ij}(x)H_{ij}(x)
\end{eqnarray*}%
\begin{equation*}
=\left[ \mathop{\rm div}(a_{ij}(x))\right] \cdot \Theta ^{t}(\nabla
h)(\Theta x)+h_{\ell \ell }(\Theta x)\sum_{i,j}a_{ij}(x)\theta _{j\ell
}\theta _{i\ell }+h_{nn}(\Theta x)\sum_{i,j}a_{ij}(x)\theta _{jn}\theta _{in}
\end{equation*}%
\begin{equation*}
=\left[ \mathop{\rm div}(a_{ij}(x))\right] \cdot \Theta ^{t}(\nabla
h)(\Theta x)+h_{\ell \ell }(\Theta x)\vec{\theta}_{\ell }\mathcal{A}(x,%
\tilde{H}(x))\vec{\theta}_{\ell }^{~t}+h_{nn}(\Theta x)\vec{\theta}_{n}%
\mathcal{A}(x,\tilde{H}(x))\vec{\theta}_{n}^{~t}.
\end{equation*}%
Here $\vec{\theta}_{j}^{t}$ denotes the $j^{\text{th}}$ column of $\Theta $
and $\vec{\theta}_{j}$ denotes the $j^{\text{th}}$ column of $\Theta $
rewritten as a row vector, and in the last two equalities, only one of the
second and third terms on the right appears in case $\ell =n$.

Now recall that $h_{i}=0$ if $i\neq \ell ,n$. Then the first term on the
right of the last equality equals 
\begin{equation*}
h_{\ell }(\Theta x)\sum_{k}\left( \sum_{i}\frac{\partial }{\partial x_{i}}%
a_{ik}(x)\right) \theta _{k\ell }+h_{n}(\Theta x)\sum_{k}\left( \sum_{i}%
\frac{\partial }{\partial x_{i}}a_{ik}(x)\right) \theta _{kn}
\end{equation*}%
\begin{equation*}
=h_{\ell }(\Theta x)\left[ \mathop{\rm div}\mathcal{A}(x,\tilde{H}(x))\right]
\cdot \vec{\theta}_{\ell }^{~t}+h_{n}(\Theta x)\left[ \mathop{\rm div}%
\mathcal{A}(x,\tilde{H}(x))\right] \cdot \vec{\theta}_{n}^{~t},
\end{equation*}%
where as usual only one of the two terms on the right appears in case $\ell
=n$.

Altogether, 
\begin{eqnarray*}
\mathcal{L}_{m}H(x) &=&h_{\ell }(\Theta x)\left[ \mathop{\rm div}\mathcal{A}%
\left( x,\tilde{H}(x)\right) \right] \cdot \vec{\theta}_{\ell
}^{~t}+h_{n}(\Theta x)\left[ \mathop{\rm div}\mathcal{A}\left( x,\tilde{H}%
(x)\right) \right] \cdot \vec{\theta}_{n}^{~t} \\
&&+h_{\ell \ell }(\Theta x)\vec{\theta}_{\ell }\mathcal{A}\left( x,\tilde{H}%
(x)\right) \vec{\theta}_{\ell }^{~t}+h_{nn}(\Theta x)\vec{\theta}_{n}%
\mathcal{A}\left( x,\tilde{H}(x)\right) \vec{\theta}_{n}^{~t},
\end{eqnarray*}%
where on the right side, only one of the first two terms and one of the
second two terms appears in case $\ell =n$.

By direct computation, we have 
\begin{eqnarray*}
\left[ \mathop{\rm div}\mathcal{A}(x,\tilde{H}(x))\right] \cdot \vec{\theta}%
_{r}^{~t} &=&\sum_{i=1}^{n}\mathcal{A}_{i}^{i}(x,\tilde{H}(x))\cdot \vec{%
\theta}_{r}^{~t}+h_{\ell }(\Theta x)\vec{\theta}_{\ell }\mathcal{A}_{z}(x,%
\tilde{H}(x))\vec{\theta}_{r}^{~t} \\
&+&h_{n}(\Theta x)\vec{\theta}_{n}\mathcal{A}_{z}(x,\tilde{H}(x))\vec{\theta}%
_{r}^{~t},
\end{eqnarray*}%
where $\mathcal{A}^{i}$ denotes the $i^{th}$-column of $\mathcal{A}$ and $%
\mathcal{A}_{i}^{i}=\partial _{i}\mathcal{A}^{i}$, and where only one of the
last two terms on the right side appears in case $\ell =n$. Substituting
this in the formula above for $\mathcal{L}_{m}H$ gives 
\begin{eqnarray*}
\mathcal{L}_{m}H(x) &=&\mathcal{L}_{m}[h(\Theta x)]=h_{\ell }(\Theta
x)\sum_{i=1}^{n}\mathcal{A}_{i}^{i}(x,\tilde{H}(x))\cdot \vec{\theta}_{\ell
}^{~t}+h_{\ell }(\Theta x)^{2}\vec{\theta}_{\ell }\mathcal{A}_{z}(x,\tilde{H}%
(x))\vec{\theta}_{\ell }^{~t} \\
&&+h_{\ell }(\Theta x)h_{n}(\Theta x)\vec{\theta}_{n}\mathcal{A}_{z}(x,%
\tilde{H}(x))\vec{\theta}_{\ell }^{~t}+h_{\ell }(\Theta x)h_{n}(\Theta x)%
\vec{\theta}_{\ell }\mathcal{A}_{z}(x,\tilde{H}(x))\vec{\theta}_{n}^{~t} \\
&&+h_{n}(\Theta x)\sum_{i=1}^{n}\mathcal{A}_{i}^{i}(x,\tilde{H}(x))\cdot 
\vec{\theta}_{n}^{~t}+h_{n}(\Theta x)^{2}\vec{\theta}_{n}\mathcal{A}_{z}(x,%
\tilde{H}(x))\vec{\theta}_{n}^{~t} \\
&&+h_{\ell \ell }(\Theta x)\vec{\theta}_{\ell }\mathcal{A}\left( x,\tilde{H}%
(x)\right) \vec{\theta}_{\ell }^{~t}+h_{nn}(\Theta x)\vec{\theta}_{n}%
\mathcal{A}\left( x,\tilde{H}(x)\right) \vec{\theta}_{n}^{~t},
\end{eqnarray*}%
and in case $\ell =n$, only the first, second and seventh terms on the right
appear. By (\ref{wirtm}), the sum of the first and fifth terms on the right
is at most 
\begin{equation*}
cB_{\mathcal{A}}^{2}+\dfrac{h_{\ell }(\Theta x)^{2}}{4}\vec{\theta}_{\ell }%
\mathcal{A}(x,\tilde{H}(x))\vec{\theta}_{\ell }^{~t}+\dfrac{h_{n}(\Theta
x)^{2}}{4}\vec{\theta}_{n}\mathcal{A}(x,\tilde{H}(x))\vec{\theta}_{n}^{~t},
\end{equation*}%
and the sum of the second, third, fourth and sixth terms is bounded by 
\begin{equation*}
cB_{\mathcal{A}}\left( h_{\ell }(\Theta x)^{2}\vec{\theta}_{\ell }\mathcal{A}%
(x,\tilde{H}(x))\vec{\theta}_{\ell }^{~t}+h_{\ell }(\Theta x)^{2}\vec{\theta}%
_{n}\mathcal{A}(x,\tilde{H}(x))\vec{\theta}_{n}^{~t}\right) ,
\end{equation*}%
where {$B_{\mathcal{A}}=B_{\mathcal{A}}\left( \Phi ,m_{0}\right) $. }Using
these estimates, we obtain 
\begin{eqnarray*}
\mathcal{L}_{m}H(x) &\geq &-cB_{\mathcal{A}}^{2}+\left[ h_{\ell \ell
}(\Theta x)-\left( cB_{\mathcal{A}}+\dfrac{1}{4}\right) h_{\ell }(\Theta
x)^{2}\right] ~\vec{\theta}_{\ell }\mathcal{A}\left( x,\tilde{H}(x)\right) 
\vec{\theta}_{\ell }^{~t} \\
&&+\left[ h_{nn}(\Theta x)-\left( cB_{\mathcal{A}}+\dfrac{1}{4}\right)
h_{n}(\Theta x)^{2}\right] ~\vec{\theta}_{n}\mathcal{A}\left( x,\tilde{H}%
(x)\right) \vec{\theta}_{n}^{~t}.
\end{eqnarray*}%
From (\ref{hl-s}) and (\ref{hn-s}), by taking $t_{1}$ small enough
(depending on $\kappa _{0}$, $m_{1}$, and $B_{\mathcal{A}}$), we obtain 
\begin{equation*}
\mathcal{L}_{m}H(x)\geq -cB_{\mathcal{A}}^{2}+\dfrac{h_{\ell \ell }(\Theta x)%
}{2}~\vec{\theta}_{\ell }\mathcal{A}\left( x,\tilde{H}(x)\right) \vec{\theta}%
_{\ell }^{~t}.
\end{equation*}%
Then from (\ref{ne}) and the fact that $h_{\ell \ell }\geq m_{1},$ 
\begin{equation}
\mathcal{L}_{m}h\geq -cB_{\mathcal{A}}^{2}+\dfrac{c_{1}m_{1}}{2}>K
\label{PK}
\end{equation}%
by taking $m_{1}$ large enough.

Finally, note that $y_{\ell }^{2}\leq y_{n}/\kappa _{0}$ by (\ref{bdry}).
Then if $x\in \Phi \bigcap \partial \mathcal{N} _{t_{1}} $, 
\begin{eqnarray}
h\left( y\right) &=&-2\psi \left( \sqrt{\rho y_{n}}\right) +m_{1}\frac{
y_{\ell }^{2}}{2}+\frac{1}{\ln y_{n}}  \notag \\
&\leq &\frac{m_{1}}{2\kappa _{0}}t_{1}^{2}+\frac{1}{\ln t_{1}}\leq -\nu
\label{Pnu}
\end{eqnarray}
by taking $t_{1}$ small enough. The condition $\mathcal{N}\subset \left\{
\left\vert x-x_{0}\right\vert <\eta \right\} $ can then be met by taking $%
t_{1}$ even smaller. Lemma \ref{barrier} now follows from (\ref{barr99}), (%
\ref{bigger}), (\ref{lapp}), (\ref{PK}) and (\ref{Pnu}).{\ 
\endproof%
}

\end{document}